\documentclass[11pt]{amsart}
\usepackage{amscd, amssymb, latexsym, amsmath, amscd, amsmath, mathtools}
\usepackage[all]{xy}
\usepackage{anysize}
\usepackage[sc]{mathpazo} 
\usepackage{color}
\marginsize{2 cm}{2 cm}{1.80cm}{1.80cm}
\usepackage{graphicx, amstext, hyperref, amsrefs, setspace}
\usepackage{enumerate}
\allowdisplaybreaks
\usepackage[stable]{footmisc} 
\usepackage[dvipsnames]{xcolor}
\usepackage{amsfonts, mathrsfs, yfonts, stmaryrd}
\usepackage{url, tikz, enumitem, caption, comment, marginnote}
\usepackage[utf8]{inputenc}
\usepackage[T1]{fontenc}
\usepackage{tikz-cd}
\usepackage{libertine}
\usetikzlibrary{matrix, positioning, arrows, shapes, chains, calc, automata}
\usetikzlibrary{decorations.pathmorphing}
\usetikzlibrary{arrows, automata, positioning}
\allowdisplaybreaks[4]

\hypersetup{linktocpage,bookmarksopen = true, bookmarksopenlevel = 1, colorlinks=true,
linkcolor = Cerulean, citecolor = RubineRed , urlcolor = gray, pdfstartview = FitR} 
\newtheorem{theorem}{Theorem}[section] 
\newtheorem{lemma}[theorem]{Lemma}
\newtheorem{corollary}[theorem]{Corollary}
\newtheorem{proposition}[theorem]{Proposition}
\newtheorem{algorithm}[theorem]{Algorithm}
\theoremstyle{definition}

\newtheorem{example}{Example}
\newtheorem{definition}{Definition}

\theoremstyle{remark}
\newtheorem{remark}{Remark}

\newtheorem*{ack}{Acknowledgements}

\begin{document}
\title[Branching laws and unitary GGP relevance]{Quotient branching laws and unitary Gan-Gross-Prasad relevance for general linear groups}
\author{Basudev Pattanayak}
\address{Department of Mathematics, Technion-Israel Institute of Technology, Haifa, Israel.}
\email{pbasudev93@gmail.com \& basudev@campus.technion.ac.il}
\date{}
\begin{abstract}
This article addresses the quotient branching problem for the pair $(\operatorname{GL}_{n+1}(F), \operatorname{GL}_n(F))$ over a non-archimedean local field $F$. We present two primary contributions. First, we explicate Chan's recent general solution by providing a practical, step-by-step combinatorial algorithm to determine whether the space $\operatorname{Hom}_{\operatorname{GL}_n(F)}(\pi, \pi')$ is non-zero for any irreducible smooth representations $\pi$ and $\pi'$. This algorithm is designed to be verifiable by hand, given the representations' Langlands or Zelevinsky parameters. Second, we specialize to the unitary case, where we establish a unifying equivalence: a pair of irreducible unitary representations satisfies Chan's generalized GGP relevance criterion if and only if it satisfies a natural extension of the original Gan–Gross–Prasad relevance condition. This result resolves a previous inconsistency in the literature, corrects and extends prior work by Gurevich, and provides a complete, explicit criterion for unitary branching laws.
\end{abstract}
\maketitle
\tableofcontents
\section{Introduction}\label{intro}
Let $F$ be a non-archimedean local field. A fundamental problem in the representation theory of reductive groups over such fields is the branching law: given an irreducible smooth representation $\pi$ of a larger group $G$ and a subgroup $H$, how does $\pi$ decompose when restricted to $H$? Let $\pi$ be a smooth irreducible representation of the group $\mathrm{GL_{n+1}(F)}$. We consider $\mathrm{GL_{n}(F)}$ as a subgroup of $\mathrm{GL_{n+1}(F)}$, where the natural embedding 
\[
\mathrm{GL_{n}(F)} \hookrightarrow  \mathrm{GL_{n+1}(F)} \text{ is defined by } g \mapsto \left( \begin{array}{cc} g &  \\  & 1 \end{array} \right).
\]

The study of quotient branching laws predicts whether an irreducible smooth representation $\pi'$ of $\mathrm{GL_{n}(F)}$ appears as a quotient of the restriction $\pi|_{\mathrm{GL_{n}(F)}}$. This involves understanding the space $\mathrm{Hom_{GL_{n}(F)}}(\pi,\pi')$, where $\pi$ is viewed as a $\mathrm{GL_{n}(F)}$-representation via restriction. A landmark result by Aizenbud, Gourevitch, Rallis, and Schiffmann \cite{AGRS} established a strong multiplicity-one theorem:
\[
\mathrm{dim_{\mathbb{C}}~~Hom_{GL_{n}(F)}}(\pi,\pi') \leq 1,
\]
for all irreducible  smooth representations $\pi $ of $\mathrm{GL_{n+1}(F)}$, and $\pi' $ of $ \mathrm{GL_n(F)}$. Then, we have the fundamental question: when is $\mathrm{Hom_{GL_{n}(F)}}(\pi,\pi')$ non-zero? This question has been studied widely \cites{BZ2, JPSS, Prasad_duke, GGP1, GGP2, CS, Gur, Cha_crelle, Cha_qbl}. Further, there are some recent interesting studies \cites{CS, Cha_mz, Saad_imrn, WZ, CQ} of Ext-branching laws for $\mathrm{GL_{n}(F)}$ following the work \cite{Pra_icm} of Prasad.

While the general solution to this problem has recently been provided by Chan \cite{Cha_qbl} through the notion of a "generalized GGP relevant pair," the resulting criterion is intricate and not easily verifiable in practice. This article has two primary objectives. First, we aim to demystify this criterion by presenting a step-by-step, combinatorial algorithm that allows one to determine, in a finite number of steps, whether $\mathrm{Hom_{GL_{n}(F)}}(\pi,\pi') \ne 0$ for any given pair of irreducible smooth representations. Second, we specialize to the important case of unitary representations, where we establish a direct and elegant equivalence between Chan's general criterion and a natural extension of the original Gan–Gross–Prasad (GGP) relevance condition \cite{GGP2}. This resolves a previous inconsistency in the literature and provides a complete, unified theory for unitary branching laws.

\subsection{Extended Gan–Gross–Prasad relevance for unitary representations} In the mid-1980s, Tadi\'c \cite{Tad} provided a complete classification of the unitary dual of general linear groups. An irreducible unitary representation can be expressed as a product of unitary Speh and complementary series representations. For the branching problem, a natural conjecture, extending the original work of Gan, Gross, and Prasad \cite{GGP2} for Arthur-type representations, would be a simple criterion in terms of the parameters defining these unitary building blocks. Gurevich \cite{Gur} made significant progress in this direction, giving necessary conditions for the Hom space to be non-zero.

However, the original necessary conditions in \cite[Theorem 5.7]{Gur} are insufficient, as demonstrated by the following counterexample (detailed in Appendix \ref{app:gur}).

\begin{example}\label{ex}
Let $\rho$ and $\rho'$ are unitary cuspidal representations of $\mathrm{GL_{k}(F)}$ and $\mathrm{GL_{k-1}(F)}$ respectively. For real number $0 <\alpha <\frac{1}{2}$,  consider the unitary representations:
    \[\pi = \mathbb{1}_2\cdot\nu^{\frac{1}{2}-\alpha}  \times \mathbb{1}_2\cdot\nu^{-\frac{1}{2}+\alpha} \times \rho \text{ and }\pi' =  \mathbb{1}_2\cdot\nu^{\alpha} \times \mathbb{1}_2\cdot\nu^{-\alpha} \times \rho'.\]
Here, $\mathbb{1}_2$ is the trivial representation of 
$\mathrm{GL_{2}(F)}$. Using Chan's generalized GGP relevant criterion (see \cite[Definition 2.4 and Theorem 4.1]{Cha_qbl}), one can verify (taking $\mathfrak{m}=\nu^\frac{1}{2}\rho + \left[\frac{3}{2}-\alpha\right]$ and $\mathfrak{n}=\rho' + \left[-\frac{1}{2}-\alpha\right]$) that $\mathrm{Hom_{GL_{k+3}(F)}}(\pi,\pi') \neq 0$. Crucially, when $\alpha \ne \frac{1}{4}$, this pair does not satisfy the necessary conditions of \cite[Theorem 5.7]{Gur}. \qed
\end{example}

This example reveals that the original Gan-Gross-Prasad relevance \cites{GGP2, Gur, Cha_crelle} for Arthur-type representations needed to be extended to the full unitary picture. We propose such an extension, which directly modifies the conditions in \cite[Theorem 5.7]{Gur} by incorporating an additional, previously missing case.

Let $\pi, \pi'$ be two irreducible unitary representations of some general linear group. Following Tadić's classification, we can write them uniquely (up to ordering) as
\begin{equation}\label{eq:unitary}
\pi \cong   \prod\limits_{i=1}^r \pi_{\rho_i}(u_i, v_i)(\alpha_i)  \text{ and } 
    \pi' \cong \prod\limits_{j=1}^l \pi_{\rho_j'}(u_j', v_j')(\beta_j) \tag{$\star$}
\end{equation}
for some unitary Speh representations $\pi_{\rho_i}(u_i, v_i), \pi_{\rho_j'}(u_j', v_j')$, real numbers $0 \le \alpha_i, \beta_j <\frac{1}{2}$ and unitary cuspidal representations $\rho_i, \rho'_j$. Here, for an irreducible representation $\tau$ of $\mathrm{GL_{n}(F)}$ and a real number $0 \le \alpha <\frac{1}{2}$, we write
\[ \tau(\alpha)= 
\begin{cases}
     \tau\nu^\alpha\times \tau\nu^{-\alpha} &\mbox{ for } 0 < \alpha <\frac{1}{2}\\
     \tau &\mbox{ for } \alpha =0.
\end{cases}
\]

\begin{definition}[Gan–Gross–Prasad relevance]\label{def:extended_relevance}
The unitary representations $\pi $ and $ \pi'$ (as in \eqref{eq:unitary})  are called Gan–Gross–Prasad relevant if there are disjoint partitions
  \[\{1,...,r\}=I_1 \sqcup I_2 \sqcup I_3 \sqcup I_4 ,~~~\{1,...,l\}=J_1 \sqcup J_2 \sqcup J_3 \sqcup J_4,\] and three bijections $\lambda_1: I_1 \longrightarrow J_1$, $\lambda_2: I_2 \longrightarrow J_2$, and $\lambda_3: I_3 \longrightarrow J_3$ such that the following relations hold:
  \begin{align}
      \left(u'_{\lambda_1(i)},v'_{\lambda_1(i)}\right)&=\left(u_i,v_i+1 \right), && \beta_{\lambda_1(i)}=\alpha_i ,  &&\rho'_{\lambda_1(i)} \cong \rho_i~~ \text{ for all } i \in I_1\tag{R1}\\
      \left(u'_{\lambda_2(i)},v'_{\lambda_2(i)}\right)&=\left(u_i,v_i-1 \right), && \beta_{\lambda_2(i)}=\alpha_i,  &&\rho'_{\lambda_2(i)} \cong \rho_i ~~ \text{ for all } i \in I_2 \tag{R2}\\ 
      \left(u'_{\lambda_3(i)},v'_{\lambda_3(i)}\right)&=\left(u_i,v_i \right), && \beta_{\lambda_3(i)}=\frac{1}{2}-\alpha_i ,  &&\rho'_{\lambda_3(i)} \cong \rho_i~~ \text{ for all } i \in I_3 \tag{R3}
  \end{align}
  \begin{equation}
    v_i = 1 \text{ for all } i \in I_4, \text{ and }  v_j' = 1 \text{ for all } j \in J_4. \tag{R4}  
  \end{equation}\qed
\end{definition}
The key new condition is (R3), which allows a factor $\pi_\rho(u,v)(\alpha)$ from $\pi$ to be matched with a factor $\pi_\rho(u,v)(\frac{1}{2}-\alpha)$ from $\pi'$. This new case is precisely what accounts for the phenomenon in Example \ref{ex}. It is easy to observe that when both $\pi$ and $\pi'$ are Arthur-type representations, Definition \ref{def:extended_relevance} coincides with the original definition of Gan–Gross–Prasad relevance in \cite[Section 3]{GGP2} (see \cite[Definition 5.1]{Gur} for more details). Recently, Chen-Chen \cite{CC} established a necessary criterion for Gan–Gross–Prasad relevance for general linear groups over Archimedean local fields. The formulation of Definition \ref{def:extended_relevance} in the non-Archimedean setting is inspired by their work. We now state one of our main results for quotient branching of unitary representations in the non-Archimedean case.

\begin{theorem}[Unitary Branching Laws]\label{main_thm:unitary}
 Let $\pi$ and $\pi'$ be the irreducible unitary representations of $\mathrm{GL}_{n+1}(F)$ and $\mathrm{GL}_{n}(F)$ respectively. Then, \[\mathrm{Hom_{GL_{n}(F)}}(\pi, \pi') \neq 0\] if and only if  $\pi \text{ and }\pi'$  are Gan–Gross–Prasad relevant (Definition \ref{def:extended_relevance}).   
\end{theorem}
This result not only provides a necessary and sufficient condition but also corrects and extends the earlier work of Gurevich \cite{Gur}. In particular, it gives an alternative proof of the local non-tempered Gan–Gross–Prasad conjecture \cite[Conjecture 5.1]{GGP2} for $p$-adic general linear groups.

\subsection{Notion of generalized GGP relevant} 
To analyze the branching problem, Chan's work \cite{Cha_qbl} relies on an interplay between two operations: {\it derivatives}, which extract smaller representations by taking Jacquet-module along a unipotent subgroup, and {\it integrals}, which build larger representations via parabolic induction. The $\eta$-invariant packages the outcome of successive derivatives.

Consider an irreducible smooth representation $\pi$ of $\mathrm{GL}_n(F)$ and an essentially square integrable representation $\sigma$ (a generalized Steinberg representation $\mathrm{St}(\Delta)$ associated to a segment $\Delta$) of $\mathrm{GL}_\ell(F)$ for $\ell < n$. Let $N_{n-\ell, \ell} \subset \mathrm{GL}_n(F)$ be the unipotent radical of the standard parabolic subgroup corresponding to the partition $(n-\ell, \ell)$ of $n$.

\subsubsection{Derivatives:}By \cite{LM16}, there exists at most one irreducible representation $\tau$ of $\mathrm{GL}_{n-\ell}(F)$ such that 
\begin{equation*}\label{def:right_der}
	\tau \otimes \mathrm{St}(\Delta)   \hookrightarrow \mathrm{Jac}_{N_{n-\ell, \ell}}(\pi),   
\end{equation*}
where $\mathrm{Jac}_{N_{n-\ell, \ell}}(\pi)$ is the normalized Jacquet module of $\pi$ associated to $N_{n-\ell, \ell}$. If such $\tau$ exists, we call $\tau$ is the {\it right derivative} $\mathrm{D}^\mathrm{R}_{\Delta}(\pi)$ of $\pi$ under $\mathrm{St}(\Delta)$. If no such $\tau$ exists, we set $\mathrm{D}^\mathrm{R}_{\Delta}(\pi)=0$. Similarly, there exists at most one irreducible smooth representation $\tau^\prime$ of $\mathrm{GL}_{n-\ell}(F)$ such that $\mathrm{St}(\Delta) \otimes \tau^\prime   \hookrightarrow \mathrm{Jac}_{N_{\ell, n-\ell}}(\pi)$.  The {\it left derivative} $\mathrm{D}^\mathrm{L}_{\Delta}(\pi)$ is defined similarly. One can iterate this process and define an invariant $\varepsilon^\mathrm{R}_\Delta(\pi)$ to be the largest non-negative integer $k$ such that $\left(\mathrm{D}^\mathrm{R}_{\Delta}\right)^k(\pi)\neq 0$.

\subsubsection{$\eta$-invariant:} For a segment $[a,b]_\rho$, we define the $\eta$-invariant for $\pi$ by the following tuple of non-negative integers:
\[\eta_{[a,b]_\rho}(\pi)=\left( \varepsilon^\mathrm{R}_{[a,b]_\rho}(\pi), \varepsilon^\mathrm{R}_{[a+1,b]_\rho}(\pi),...,\varepsilon^\mathrm{R}_{[b,b]_\rho}(\pi)\right).\]

\subsubsection{Integrals:}Let $\pi$ be any irreducible representation and $\sigma$ be a square-irreducible representation. Then, the normalized parabolic induction $\pi \times \sigma$  (resp. $\sigma \times \pi$) has a unique simple submodule (see \cites{LM16, LM18} for more details). We denote $\mathrm{I}^\mathrm{R}_{\Delta}(\pi)$ (resp. $\mathrm{I}^\mathrm{L}_{\Delta}(\pi)$ for the unique simple submodule of $\pi \times \mathrm{St}(\Delta)$ (resp. $\mathrm{St}(\Delta) \times \pi$), and called the right (resp. left) integral of $\pi$ under $\mathrm{St}(\Delta)$.

\subsubsection{RdLi-commutativity:} 
The central notion of a strongly RdLi-commutative triple—where RdLi stands for right derivative and left integral—ensures that the operations of taking right derivatives and left integrals can be performed in any order without affecting the final result. We avoid defining the original definition of a tuple $(\Delta, \Delta', \pi)$ to be a {\it strongly RdLi-commutative triple}. We refer the reader to see \cite[Definition 2.2]{Cha_qbl} or \cite[Definition 1.1]{Cha_duality}. Applying \cite[Theorem 1.5]{Cha_duality}, its combinatorial formulation is as follows:
\begin{definition}[Combinatorially RdLi-commutative triple]\label{def:RdLi_comb}
  Let $\Delta, \Delta'$ be two segments and $\pi$ be an irreducible smooth representation of $\mathrm{GL}_n(F)$. The tuple $(\Delta, \Delta', \pi)$ is called a {\it combinatorially RdLi-commutative triple} if 
\[\mathrm{D}^\mathrm{R}_\Delta(\pi) \neq 0  \text{ and } \eta_\Delta\left(\mathrm{I}^\mathrm{L}_{\Delta'}(\pi)\right)=\eta_\Delta(\pi).\]
\end{definition}

For multisegments $\mathfrak{m}$ and $ \mathfrak{n}$, we define strongly RdLi-commutative triple for $(\mathfrak{m}, \mathfrak{n}, \pi)$ using the combinatorially RdLi-commutative triple and \cite[Theorem 1.5]{Cha_duality}. For the multisegment $\mathfrak{m}=\Delta_1+\cdots+\Delta_r$, we write the segments in $\mathfrak{m}$ in an ascending order by $\Delta_j \nprec \Delta_i$ for $i <j$. We define
\begin{align*}
  \mathrm{D}^\mathrm{R}_\mathfrak{m}(\pi) = \mathrm{D}^\mathrm{R}_{\Delta_r}\circ...\circ \mathrm{D}^\mathrm{R}_{\Delta_1}(\pi), &\text{ and } \mathrm{D}^\mathrm{L}_\mathfrak{m}(\pi) = \mathrm{D}^\mathrm{L}_{\Delta_1}\circ...\circ \mathrm{D}^\mathrm{L}_{\Delta_r}(\pi),\\
  \mathrm{I}^\mathrm{L}_\mathfrak{m}(\pi) = \mathrm{I}^\mathrm{L}_{\Delta_r}\circ...\circ \mathrm{I}^\mathrm{L}_{\Delta_1}(\pi), &\text{ and } \mathrm{I}^\mathrm{R}_\mathfrak{m}(\pi) = \mathrm{I}^\mathrm{R}_{\Delta_1}\circ...\circ \mathrm{I}^\mathrm{R}_{\Delta_r}(\pi).
\end{align*}

\begin{definition}[Strongly RdLi-commutative triple]\label{def:RdLi}
  Let $\mathfrak{m}=\Delta_1+\cdots+\Delta_r$ and $\mathfrak{n}=\Delta_1'+\cdots+\Delta_s'$  be two multisegments in ascending order. Write $\mathfrak{m}_i=\Delta_1+\cdots+\Delta_i$ and $\mathfrak{n}_j=\Delta_1'+\cdots+\Delta_j'$. Then, the tuple $(\mathfrak{m}, \mathfrak{n}, \pi)$ is called a {\it strongly RdLi-commutative triple} if the tuple $\left(\Delta_{i+1},\Delta_{j+1}', \mathrm{I}^\mathrm{L}_{\mathfrak{n}_j} \circ \mathrm{D}^\mathrm{R}_{\mathfrak{m}_i}(\pi)\right)$ is a combinatorially RdLi-commutative triple for any $0 \le i \le r-1$ and $0 \le j \le s-1$. \qed
\end{definition}
\subsubsection{Generalized GGP relevant:}We now recall the notion of a generalized GGP relevant pair as defined in \cite{Cha_qbl}. This notion plays the central role in general quotient branching. Let $\nu$ be the character $|\cdot|_F \circ \mathrm{det}$ of general linear groups, where $|\cdot|_F$ denotes the non-archimedean absolute value.
\begin{definition}[Generalized GGP relevant pair]\label{def:relevant}
Let $\pi$ and $\pi'$ be irreducible smooth complex representations of $\mathrm{GL}_{k}(F)$ and $\mathrm{GL}_{k'}(F)$ respectively. The pair $(\pi, \pi')$ is called a generalized GGP relevant pair if there exist multisegments $\mathfrak{m}$ and $\mathfrak{n}$ such that:
\begin{enumerate}
    \item Derivative Matching: $\mathrm{D}^\mathrm{R}_\mathfrak{m}(\nu^{1/2}\pi) \cong \mathrm{D}^\mathrm{L}_\mathfrak{n}(\pi')$,
    \item Commutativity: $\left(\mathfrak{m},\mathfrak{n} ,\nu^{1/2}\pi\right)$ is a strongly RdLi-commutative triple.
\end{enumerate} 
\end{definition}
The main result \cite[Theorem 4.1]{Cha_qbl} of Chan \cite{Cha_qbl} says that if $\pi$ and $\pi'$ are the irreducible smooth complex representations of $\mathrm{GL}_{n+1}(F)$ and $\mathrm{GL}_{n}(F)$ respectively, then $\mathrm{Hom_{GL_{n}(F)}}(\pi, \pi') \neq 0$  if and only if $(\pi,\pi')$  is a generalized GGP relevant pair. This provides a complete theoretical answer, but the definition itself is non-constructive, leaving open the question of how to find such $\mathfrak{m,n}$ and verify the conditions for a given pair.

\subsection{An algorithm for general branching laws}
The central contributions of this paper are twofold. First, we provide a practical solution to the general branching problem by demonstrating an easy (can be computed by hand) algorithm (see Algorithm \ref{alg:relevant}) to determine the generalized relevant pairs and hence the quotient branching laws in $\mathrm{GL}_n(F)$. 

Now we describe the main idea of the algorithm for determining the generalized GGP relevant pairs. It is a few-step procedure. Roughly, the algorithm (see Algorithm \ref{alg:relevant} for explicit description) proceeds as follows: {\it for a given pair $(\pi, \pi')$ of irreducible (non-generic) smooth representations,
\begin{enumerate}
    \item (Reduction): we either reduce it to a pair with smaller cuspidal support by applying a right derivative to $\pi$ or a left derivative to $\pi'$, or we swap the pair if necessary for reduction.
    \item (Generic base case): Repeating the above reduction or interchange-cum-reduction process eventually yields a pair $(\pi_*, \pi_*')$ of generic representations. For such a pair, the conditions for a generalized GGP relevant pair are automatically satisfied, and we can explicitly compute the minimal multisegments $\mathfrak{m}_*$ and $\mathfrak{n}_*$ by inspection.
    \item (Reverse Admissibility): Working backwards through the reduction steps, we check a specific non-vanishing condition (admissibility) for the relevant derivatives. If this condition holds at every step, then the original pair is relevant.
\end{enumerate}
}
\begin{theorem}[see Theorem \ref{thm:main}]
    The pair $(\pi, \pi')$ is a generalized GGP relevant pair if and only if the admissibility condition holds in each reduction step of the above algorithm (Algorithm \ref{alg:relevant}). 
\end{theorem}
Through this algorithm, we can determine whether $\mathrm{Hom_{GL_{n}(F)}}(\pi, \pi') \neq 0$ for any irreducible smooth complex representations $\pi$ and $\pi'$ of $\mathrm{GL}_{n+1}(F)$ and $\mathrm{GL}_{n}(F)$ respectively, provided that either Zelevinsky data or Langlands data (in terms of multisegment) of these representations are known. We have demonstrated a few non-trivial examples in section \ref{sec:example}.

\subsection{Equivalence of both notions}
For the special case of unitary representations, we prove a unifying equivalence that ties together the various formulations of the GGP relevant pair.
\begin{theorem}[Equivalence]\label{thm:unitary}
Let $\pi$ and $\pi'$ be two irreducible unitary representations of general linear groups. Then,
$(\pi, \pi')$ is a generalized GGP relevant pair (Definition \ref{def:relevant}) if and only if  $\pi$ and $\pi'$  are Gan–Gross–Prasad relevant (Definition \ref{def:extended_relevance}). 
\end{theorem}
This theorem bridges the abstract, general criterion with a concrete, parameter-based condition, making it a powerful tool for explicit computations. Combining this with Chan's result \cite[Theorem 4.1]{Cha_qbl} yields a clean and complete criterion for quotient branching laws for all unitary representations in Theorem \ref{main_thm:unitary}.

%

%
\begin{ack}
 The author would like to thank Kei Yuen Chan for his continuous suggestions and insightful discussions. This work is supported by the Research Grants Council of the Hong Kong Special Administrative Region, China (Project No: 17305223) and the NSFC grant for Excellent Young Scholar (Project No: 12322120) awarded to Kei Yuen Chan. The author is also grateful to Mohammed Saad Qadri for many helpful discussions. Finally the author would like to thank Maxim Gurevich for his encouragement and for contributing an Appendix to the article. 
\end{ack}

%
\section{Notations and preliminaries}\label{prelim}
Let $F$ be a non-archimedean local field with normalized
 absolute value $|\cdot|_F$. For every integer $n \geq 0$, let $G_n=\mathrm{GL}_n(F)$ be the rank $n$ general linear group over $F$, where $G_0$ is considered as the trivial group. The character $\nu_n:G_n \rightarrow \mathbb{C}^\times$ is defined by $\nu_n(g)=|\mathrm{det}(g)|_F$ for $g \in G_n$. For any integer $n \geq 0$, let $\mathrm{Rep}(G_n)$ be the category of smooth complex representations of $G_n$ and let $\mathrm{Irr}(G_n)$ be the set of irreducible objects of $\mathrm{Rep}(G_n)$ up to equivalence. For every integer $n \geq 1$, let $\mathrm{Irr^{cusp}}(G_n)$ be the set of irreducible supercuspidal representations of $G_n$ and $\mathrm{Irr^{unit}}(G_n)$ be the set of irreducible unitary representations of $G_n$. We set
 \[ ~\mathrm{Irr}=\bigsqcup\limits_{n\geq 0}\mathrm{Irr}(G_n), ~\mathrm{Irr^{cusp}}=\bigsqcup\limits_{n \geq 1}\mathrm{Irr^{cusp}}(G_n), \text{ and } \mathrm{Irr^{unit}}=\bigsqcup\limits_{n \geq 1}\mathrm{Irr^{unit}}(G_n).\]
Let $P_n=M_nN_n$ be a standard parabolic subgroup of $G_n$, where the Levi subgroup $M_n$ is isomorphic to $G_{n_1} \times \cdots \times G_{n_r}$ for some composition $n=n_1+\cdots+n_r$ of $n$. Let $\pi_i$ be a smooth representation of $G_{n_i}$ for $1 \leq i \leq r$ and let $\pi$ denote a smooth representation of $G_n$. The normalized parabolic-induced representation is denoted by $\pi_1 \times \cdots \times \pi_r = \mathrm{Ind}^{G_n}_{P_n} (\pi_1 \boxtimes \cdots \boxtimes \pi_r),$ and the normalized Jacquet module of $\pi$ with respect to $P_n$ is denoted by $\mathrm{Jac}_{N_n}(\pi)=\delta_n^{-\frac{1}{2}}\cdot \pi/\mathrm{span}\{ x \cdot v -v \mid x \in N, v \in \pi\}$, where $\delta_n$ is the modular character of $P_n$. We denote $n_\pi=n$ for $\pi \in \mathrm{Rep}(G_n)$.

\subsection{Cuspidal support}\label{sec:usp:supp} For each $\pi \in \mathrm{Irr}$, there exists a unique multiset $\{\rho_1,...,\rho_r\} \subset \mathrm{Irr^{cusp}}$ such that $\pi$ is a simple composition factor of $\rho_1 \times...\times\rho_r$. This multiset is called the cuspidal support of $\pi$ and is denoted $\mathrm{csupp}(\pi)$. For $\pi_1, \pi_2 \in \mathrm{Irr}$, the induced representation $\pi_1 \times \pi_2$ is irreducible unless there exists $\rho \in \mathrm{csupp}(\pi_1)$ such that either $\nu\rho$ or $\nu^{-1}\rho$ lies in $\mathrm{csupp}(\pi_2)$. We define an order between two cuspidal representations $\rho, \rho'$ as $\rho < \rho'$ (resp. $\rho = \rho'$ or $\rho > \rho'$) if there exists a non-negative number $a$ such that $\rho'=\nu^a\rho$ (resp. $\rho = \rho'$ or $\rho =\nu^a \rho'$). We define the cuspidal line $\rho$ by $\rho[\mathbb{Z}]=\{\nu^i \rho \mid i \in \mathbb{Z}\}$. Then, $\mathrm{Irr}_\rho$ consists of those $\pi \in \mathrm{Irr}$ such that $\mathrm{csupp}(\pi) \subset \rho[\mathbb{Z}]$.

\subsection{The notions of segment and multisegment \cite{Zel}}\label{sec:seg}
Let $a,b \in \mathbb{R}$ such that $b-a \in \mathbb{Z}_{\geq 0}$ and let $\rho \in \mathrm{Irr^{cusp}}(G_k)$. A segment in the cuspidal line $\rho$ is denoted either by $\Delta = [a,b]_\rho$, which is essentially the set $\{ \nu^a \rho, \nu^{a+1}\rho, ..., \nu^b \rho \}$ with the character $\nu=\nu_k$  or by a void set. We write  $[a,a]_\rho$ simply as $[a]_\rho$. Let $\mathrm{Seg}$ denote the set of all segments and $\mathrm{Seg}_\rho$ the set of segments in the cuspidal line $\rho$. For convenience, we set $[a,a-1]_\rho=\emptyset$ for $a \in \mathbb{R}$.  For a segment $\Delta = [a,b]_\rho$, the beginning (or start) is $s(\Delta)= \nu^a \rho$  with $s_\rho(\Delta)=a$, and the ending is $e(\Delta)=\nu^b \rho$ with $e_\rho(\Delta)=b$. The absolute length of $\Delta = [a,b]_\rho$ is denoted by $\ell_{abs}(\Delta)=(b- a + 1)n_\rho$. By convention, the length of the void segment is $0$. 

Two segments $\Delta, \Delta' \in \mathrm{Seg}$ are called {\it linked} if there exists supercuspidal $\rho \in \mathrm{Irr}^c$ and $a, b, a', b' \in \mathbb{R}$ with mutually integer difference such that $\Delta=[a,b]_\rho$, $\Delta'=[a',b']_\rho$ and one of the following holds: 
    \[a <a' \leq b+1 <b'+1 \text{ or }  a' < a \leq b'+1 < b+1.\]
If two segments are not linked, they are called unlinked. In particular, segments lying in distinct cuspidal lines are unlinked. For two linked segments $\Delta=[a,b]_\rho$ and $\Delta^\prime=[a^\prime,b^\prime]_\rho$, we say that $\Delta$ precedes $\Delta^\prime$, denoted $\Delta \prec \Delta^\prime$, if $a < a^\prime$, $b < b^\prime$ and $a^\prime \leq b+1$. 

For a non-void segment $\Delta = [a,b]_\rho$, we define $\overset{\rightarrow}{\Delta}=[a+1,b+1]_\rho,~ \overset{\leftarrow}{\Delta}=[a-1,b-1]_\rho,~   \Delta^+ =[a,b+1]_\rho,~\Delta^- = [a,b-1]_\rho,~ {^+}\Delta =[a-1,b]_\rho,$  and $ {^-}\Delta = [a+1,b]_\rho$ with
the convention that $\Delta^-$ and ${^-}\Delta$ are void if $a=b$. Similarly, $\nu^\alpha \Delta = [a+\alpha, b+\alpha]_\rho \text{ for } \alpha \in \mathbb{R}$.

A multisegment is a multiset $\mathfrak{m}=\{\Delta_1,...,\Delta_r\}$ of non-void segments and is represented as $\mathfrak{m}=\Delta_1+\cdots +\Delta_r$. Let $\mathrm{Mult}$ be the set of all multisegments and $\mathrm{Mult}_\rho$ be the set of those multisegments consisting of segments in the cuspidal line $\rho$. The absolute length of a multisegment $\mathfrak{m} \in \mathrm{Mult}_\rho$ is defined by $\ell_{abs}(\mathfrak{m})=\sum_{\Delta \in \mathfrak{m}}\ell_{abs}(\Delta)$with the convention that the empty multisegment has length $0$. The support of a multisegment
$\mathfrak{m}$ is the multiset of integers obtained by taking the union (with multiplicities) of the segments in $\mathfrak{m}$. For two multisegments $\mathfrak{m}, \mathfrak{m}^\prime \in \mathrm{Mult}$, we write $\mathfrak{m} + \mathfrak{m}^\prime$ for the union $\mathfrak{m} $ and $\mathfrak{m}^\prime$ counting multiplicities. For a segment $\Delta$, we set $\mathfrak{m}+\Delta=\mathfrak{m}+\{\Delta\}$ if $\Delta \neq \emptyset$, and $\mathfrak{m}+\Delta=\mathfrak{m}$ if $\Delta=\emptyset$. Similarly, we define $\mathfrak{m} - \mathfrak{m}^\prime$ and $\mathfrak{m}-\Delta$. For $\mathfrak{m} \in \mathrm{Mult}_\rho$, $\rho \in \mathrm{Irr^{cusp}}$, and $x \in \mathbb{R}$, we define the following multisets: $\mathfrak{m}[x] =\left\{[a,b]_\rho \in \mathfrak{m} \mid a=x \right\}$,
\begin{align*}
   \mathfrak{m}\left\langle x \right\rangle =\left\{[a,b]_\rho \in \mathfrak{m} \mid b=x \right\}, \text{ } \mathfrak{m}_{s=\rho}=\left\{\Delta \in \mathfrak{m} \mid s(\Delta) \cong \rho \right\}, \text{ and } \mathfrak{m}_{e=\rho}=\left\{\Delta \in \mathfrak{m} \mid e(\Delta) \cong \rho \right\}. 
\end{align*}
For a real number $0 \le \alpha <\frac{1}{2}$ and $\mathfrak{m} \in \mathrm{Mult}$, further we write
\[ \mathfrak{m}(\alpha)= 
\begin{cases}
     \mathfrak{m}\nu^\alpha + \mathfrak{m}\nu^{-\alpha} &\mbox{ for } 0 < \alpha <\frac{1}{2}\\
     \mathfrak{m} &\mbox{ for } \alpha =0.
\end{cases}
\]
\subsection{Langlands and Zelevinsky classification}\label{sec:LZ_classification}
For a segment $\Delta=[a,b]_\rho$ with $\rho \in \mathrm{Irr^{cusp}}(G_k)$, the normalized parabolic-induced representation $\nu^a \rho \times \nu^{a+1}\rho \times \cdots \times \nu^b \rho $ has a unique irreducible submodule, denoted by  $Z (\Delta)$, and has a unique irreducible quotient, denoted by $L(\Delta)$. The representation $L(\Delta)$ is often called the generalized Steinberg representation and is also denoted by $\mathrm{St}(\Delta)$. 

For the Zelevinsky classifications, we consider an ordered multisegment $\mathfrak{m}=\Delta_1+\Delta_2+ \cdots+\Delta_r$ with the property $\Delta_i \nprec \Delta_j$ for $i < j$. Then, the normalized parabolic-induced representation $Z(\Delta_1) \times Z(\Delta_2) \times \cdots \times Z(\Delta_r)$ has a unique irreducible submodule, which we denote by $Z(\mathfrak{m})$. Zelevinsky \cite{Zel} proved that every irreducible smooth representation $\pi$ of $G_n$ is isomorphic to $Z(\mathfrak{m})$ for a unique multisegment $\mathfrak{m}$.

For the Langlands classifications, let $\mathfrak{m}=\Delta_1+\Delta_2+ \cdots+\Delta_r$ be an ordered multisegment with the property $\Delta_i \nprec \Delta_j$ for $i > j$. Then, the induced representation $L(\Delta_1) \times L(\Delta_2) \times  \cdots \times L(\Delta_r)$ has a unique irreducible subrepresentation, which we denote by $L(\mathfrak{m})$. Further, every irreducible smooth representation $\pi$ of $G_n$ is isomorphic to $L(\mathfrak{m})$ for a unique multisegment $\mathfrak{m}$.

The two classifications are related by an involution  $\mathfrak{m} \mapsto \mathfrak{m}^\#$ on $\mathrm{Mult}$, characterized by the property $Z(\mathfrak{m}) \cong L(\mathfrak{m}^\#)$. Mœglin and Waldspurger \cite{MW} provide an algorithm to compute the multisegment $\mathfrak{m}^\#$ associated to each multisegment $\mathfrak{m} \in \mathrm{Mult}$ such that $Z(\mathfrak{m}) \cong L(\mathfrak{m}^\#) \text{ and } L(\mathfrak{m}) \cong Z(\mathfrak{m}^\#).$

For $\mathfrak{m} \in \mathrm{Mult}_\rho$, we denote $e(\mathfrak{m})=\mathrm{max} \left\{e(\Delta): \Delta \in \mathfrak{m}\right\}$ and $s(\mathfrak{m})=\mathrm{min} \left\{s(\Delta): \Delta \in \mathfrak{m}\right\}$ under the ordering as defined in \S \ref{sec:usp:supp}. We have $e(\mathfrak{m}^\#)=e(\mathfrak{m})$ and $s(\mathfrak{m}^\#)=s(\mathfrak{m})$. For $\pi \in \mathrm{Irr}_\rho$, there exists $\mathfrak{m} \in \mathrm{Mult}_\rho$ such that $\pi \cong L(\mathfrak{m})$ and we define \[s(\pi)=s(\mathfrak{m}) \text{ and } e(\pi)=e(\mathfrak{m}).\] 
\subsection{Ladder and Speh representations} Let $\rho \in \mathrm{Irr}^\mathrm{cusp}(\mathrm{GL}_k(F))$ and fix the cuspidal line $\rho$. A multisegment $\mathfrak{m}=[a_1, b_1]_\rho+...+[a_t, b_t]_\rho$is called a {\it ladder} if its segments satisfy the strict inequality \[a_1 < a_2 <... <a_t \text{ and }b_1 < b_2 <... < b_t.\] The unique irreducible representation $L(\mathfrak{m})$ obtained via Langlands classification is then called a {\it ladder representation}. Ladder representations form a fundamental class that includes many important families, such as the generalized Steinberg representations (when $t=1$) and the Speh representations described below.

A ladder of the special form \[\mathfrak{m}=\sum\limits_{i=0}^h [a+i, b+i]_\rho\] where $a,b \in \mathbb{R}$ with $b-a, h \in \mathbb{Z}_{\geq 0}$, is called Speh multisegment. The corresponding irreducible representation $L(\mathfrak{m})$ is called a (shifted) Speh representation.

\subsection{Unitary dual and Tadi\'c classifications}\label{sec:tadic}
Let  $\rho \in \mathrm{Irr^{cusp}}(\mathrm{G_k})$ be a unitary cuspidal representation.  For a positive integers $u$, define the segment \[\Delta_\rho(u)=\left[-\frac{u-1}{2}, \frac{u-1}{2}\right]_\rho\]
For a pair of positive integers $u,v$, we form the multisegment \[\mathfrak{m}_\rho(u,v)=\nu^{-\frac{v-1}{2}} \Delta_\rho(u) + \nu^{-\frac{v-1}{2}+1}\Delta_\rho(u)+...+\nu^{\frac{v-1}{2}} \Delta_\rho(u).\] Then, the irreducible representation $\pi_\rho(u,v)=L(\mathfrak{m}_\rho(u,v))$ (via Langlands classification) is called a {\it unitary Speh} representation. Every unitary Speh representation arises in this way, and these representations are precisely the irreducible unitary representations that are square-integrable modulo the center \cite{Tad}. 

A key property, due to Bernstein \cite{Ber}, is that for any two irreducible unitary representations $\pi, \pi'$, the normalized parabolic induction $\pi \times \pi'$ is again irreducible. This allows us to form products of unitary representations without encountering issues of reducibility.

An irreducible unitary representation $\pi$ is called {\it Arthur-type} if it can be written as a product \[\pi \cong \pi_1 \times ...\times \pi_r\] where each $\pi_i$ is a unitary Speh representation. Arthur-type representations form the building blocks of the unitary dual for general linear groups over non-archimedean local fields.

Beyond Arthur-type representations, the unitary dual contains complementary series representations. Let $\pi=L(\mathfrak{m})$ be an unitary Speh representation of $\mathrm{G_n}$, and let $\alpha$ be a real number with $0 < \alpha <\frac{1}{2}$. We can attach to $(\pi, \alpha)$ an irreducible unitary representation, called the {\it complementary series representation}, and defined by
\[\pi(\alpha):= \pi\nu^\alpha\times \pi\nu^{-\alpha}= L\left(\nu^{\alpha} \mathfrak{m} + \nu^{-\alpha}\mathfrak{m} \right).\]  
The complete classification of the unitary dual for general linear groups over non-archimedean local fields is due to Tadić \cite{Tad}. He proves that every irreducible unitary representation $\pi$ of $\mathrm{G_n}$ can be expressed uniquely (up to permutation of factors) as
\[\pi \cong \pi_1 \times ...\times \pi_r,\] where $\pi_i$ is either a unitary Speh or complementary series representation for $1 \le i \le r$. 

\subsection{Intersection-union process and minimality}\label{sec:minimal}
A multisegment $\mathfrak{n}$ is obtained from $\mathfrak{m}$ by an elementary intersection-union operation means for two segments $\Delta_1, \Delta_2 \in \mathfrak{m}$, we have \[\mathfrak{n}= \mathfrak{m} - \Delta_1-\Delta_2 + \{\Delta_1 \cap \Delta_2\} + \{\Delta_1 \cup \Delta_2\}.\] For two multisegments $\mathfrak{n}$ and $\mathfrak{m}$, we define an ordering $\mathfrak{n} \le_Z \mathfrak{m}$ if $\mathfrak{n}$ can be obtained by a sequence of elementary intersection-union operations from $\mathfrak{m}$ or $\mathfrak{n}=\mathfrak{m}$. A multisegment $\mathfrak{m}$ is said to be {\it minimal} for some property $\mathcal{P}$ if there does not exist any multisegment $\mathfrak{n} \lneq_Z \mathfrak{m}$ satisfying the property $\mathcal{P}$.

\subsection{Rd-minimal and Li-minimal} Let $\pi \in \mathrm{Irr}$. A multisegment $\mathfrak{m}$ is said to be Rd-minimal (resp. Li-minimal) to
$\pi$ if $\mathrm{D}^\mathrm{R}_\mathfrak{m}(\pi) \ne 0$ and $\mathrm{D}^\mathrm{R}_\mathfrak{m}(\pi) \ncong \mathrm{D}^\mathrm{R}_\mathfrak{n}(\pi)$ (resp. $\mathrm{I}^\mathrm{L}_\mathfrak{m}(\pi) \ncong \mathrm{I}^\mathrm{L}_\mathfrak{n}(\pi)$)
 for any other multisegment $\mathfrak{n} \lneq_Z \mathfrak{m}$. One can have analogous notions for Ld-minimal and Ri-minimal for left derivative and right integral, respectively.

 Let $(\pi,\pi')$ be a generalized GGP relevant pair as in Definition \ref{def:relevant}. By \cite[Theorem 2.7]{Cha_qbl}, there exist
unique multisegments $\mathfrak{m}$ and $\mathfrak{n}$ such that $\mathfrak{m}$ is Rd-minimal to $\nu^\frac{1}{2}\pi$ and $\mathfrak{n}$ is Ld-minimal to $\pi'$, and both the conditions in Definition \ref{def:relevant} are satisfied by $\mathfrak{m}$ and $\mathfrak{n}$.

We frequently use the following main results of \cite{Cha_csq_iii} without mentioning further. Let $\pi \in \mathrm{Irr}_\rho$ and $\mathfrak{m} \in \mathrm{Mult}_\rho$ be Rd-minimal to $\pi$. Then, (see \cite[Theorem 1.3, 1.4 and Corollary 1.5]{Cha_csq_iii})
\begin{enumerate}
    \item any submultisegment $\mathfrak{m}'$ of $\mathfrak{m}$ is also Rd-minimal to $\pi$ and in particular, $\mathrm{D}^\mathrm{R}_{\mathfrak{m}'}(\pi) \ne 0$.
    \item for any submultisegment $\mathfrak{m}'$ of $\mathfrak{m}$, the submultisegment $\mathfrak{m-m'}$ is Rd-minimal to $\mathrm{D}^\mathrm{R}_{\mathfrak{m}'}(\pi)$ and \[\mathrm{D}^\mathrm{R}_{\mathfrak{m}}(\pi)  \cong  \mathrm{D}^\mathrm{R}_{\mathfrak{m-m'}} \circ   \mathrm{D}^\mathrm{R}_{\mathfrak{m}'}(\pi)\]
    \item Write $\mathfrak{m}=\Delta_1 + ...+\Delta_r$ in any ordering. Then, $\mathrm{D}^\mathrm{R}_{\mathfrak{m}}(\pi)  \cong  \mathrm{D}^\mathrm{R}_{\Delta_r} \circ ... \circ  \mathrm{D}^\mathrm{R}_{\Delta_1}(\pi)$.
\end{enumerate}
The similar results also hold for Ld-minimal situations.


\subsection{Gelfand-Kazhdan involution} \label{sec:gk_invol}
Let $\theta: G_n \rightarrow G_n$ be given by $\theta(g)=g^{-T}$, the inverse transpose of the matrix $g \in G_n$. This induces a covariant auto-equivalence $\theta: \mathrm{Rep}(G_n) \rightarrow \mathrm{Rep}(G_n)$. On the combinatorial side, we define $\theta: \mathrm{Seg}_{\rho}\rightarrow \mathrm{Seg}_{\rho^{\vee}}$ by $\theta\left([a,b]_{\rho}\right)=[-b,-a]_{\rho^{\vee}}$ and the map 
\[ \Theta: \mathrm{Mult}_{\rho}\rightarrow \mathrm{Mult}_{\rho^{\vee}}, \text{ given by } \Theta( \Delta_1+ \ldots+ \Delta_k)= \theta(\Delta_1)+ \ldots + \theta(\Delta_k) .\]

\subsection{Removal process}\label{sec:removal}
For a given multisegment $\mathfrak{m} \in \text{Mult}_\rho$ and a segment $\Delta \in \text{Seg}_\rho$, Chan \cite[Definition 8.2]{Cha_tams} associates a new segment $\mathfrak{r}^\mathrm{R}(\Delta, \mathfrak{m})$ (resp. $\mathfrak{r}^\mathrm{L}(\Delta, \mathfrak{m})$) via a removal process. We start by defining an order $<^L$ on segments in $\text{Seg}_\rho$ by: $[a,b]_\rho <^L [a^\prime,b^\prime]_\rho$ if $a < a^\prime$, or $a=a^\prime$ and $b < b^\prime$.

Suppose $\mathfrak{m} \in \text{Mult}_\rho$ and $\Delta=[a,b]_\rho$ is a segment such that there exists a segment in $\mathfrak{m}$ of the form $[a,c]_\rho$ for some $c \geq b$. The {\it (right) removal process} $\mathfrak{r}^\mathrm{R}(\Delta, \mathfrak{m})$ on $\mathfrak{m}$ by $\Delta$ is as follows: 

\begin{enumerate}
    \item Choose the shortest segment $\Delta_1=[a_1, b_1]_\rho$ in $\mathfrak{m}$ with $a_1=a$ and $b_1 \geq b$.
    \item For $i \geq 2$, recursively choose 
segments $\Delta_i=[a_i, b_i]_\rho$ such that $\Delta_i$ is the minimal segment in $\mathfrak m$ with respect to $<^L$ satisfying $a_{i} > a_{i-1}$ and
$b \leq b_i < b_{i-1}$. This process terminates (after a finite number of steps) when no further such segment can be found, and let $\Delta_r$ be the last segment in the process.
    \item Define the new truncation segments:
                   \[\widetilde{\Delta}_i = [a_{i+1}, b_i]_\rho  \text{ for } 1 \leq i<r, \text{ and }
                       \widetilde{\Delta}_r = [b+1, b_r]_\rho.\]
    \item Then, 
    \[\mathfrak{r}^\mathrm{R}(\Delta, \mathfrak{m})=\mathfrak{m} - \sum\limits_{i=1}^{r} \Delta_i +  \sum\limits_{i=1}^{r} \widetilde{\Delta}_i.\]
\end{enumerate}
For a multisegment $\mathfrak{n}=\Delta_1+...+\Delta_r \in \text{Mult}_\rho$, with segments written in ascending
order, we define
     \[\mathfrak{r}^\mathrm{R}(\mathfrak{n}, \mathfrak{m})= \mathfrak{r}^\mathrm{R}(\Delta_r,\mathfrak{r}^\mathrm{R}(\Delta_{r-1},...\mathfrak{r}^\mathrm{R}(\Delta_1,\mathfrak{m})...)).\]
The left removal process $\mathfrak{r}^\mathrm{L}(\Delta, \mathfrak{m})$ and its iteration are defined analogously; they satisfy
\[\mathfrak{r}^\mathrm{L}(\Delta, \mathfrak{m})=\Theta\left(\mathfrak{r}^\mathrm{R} \left( \theta(\Delta), \Theta(\mathfrak{m}) \right) \right) \text{ and } \mathfrak{r}^\mathrm{L}(\mathfrak{n}, \mathfrak{m})=\Theta\left(\mathfrak{r}^\mathrm{R} \left( \Theta(\mathfrak{n}), \Theta(\mathfrak{m}) \right) \right),\]
where $\theta$ and $\Theta$ denote the Gelfand–Kazhdan involution (see section \ref{sec:gk_invol}).

\subsection{A sufficient condition for strongly RdLi-commutativity} In practice, verifying that a triple 
$(\Delta, \Delta', \pi)$ is strongly RdLi-commutative can be nontrivial. The following lemma provides a simple sufficient condition that covers many cases encountered in our algorithm. We will use this result extensively throughout the paper to establish strong commutativity without resorting to the full definition.
\begin{lemma}\label{lem:example}
Let $\Delta, \Delta'$ be two segments. Suppose that at least one of the following conditions holds:
\begin{enumerate}
    \item $\Delta \cap \Delta' = \emptyset$ (i.e., the two segments have disjoint cuspidal supports);
    \item $s(\Delta')< s(\Delta)$; 
    \item $e(\Delta')< e(\Delta)$.
\end{enumerate}
Then, for any irreducible smooth representation $\pi \in \mathrm{Irr}$ with $\mathrm{D}^\mathrm{R}_{\Delta}(\pi) \neq 0$, the tuple $(\Delta, \Delta', \pi)$ is a strongly RdLi-commutative triple
\end{lemma}
\begin{proof}
The statement follows immediately from \cite[Example 9.2 and Theorem 9.4]{Cha_qbl}. The essential idea is that under any of the given geometric conditions, the operations of taking the right derivative with respect to $\Delta$ and the left integral with respect to $\Delta'$ commute in the appropriate sense, and the required $\eta$-invariant condition is automatically satisfied.
\end{proof}

\section{Algorithms for derivatives and integrals}\label{sec:der_int}
Algorithms for computing derivatives and integrals (also known as Jacquet modules and parabolic inductions) have been studied extensively in the literature. Foundational work includes the combinatorial approaches of Jantzen \cite{Jan}, Minguez \cite{Min}, and Lapid–Minguez \cite{LM16}. More recently, Chan and Pattanayak \cite{CP} provided a unified algorithmic framework for computing these operations in both the Zelevinsky and Langlands classifications.

In this section, we recall the algorithms for computing St-derivatives (i.e., derivatives with respect to a generalized Steinberg representation $\mathrm{St}(\Delta)$) in the Langlands classification, as developed by Lapid–Minguez \cite{LM16} and Chan–Pattanayak \cite{CP}. These algorithms are essential for executing the main algorithm for quotient branching laws presented in Section \ref{sec:alg}. We also recall the algorithm for computing the highest derivative multisegment from \cite{CP}, which plays a key role in the reduction steps of Algorithm \ref{alg:relevant}.

\subsection{Algorithms of Lapid-Minguez}
Lapid and Minguez \cite{LM16} developed a systematic method for computing right and left derivatives in the Langlands classification using the Zelevinsky involution. We recall their main result. For a multisegment $\mathfrak{m} \in \mathrm{Mult}$ and a cuspidal representation $\tau \in \mathrm{Irr^{cusp}}$, define the following sub-multisegments:
\[\mathfrak{m}_{\le_e \tau}=\{\Delta \in \mathfrak{m} \mid e(\Delta) \le \tau \} \text{ and } \mathfrak{m}_{{_e}> \tau}=\{\Delta \in \mathfrak{m} \mid e(\Delta) > \tau \}.\]

\begin{theorem}\cite[Corollary 6.8]{LM16}
Let $\Delta \in \mathrm{Seg}_\rho$ and $\pi=L(\mathfrak{m}) \in \mathrm{Irr}_\rho$ for some $\mathfrak{m} \in \mathrm{Mult}_\rho$. Then, $\mathrm{D}^\mathrm{L}_\Delta \left(\pi \right) \ne 0$ if and only if there exists a multisegment $\mathfrak{m}_*$ such that 
\[\left( \mathfrak{m}_{\le_e ~e(\Delta)} \right)^\# = \Delta^\# + \mathfrak{m}_*.\]
If such $\mathfrak{m}_*$ exists, we define \[\mathcal{D}_\Delta^\mathrm{L}(\mathfrak{m})= \mathfrak{m}_*^\# + \mathfrak{m}_{{_e}> ~e(\Delta)},\] and otherwise we set $\mathcal{D}_\Delta^\mathrm{L}(\mathfrak{m})= \infty$ (indicating that the left derivative vanishes). Further, define the right version
\[\mathcal{D}_{\Delta}^\mathrm{R}(\mathfrak{m}) = \begin{cases}
   \Theta \left( \mathcal{D}_{\theta(\Delta)}^\mathrm{L} \left(\Theta(\mathfrak{m})\right) \right)     &\mbox{if } \mathcal{D}_\Delta^\mathrm{L}(\mathfrak{m}) \neq \infty\\
     \infty &\mbox{otherwise }. 
\end{cases}
\]
Then, we have
 \[\mathrm{D}^\mathrm{R}_\Delta \left(L(\mathfrak{m}) \right) \cong \begin{cases}
     L \left( \mathcal{D}_\Delta^\mathrm{R}(\mathfrak{m})  \right)    &\mbox{if } \mathcal{D}_\Delta^\mathrm{R}(\mathfrak{m}) \neq \infty\\
     0 &\mbox{otherwise }.
\end{cases} \] and 
\[\mathrm{D}^\mathrm{L}_\Delta \left(L(\mathfrak{m}) \right) \cong \begin{cases}
     L \left( \mathcal{D}_\Delta^\mathrm{L}(\mathfrak{m})  \right)    &\mbox{if } \mathcal{D}_\Delta^\mathrm{L}(\mathfrak{m}) \neq \infty\\
     0 &\mbox{otherwise }.
\end{cases} \]
\end{theorem}

\subsection{An alternate algorithm of St-derivatives}
Chan and Pattanayak \cite{CP} developed an alternative, more explicit combinatorial algorithm for computing derivatives in the Langlands classification. This algorithm proceeds by analyzing upward sequences within a multisegment. 

\begin{definition}[Upward sequence $\mathcal{U}$]
 Let  $\mathfrak{n} \in \mathrm{Mult}_\rho$ in a fixed cuspidal line $\rho$. Define its {\it upward sequence $\mathcal{U}(\mathfrak{n})$} as follows:
 \begin{enumerate}
     \item Let $a_1$ be the smallest number for which $\mathfrak{n}[a_{1}] \neq \emptyset$. Choose the longest segment $\Delta_{1}\in \mathfrak{n}[a_{1}]$.
     \item Recursively for $j \geq 2$, if possible, find the smallest number $a_{j}> a_{j-1}$ such that there exists a segment $\Delta^\prime_j \in \mathfrak{n}[a_{j}]$ with $\Delta_{j-1} \prec \Delta^\prime_j$. Then, choose the longest segment $\Delta_{j}\in \mathfrak{n}[a_{j}]$ satisfying $\Delta_{j-1} \prec \Delta_j$.
     \item This process terminates after a finite number of steps, say $r$.  The upward sequence is then defined as
\[\mathcal{U}(\mathfrak{n})=\Delta_1 + \Delta_2+\ldots + \Delta_r.\] 
 \end{enumerate}
 If no such segment exists at any step, the process terminates and we set $\mathcal{U}(\mathfrak{n})=\emptyset$. 
\end{definition}

For a given multisegment $\mathfrak{m} \in \mathrm{Mult}_\rho$ and a segment $\Delta=[a,b]_\rho \in \mathrm{Seg}_\rho$, we define the sub-multisegmnt \[\mathfrak{m}_{[a,b]}=\{[a^\prime, b^\prime]_\rho \in \mathfrak{m} \mid a \leq a^\prime \leq b+1 \leq b^\prime +1 \}.\] This captures all segments in $\mathfrak{m}$ that are “linked” to $\Delta$ in a way that could contribute to the derivative computation.

\begin{algorithm}\cite[Algorithm 3.3]{CP}\label{alg:der:Lang}
Suppose $\mathfrak{m} \in \mathrm{Mult}_\rho$ and $[a,b]_\rho \in \mathrm{Seg}_\rho$ are given. We set $\mathfrak{m}_1=\mathfrak{m}_{[a,b]}$.

Step 1. ({\bf Arrange all upward sequences}):  Compute the upward sequence $\mathcal{U}(\mathfrak{m}_1)= \Delta_{1,1}+ \Delta_{1,2}+ \ldots+ \Delta_{1,r_1}$ with $\Delta_{1,j} \prec \Delta_{1,j+1}$. Recursively for $i \ge 2$, define \[\mathfrak{m}_{i}= \mathfrak{m}_{i-1} - \mathcal{U}(\mathfrak{m}_{i-1}) \text{ and } \mathcal{U}(\mathfrak{m}_i)= \Delta_{i,1} + \Delta_{i,2} + \ldots +\Delta_{i,r_i} \text{ with }\Delta_{i,j} \prec \Delta_{i,j+1}.\]
Continue until $\mathfrak{m}_{k+1}=\emptyset$ for some integer $k$. 

Step 2. ({\bf Find removable free points}): For each segment $\Delta_{i,j}=\left[a_{i,j},b_{i,j}\right]_\rho$ obtained in Step 1, we define its removable free section as
\begin{align*}\label{eqn removable free}
\mathfrak{rf}\left(\Delta_{i,j} \right):= \begin{cases}
    \left[a_{i,j},~ a_{i,j+1}-2 \right]_\rho &\mbox{ if } 1 \leq j < r_i\\
    \Delta_{i,r_i} &\mbox{ if } j=r_i.
\end{cases}
\end{align*}
For a real number $y$, we call $[y]_\rho$ a `{\it removable free point}' of $\Delta_{i,j} $  if $ \mathfrak{rf}\left(\Delta_{i,j}\right)=[x,z]_\rho$ and $x \leq y \leq z$.

Step 3. ({\bf Selection}): Select a collection of segments $\Delta_{i,j}$ having a removable free point according to the following recursive procedure:
\begin{enumerate}
    \item[(i)] Choose a segment $\Delta_{i_1,j_1} \in \mathfrak{m}_1$ (if it exists) where $i_1$ is the largest integer in $\{1,...,k\}$ such that $[a_{i_1,j_1},b]_\rho \subseteq \mathfrak{rf}\left(\Delta_{i_1,j_1} \right)$ for some $j_1 \in \left\{ 1,...,r_{i_1} \right\}$.
    \item[(ii)] Recursively for $t \geq 2$, we choose a segment $\Delta_{i_t,j_t} \in \mathfrak{m}_1$ (if it exists), where $i_t$ is the largest integer in $\{1,...,i_{t-1}\}$ such that 
     \[\left[a_{i_t,j_t}, a_{i_{t-1}, j_{t-1}}-1\right]_\rho \subseteq \mathfrak{rf}\left(\Delta_{i_t,j_t} \right).\]
    \item[(iii)] This process terminates after a finite number of steps, say $\ell$, and suppose $\Delta_{i_\ell, j_\ell}$ is the last segment selected.
\end{enumerate}

Step 4. ({\bf Construct the derivative}):
If $a_{i_\ell, j_\ell} = a,$ define truncations of the selected segments as follows:
 \[\Delta_{i_1, j_1}^\mathrm{trc} = \left[b+1, ~ b_{i_1, j_1}\right]_\rho, \text{ and } \Delta_{i_t, j_t}^\mathrm{trc}= \left[a_{i_{t-1},j_{t-1}}, ~ b_{i_t, j_t}\right]_\rho  \text{ for }  2 \leq i \leq \ell.\] Then, the right derivative multisegment in the Langlands classification is given by 
\[\mathcal{D}_{[a,b]_\rho}^\mathrm{R}(\mathfrak{m}) = \mathfrak{m}- \sum\limits_{t=1}^{\ell} \Delta_{i_t, j_t} + \sum\limits_{t=1}^{\ell} \Delta_{i_t, j_t}^\mathrm{trc}.\]
 If $\Delta_{i_1,j_1}$ does not exist, or $a_{i_\ell, j_\ell} \neq a,$ we set $\mathcal{D}_{[a,b]_{\rho}}^\mathrm{R}(\mathfrak{m}) = \infty,$  indicating that $\mathrm{D}^\mathrm{R}_{[a,b]_\rho}(L(\mathfrak{m}))=0$ \qed
\end{algorithm}

\begin{remark} \label{rmk:removal}
For a generic representation $\pi = L(\mathfrak{m}) \in \mathrm{Irr}_\rho$ and a segment $\Delta \in \mathrm{Seg}_\rho$, the derivative computed via Algorithm \ref{alg:der:Lang} coincides with the removal process defined in Section \ref{sec:removal}. Specifically,
    \[\mathcal{D}^\mathrm{R}_\Delta(\mathfrak{m}) = \mathfrak{r}^\mathrm{R}(\Delta, \mathfrak{m}), \text{ and } \mathcal{D}^\mathrm{L}_\Delta(\mathfrak{m}) = \mathfrak{r}^\mathrm{L}(\Delta, \mathfrak{m}).\]
\end{remark}

\begin{lemma}[A special case]\label{lem:rmk}
Suppose $\mathfrak{n}_1= \sum\limits_{i=1}^k \mathcal{U}(\mathfrak{n}_i)$ is a multisegment, where $\mathcal{U}(\mathfrak{n}_{1})=[a, b]_\rho + [c+1, b+1]_\rho+...+[y_1, y_1']_\rho$ with $a <c$, and for $i>1$, recursively $\mathfrak{n}_i=\mathfrak{n}_{i-1}-\mathcal{U}(\mathfrak{n}_{i-1})$ and $\mathcal{U}(\mathfrak{n}_{i})=[x_i, b]_\rho + [x_i', b+1]_\rho+...+[y_i, y_i']_\rho$ with $a < x_i \le c \le x_i'-1$. Then,
\[\mathcal{D}_{[a,c-1]_\rho}^\mathrm{R}(\mathfrak{n}_1) = \mathfrak{n}_1-[a,b]_\rho+[c,b]_\rho.\]    
\end{lemma}
\begin{proof}
 This follows immediately from Algorithm \ref{alg:der:Lang} by tracking the effect of the upward sequences and the selection process.
\end{proof}

\subsection{An algorithm for the highest derivatives multisegments} The highest derivative of a representation plays a crucial role in the reduction steps of our main algorithm. Intuitively, it captures the “maximal” derivative that can be taken while preserving irreducibility.

\begin{definition}[Highest derivative multisegment]\label{def:hd}
For $\pi \in \mathrm{Irr(GL_n(F))}$, the {\it highest right derivative multisegment} of $\pi$, denoted by $\mathfrak{hd}^\mathrm{R}(\pi)$, is the unique minimal multisegment (with respect to partial order $\le_Z$ defined in Section \ref{sec:minimal}) such that
\[\mathrm{D}^\mathrm{R}_{\mathfrak{hd}^\mathrm{R}(\pi)}(\pi) \cong \pi^- \quad \] where $\pi^-$ denotes the highest right Bernstein-Zelevinsky derivative of $\pi$. Analogously, the {\it highest left derivative multisegment} $\mathfrak{hd}^\mathrm{L}(\pi)$ satisfies 
\[\mathrm{D}^\mathrm{L}_{\mathfrak{hd}^\mathrm{L}(\pi)}(\pi) \cong {^-}\pi,\]
where ${^-}\pi$ denotes the highest left Bernstein-Zelevinsky derivative of $\pi$.
\end{definition}

\begin{algorithm}\cite[Algorithm 7.4]{CP}\label{alg:hd:multisegment:zel}
Let $\mathfrak m \in \mathrm{Mult}_{\rho}$ be a multisegment. The highest right derivative multisegment $\mathfrak{hd}^\mathrm{R}(\mathfrak m)$ is constructed as follows. 
\begin{enumerate}
    \item Set $\mathfrak m_1=\mathfrak m$
    \item Iteration: For $i \ge 1$, let $a_i$ be the smallest integer such that $\mathfrak m_i\langle a_i \rangle\neq \emptyset$. Choose the longest segment $\Delta_{i,a_i}$ in $\mathfrak m_i\langle a_i\rangle$. For $j \geq a_i+1$,  recursively choose the longest segment $\Delta_{i,j}$ in $\mathfrak m_i\langle j\rangle$ such that $\Delta_{i,j}$ is linked to $\Delta_{i,j-1}$. This process terminates when no such segment can be found; let $b_i$ be the ending index of the last segment $\Delta_{i,b_i}$ chosen.
    \item Removal: Define
              \[ \mathfrak m_{i+1}=\mathfrak m_i-\sum\limits_{j=a_i}^{b_i}\Delta_{i,j}.\]
    \item Termination: Repeat steps 2–3 until $\mathfrak{m}_{\ell +1} = \emptyset$. Then, define
                \[ \mathfrak{hd}^\mathrm{R}(\mathfrak m)=[a_1,b_1]_{\rho}+...+[a_\ell,b_\ell]_{\rho} .\]
\end{enumerate} 
\end{algorithm}
For the left derivative, the highest left derivative multisegment is obtained via the Gelfand–Kazhdan involution:
\[\mathfrak{hd}^\mathrm{L}(\mathfrak m) =      \Theta \left( \mathfrak{hd}^\mathrm{R} (\Theta(\mathfrak m))  \right).\]
\begin{theorem}\cite[Theorem 7.5]{CP}
  For a multisegment $\mathrm{m} \in \mathrm{Mult}_\rho$, we have,
  \[\mathfrak{hd}^\mathrm{R}(Z(\mathfrak m))= Z\left( \mathfrak{hd}^\mathrm{R}(\mathfrak m)  \right)   \text{ and } \mathfrak{hd}^\mathrm{L}(Z(\mathfrak m))= Z\left( \mathfrak{hd}^\mathrm{L}(\mathfrak m) \right).\]
\end{theorem}

\section{Generic representations and branching law}\label{sec:generic}
Let $\pi$ and $ \pi'$ be irreducible generic representations of $\mathrm{GL}_{n+1}(F)$ and $\mathrm{GL}_{n}(F)$, respectively. Generic representations—those possessing a non-zero Whittaker functional—form the building blocks of the representation theory of general linear groups over non-archimedean local fields. In this section, we show that for any such pair is automatically a generalized GGP relevant pair. This provides the base case for our inductive algorithm in Section \ref{sec:alg}.

Using the Rankin–Selberg theory, Jacquet, Piatetski-Shapiro, and Shalika \cite{JPSS} established the following fundamental result:\[\mathrm{dim}~\mathrm{Hom}_{\mathrm{GL}_{n}(F)}(\pi, \pi') =1.\]

In terms of the generalized GGP relevance criterion (Definition \ref{def:relevant}), this means that every generic pair $(\pi,\pi')$ with $n_{\pi}=n_{\pi'}+1$ is a generalized GGP relevant pair. Our goal in this section is to make this explicit: we construct explicit multisegments $\mathfrak{p}$ and $\mathfrak{q}$ that witness the relevance, and we verify the strongly RdLi-commutative condition using the sufficient criterion from Lemma \ref{lem:example}.

Let $\pi=L(\mathfrak{m})$ and $\pi'=L(\mathfrak{n})$ be any two generic representations, where $\mathfrak{m}, \mathfrak{n} \in \mathrm{Mult}$ are multisegments consisting of mutually unlinked segments. This is a characterization of generic representations: an irreducible representation is generic if and only if its Langlands multisegment consists of pairwise unlinked segments (see \cite{Zel}).

\subsection{Construction of $\mathfrak{p}_{\mathfrak{m}, \mathfrak{n}}$ and $\mathfrak{q}_{\mathfrak{m}, \mathfrak{n}}$}\label{sec:construction}
Fix a cuspidal representation $\rho \in \mathrm{Irr}^\mathrm{cusp}$. Let $\mathrm{Mult}^{\mathrm{ul}}_\rho \subset \mathrm{Mult}_\rho$ (resp. $\mathrm{Mult}^{\mathrm{ul}} \subset \mathrm{Mult}$) consists of those multisegment $\mathfrak{s} \in \mathrm{Mult}_\rho$ (resp. $\mathfrak{s} \in \mathrm{Mult}$) whose segments are mutually unlinked. For a general multisegment (possibly involving multiple cuspidal lines), we can decompose it as a sum over distinct cuspidal lines, as generic representations have no linking between segments in different cuspidal lines.

Let $\mathfrak{m}_\rho, \mathfrak{n}_\rho \in \mathrm{Mult}^{\mathrm{ul}}_\rho$. We now construct three multisegments $\mathfrak{p}_{\mathfrak{m}_\rho, \mathfrak{n}_\rho}, \mathfrak{q}_{\mathfrak{m}_\rho, \mathfrak{n}_\rho}$ and $\mathfrak{t}_{\mathfrak{m}_\rho, \mathfrak{n}_\rho}$ via an iterative matching procedure: first set $\mathfrak{m}_1= \mathfrak{m}_\rho$ and $\mathfrak{n}_1= \mathfrak{n}_\rho$. 

\begin{itemize}
    \item Selection step: Let $\nu^{a_1}\rho=\mathrm{min} \{s(\Delta) \mid \Delta \in \mathfrak{m}_1\}$ be the minimal starting point among segments in $\mathfrak{m}_1$. Choose the longest segment $\Delta_1=[a_1, b_1]_\rho \in \mathfrak{m}_1$ with starting point $s(\Delta_1)=\nu^{a_1}\rho$. If the set $\{[a', b']_\rho \in \mathfrak{n}_1 \mid a_1 \leq a' \leq b_1 \leq b' \}$ is non-empty, we define $[a_1', b_1']_\rho$ to be its longest segment; otherwise we set $[a_1', b_1']_\rho=\emptyset$.
    \item Removal step: Define $\mathfrak{m}_2= \mathfrak{m}_1 - [a_1, b_1]_\rho$ and $\mathfrak{n}_2= \mathfrak{n}_1 - [a_1', b'_1]_\rho$
    \item Repeat: Continue this process. After $k$ steps, we obtain $\mathfrak{m}_{k+1}=\emptyset$ (since $\mathfrak{m}_\rho$ is finite). Thus, we have produced a sequence of matched pairs of segments $\left([a_i, b_i]_\rho, [a_i', b_i']_\rho \right)$ for $i=1,...,k$, where some $[a_i', b_i']_\rho$ may be empty.
    \item Final construction: We define the following multisegments:
     \[\mathfrak{p}_{\mathfrak{m}_\rho, \mathfrak{n}_\rho}=\sum\limits_{i=1}^k [a_i, a'_i-1]_\rho, ~\mathfrak{q}_{\mathfrak{m}_\rho, \mathfrak{n}_\rho}= \sum\limits_{i=1}^k [b_i+1, b_i']_\rho+ \mathfrak{n}_{k+1} \text{ and }\mathfrak{t}_{\mathfrak{m}_\rho, \mathfrak{n}_\rho}=\sum\limits_{i=1}^k [a_i', b_i]_\rho,\] with the conventions that if $[a_i', b_i']_\rho=\emptyset$, then we set $[a_i, a'_i-1]_\rho=[a_i, b_i]_\rho$, and $[b_i+1, b_i']_\rho=\emptyset=[a_i', b_i]_\rho$.
\end{itemize}

\begin{lemma}\label{lem:derivative_R=L}
With the above notations, we have \[\mathrm{D}^\mathrm{R}_{\mathfrak{p}_{\mathfrak{m}_\rho, \mathfrak{n}_\rho}}\left(L(\mathfrak{m}_\rho)\right) \cong L\left(\mathfrak{t}_{\mathfrak{m}_\rho, \mathfrak{n}_\rho}\right) \cong \mathrm{D}^\mathrm{L}_{\mathfrak{q}_{\mathfrak{m}_\rho, \mathfrak{n}_\rho}}\left(L\left(\mathfrak{n}_\rho\right)\right).\]
\end{lemma}
\begin{proof}
This follows directly from the derivative algorithm in the Langlands classification (Algorithm \ref{alg:der:Lang}). Alternatively, one can observe that the removal process (Section \ref{sec:removal}) yields
    \[\mathfrak{r}^\mathrm{R}\left(\mathfrak{p}_{\mathfrak{m}_\rho, \mathfrak{n}_\rho}, \mathfrak{m}_\rho \right)=  \mathfrak{t}_{\mathfrak{m}_\rho, \mathfrak{n}_\rho}  =\mathfrak{r}^\mathrm{L}\left(\mathfrak{q}_{\mathfrak{m}_\rho, \mathfrak{n}_\rho},\mathfrak{n}_\rho \right).\]
The equivalence between derivatives and the removal process for generic representations (see Remark \ref{rmk:removal}) then gives the desired isomorphisms.
\end{proof}
Now let $\mathfrak{m}, \mathfrak{n} \in \mathrm{Mult}^{\mathrm{ul}}$ be arbitrary generic multisegments (possibly involving multiple cuspidal lines). Since segments in distinct cuspidal lines are automatically unlinked, we can decompose $\mathfrak{m}$ and $\mathfrak{n}$ as sums over distinct cuspidal types. More precisely, there exist mutually distinct cuspidal representations $\rho_1,...,\rho_k$ (with $\rho_i \not\simeq \nu^x \rho_j$ for any $x \in \mathbb{Z}$ when $i \neq j$) such that \[\mathfrak{m}=\mathfrak{m}_{\rho_1}+...+\mathfrak{m}_{\rho_k} \text{ and } \mathfrak{n}=\mathfrak{n}_{\rho_1}+...+\mathfrak{n}_{\rho_k},\] where each $\mathfrak{m}_{\rho_i}, \mathfrak{n}_{\rho_i} \in \mathrm{Mult}^{\mathrm{ul}}_{\rho_i}$ and it is possible that only one of  $\mathfrak{m}_{\rho_i}$ or $\mathfrak{n}_{\rho_i}$ is empty for a given $i$. For each $i$, we recall the construction of $\mathfrak{p}_{\mathfrak{m}_{\rho_i}, \mathfrak{n}_{\rho_i}}, \mathfrak{q}_{\mathfrak{m}_{\rho_i}, \mathfrak{n}_{\rho_i}}$ and $\mathfrak{t}_{\mathfrak{m}_{\rho_i}, \mathfrak{n}_{\rho_i}}$ as above. Then define
\begin{equation}\label{eq:constr_p_q}
    \mathfrak{p}_{\mathfrak{m}, \mathfrak{n}}= \sum\limits_{i=1}^k \mathfrak{p}_{\mathfrak{m}_{\rho_i}, \mathfrak{n}_{\rho_i}}, ~\mathfrak{q}_{\mathfrak{m}, \mathfrak{n}}= \sum\limits_{i=1}^k \mathfrak{q}_{\mathfrak{m}_{\rho_i}, \mathfrak{n}_{\rho_i}} \text{ and } \mathfrak{t}_{\mathfrak{m}, \mathfrak{n}}= \sum\limits_{i=1}^k \mathfrak{t}_{\mathfrak{m}_{\rho_i}, \mathfrak{n}_{\rho_i}}.
\end{equation}

\begin{lemma}\label{lem:csupp_seg}
    Assume all the above notations. Let $\Delta \in \mathfrak{p}_{\mathfrak{m}, \mathfrak{n}}$ and $\Delta' \in \mathfrak{q}_{\mathfrak{m}, \mathfrak{n}}$. Then at least one of the following holds:
    \begin{itemize}
        \item[(i)] $s(\Delta) > s(\Delta')$
        \item[(ii)] $e(\Delta) > e(\Delta')$
        \item[(iii)] $\Delta ~\cap ~\Delta'=\emptyset$ (i.e. the cuspidal supports of $\Delta$ and $\Delta'$ are different).
    \end{itemize}
\end{lemma}
\begin{proof}
This follows from the construction \eqref{eq:constr_p_q} of $\mathfrak{p}_{\mathfrak{m}, \mathfrak{n}}$ and $\mathfrak{q}_{\mathfrak{m}, \mathfrak{n}}$. For segments coming from the same cuspidal line $\rho_i$, the matching procedure ensures that the segments in $\mathfrak{p}_{\mathfrak{m}_{\rho_i}, \mathfrak{n}_{\rho_i}}$  are “to the left” of those in $\mathfrak{q}_{\mathfrak{m}_{\rho_i}, \mathfrak{n}_{\rho_i}}$ in terms of their starting or ending points. For segments from different cuspidal lines, condition (iii) applies.
\end{proof}

\subsection{Generalized GGP relevance for generic pairs}We now assemble the above constructions to prove that any generic pair is a generalized GGP relevant pair. Recall that in Definition \ref{def:relevant}, we consider the shifted representation $\nu^{1/2} \pi$. Since $\pi$ is generic, $\nu^{1/2} \pi$ is also generic, and its Langlands multisegment is $\nu^{1/2} \mathfrak{m}$.

\begin{proposition}\label{prop:generic}
Let $\pi=L(\mathfrak m)$  be a generic representations of $\mathrm{GL}_{n}(F)$  and  $\pi'=L(\mathfrak n)$ be a generic representations of $\mathrm{GL}_{n'}(F)$. Then, the pair $(\pi, \pi')$ is a generalized GGP relevant pair. Explicitly, it is relevant with respect to the Rd-minimal multisegment $\mathfrak{p}_{\nu^{\frac{1}{2}}\mathfrak{m}, \mathfrak{n}}$ and the Ld-minimal multisegment $\mathfrak{q}_{\nu^\frac{1}{2}\mathfrak{m}, \mathfrak{n}}$.
\end{proposition}
\begin{proof}
We first verify the derivative matching conditions of Definition \ref{def:relevant}. Using Lemma \ref{lem:derivative_R=L} and the fact that the construction respects decomposition over cuspidal lines, we compute
\[\mathrm{D}^\mathrm{R}_{\mathfrak{p}_{\nu^{1/2}\mathfrak{m}, \mathfrak{n}}}(\nu^{1/2}\pi) \cong   \mathrm{D}^\mathrm{R}_{\mathfrak{p}_{\nu^{1/2}\mathfrak{m}, \mathfrak{n}}}\left(L(\nu^{1/2}\mathfrak{m})\right) 
 \cong  \prod\limits_{i} \mathrm{D}^\mathrm{R}_{\mathfrak{p}_{\nu^{1/2}\mathfrak{m}_{\rho_i}, \mathfrak{n}_{\rho_i}}}\left(L\left(\nu^{1/2}\mathfrak{m}_{\rho_i}\right)\right)
 \cong  \prod\limits_{i} L\left(\mathfrak{t}_{\nu^{1/2}\mathfrak{m}_{\rho_i}, \mathfrak{n}_{\rho_i}}\right).\]
Similarly, 
\[ \mathrm{D}^\mathrm{L}_{\mathfrak{q}_{\nu^{1/2}\mathfrak{m}, \mathfrak{n}}}(\pi') \cong  \prod\limits_{i} \mathrm{D}^\mathrm{L}_{\mathfrak{q}_{\nu^{1/2}\mathfrak{m}_{\rho_i}, \mathfrak{n}_{\rho_i}}}(L(\mathfrak{n}_{\rho_i})) \cong  \prod\limits_{i} L\left(\mathfrak{t}_{\nu^{1/2}\mathfrak{m}_{\rho_i}, \mathfrak{n}_{\rho_i}}\right). \]

We now show that $\left(\mathfrak{p}_{\nu^{\frac{1}{2}}\mathfrak{m}, \mathfrak{n}}, \mathfrak{q}_{\nu^{\frac{1}{2}}\mathfrak{m}, \mathfrak{n}}, \nu^\frac{1}{2} \pi \right)$ is  a strongly RdLi-commutative triple. By Definition \ref{def:RdLi}, this requires checking that for any $\Delta \in \mathfrak{p}_{\nu^\frac{1}{2}\mathfrak{m}, \mathfrak{n}}$ and $\Delta' \in \mathfrak{q}_{\nu^\frac{1}{2}\mathfrak{m}, \mathfrak{n}}$, the triple $(\Delta, \Delta', \sigma)$ is stongly RdLi-commutative for certain intermediate representations $\sigma$. By Lemma \ref{lem:csupp_seg}, any such pair $(\Delta, \Delta')$ satisfies one of the sufficient conditions in Lemma \ref{lem:example}. Specifically, either the segments lie in distinct cuspidal lines, or they satisfy $s(\Delta)> s(\Delta')$ or $e(\Delta)> e(\Delta')$. In all cases, Lemma 2.1 guarantees that $(\Delta, \Delta', \sigma)$ is a strongly RdLi-commutative triple for any $\sigma$ with $\mathrm{D}^\mathrm{R}_\Delta(\sigma) \ne 0$. The required intermediate representations in Definition \ref{def:RdLi} are precisely of this form, so the condition holds.

By construction, $\mathfrak{p}_{\nu^{\frac{1}{2}}\mathfrak{m}, \mathfrak{n}}$ and $\mathfrak{q}_{\nu^\frac{1}{2}\mathfrak{m}, \mathfrak{n}}$ lie in $\mathrm{Mult}^\mathrm{ul}$ and so, they are minimal with respect to the partial order $\le_Z$. Therefore, we conclude that  $(\mathfrak{p}_{\nu^{\frac{1}{2}}\mathfrak{m}, \mathfrak{n}}, \mathfrak{q}_{\nu^{\frac{1}{2}}\mathfrak{m}, \mathfrak{n}}, \nu^\frac{1}{2} \pi)$ is generalized GGP relevant for the Rd-minimal $\mathfrak{p}_{\nu^{\frac{1}{2}}\mathfrak{m}, \mathfrak{n}}$ and the Ld-minimal $\mathfrak{q}_{\nu^\frac{1}{2}\mathfrak{m}, \mathfrak{n}}$.
\end{proof}

\begin{corollary}\cite{JPSS}\label{cor:generic}
Let $\pi$ be a generic representations of $\mathrm{GL}_{n+1}(F)$ and $\pi'$ be a generic representations of $\mathrm{GL}_{n}(F)$. Then, \[\mathrm{dim}_\mathbb{C} ~\mathrm{Hom}_{\mathrm{GL}_{n}(F)}(\pi, \pi') = 1.\]    
\end{corollary}
\begin{proof}
  This follows immediately from Proposition \ref{prop:generic} and \cite[Theorem 4.1]{Cha_qbl}.
\end{proof}

\begin{remark}
    The construction of $\mathfrak{p}_{\nu^{1/2}\mathfrak{m}, \mathfrak{n}}$ and $\mathfrak{q}_{\nu^{1/2}\mathfrak{m}, \mathfrak{n}}$ is explicit and can be carried out by hand for any given generic pair. This provides the “generic step” in Algorithm \ref{alg:relevant}, where, after a series of reductions, we obtain a generic pair and can directly read off the Rd-minimal and Ld-minimal multisegments.
\end{remark}

\section{Algorithm for general quotient branching}\label{sec:alg}
In this section, we present an algorithm that determines, for any given pair of irreducible smooth representations $(\pi, \pi')$, whether they form a generalized GGP relevant pair (Definition \ref{def:relevant}). By Chan's main theorem \cite[Theorem 4.1]{Cha_qbl}, this is equivalent to determining whether $\mathrm{Hom}_{\mathrm{GL}_{n}(F)}(\pi, \pi') \ne 0$ when $\pi$ and $ \pi'$ are representations of $\mathrm{GL}_{n+1}(F)$ and $\mathrm{GL}_{n}(F)$, respectively.

The algorithm is designed to be executable by hand (or by computer), given the Langlands or Zelevinsky data of the representations. It proceeds by systematically reducing the pair to a generic pair, where the relevance condition is easily verified (Section \ref{sec:generic}), and then checking a series of admissibility conditions in reverse order. The key ingredients are the algorithms for derivatives and highest derivative multisegments developed in Section \ref{sec:der_int}, together with the reduction and interchange results established below.

\subsection{Reductions process}
To give an algorithm for generalized quotient branching, we need the following reduction results. These results allow us to simplify a pair $(\pi, \pi')$ by removing certain `highest' or `lowest' components that do not affect the relevance condition. The techniques of proofs are based on \cite{Cha_qbl} and the properties of minimal multisegments from \cite{Cha_csq_iii}.

We begin with some terminology. For a cuspidal representation $\rho$, an irreducible representation $\pi$ is called strongly $R$-$\rho$-reduced (resp. strongly
$L$-$\rho$-reduced) if for any segment $\Delta$ with ending $e(\Delta) \cong \rho$ (resp. starting $s(\Delta) \cong \rho$), we have $\mathrm{D}^\mathrm{R}_\Delta(\pi)=0 ~ (\text{resp. }\mathrm{D}^\mathrm{L}_\Delta(\pi)=0)$. These notions capture when $\pi$ (resp. $\pi'$) has no "room" to take further derivatives at the cuspidal $\rho$.

\begin{proposition}[Reduction from the right]\label{prop:left_reduction}
Let $\pi, \pi' \in \mathrm{Irr}$ and $\rho$ be a $\leq$-maximal element in $\mathrm{csupp}(\pi)$. Suppose $\rho \notin \mathrm{csupp}(\nu^{-\frac{1}{2}}\pi')$. Define the multisegment
\[\mathfrak m^e_\rho(\pi)= \sum\limits_{\Delta \in \mathfrak{hd}^\mathrm{R}(\pi)_{e=\rho}} \Delta.\]
Then the following hold:
\begin{itemize}
    \item[(A)](\textbf{Reduction}:) If $(\pi, \pi')$ is a generalized GGP relevant pair, then $\left(\mathrm{D}^\mathrm{R}_{\mathfrak m^e_\rho(\pi)}(\pi), \pi'\right)$ is also generalized GGP relevant.
    \item[(B)](\textbf{Admissibility}:) Suppose  $\left(\mathrm{D}^\mathrm{R}_{\mathfrak m^b_\rho(\pi)}(\pi), \pi'\right)$ is a generalized GGP relevant pair with Rd-minimal $\mathfrak m$ and Ld-minimal $\mathfrak n$. If the {\it admissibility condition}  \[\mathrm{D}^\mathrm{R}_{\mathfrak m + \nu^{\frac{1}{2}}\mathfrak m^e_\rho(\pi)}(\nu^{\frac{1}{2}}\pi) \neq 0\] holds, then $(\pi, \pi')$ is also generalized GGP relevant with Rd-minimal $\mathfrak m + \nu^{\frac{1}{2}}\mathfrak m^e_\rho(\pi)$ and Ld-minimal $\mathfrak n$.
\end{itemize}
\end{proposition}
\begin{proof}
(A) As $(\pi, \pi')$ is a generalized GGP relevant pair, there exist a Rd-minimal $\mathfrak{m}$ and a Rd-minimal $\mathfrak{n}$ such that 
\begin{equation}\label{eq:P1}
 \left(\mathfrak{m}',\mathfrak{n},\nu^{1/2}\pi\right) \text{ is a minimal strongly RdLi-commutative triple,}  
\end{equation}
and
\begin{equation}\label{eq:P2}
   \mathrm{D}^\mathrm{R}_\mathfrak{m'}\left( \nu^{1/2}\pi\right) \cong \mathrm{D}^\mathrm{L}_\mathfrak{n}\left( \pi'\right).
\end{equation}
Let $\mathfrak{p}=\mathfrak m^e_\rho(\pi)$. Then, \[\nu^{1/2} \mathfrak{p}=\sum\limits_{\Delta \in \mathfrak{hd}^\mathrm{R}\left(\nu^{1/2}\pi\right)_{e=\nu^{1/2}\rho}} \Delta.\]
Since $\nu^{1/2}\rho \notin \mathrm{csupp}(\pi')$, by \eqref{eq:P2} one can observe that $\mathrm{D}^\mathrm{R}_\mathfrak{m'}\left( \nu^{1/2}\pi\right)$ is strongly $R$-$\nu^{1/2}\rho$-reduced and so $\eta_\Delta\left(\mathrm{D}^\mathrm{R}_\mathfrak{m'}\left( \nu^{1/2}\pi\right)\right)=0$ for all $\Delta \in \nu^{1/2} \mathfrak{p}$. Therefore, by \cite[Lemma 19.3]{Cha_qbl}, we have $\nu^{1/2} \mathfrak{p} \subset \mathfrak{m'}$. Now, apply \cite[Corollary 15.10]{Cha_qbl} on \eqref{eq:P1} to conclude that $\left(\mathfrak{m'}-\nu^{1/2} \mathfrak{p},\mathfrak{n},\nu^{1/2}  \mathrm{D}^\mathrm{R}_\mathfrak{p}(\pi)\right)$ is also a  minimal strongly RdLi-commutative triple. Also, isomorphism \eqref{eq:P2} gives
 \[\mathrm{D}^\mathrm{R}_{\mathfrak{m'}- \nu^{1/2} \mathfrak{p}}\left(\nu^{1/2}  \mathrm{D}^\mathrm{R}_{\mathfrak{p}}\left( \pi\right)\right) \cong \mathrm{D}^\mathrm{R}_{\mathfrak{m'}- \nu^{1/2} \mathfrak{p}}\circ  \mathrm{D}^\mathrm{R}_{\nu^{1/2}\mathfrak{p}}\left( \nu^{1/2}\pi\right) \cong \mathrm{D}^\mathrm{R}_\mathfrak{m'}\left( \nu^{1/2}\pi\right) \cong \mathrm{D}^\mathrm{L}_\mathfrak{n}\left( \pi'\right).\] 
Therefore, $\left(\mathrm{D}^\mathrm{R}_{\mathfrak m^e_\rho(\pi)}(\pi), \pi'\right)$ is also a generalized GGP relevant pair with Rd-minimal $\mathfrak{m'}- \nu^{1/2} \mathfrak m^e_\rho(\pi)$ and Ld-minimal $\mathfrak{n}$.

(B) We set $\mathfrak{p}=\mathfrak m^e_\rho(\pi)$. Given that $\left(\mathfrak{m}, \mathfrak{n},\nu^\frac{1}{2}\mathrm{D}^\mathrm{R}_{\mathfrak{p}}\left( \pi\right) \right)$ is a minimal strongly RdLi-commutative triple and
\begin{equation}\label{eq:P3}
 0\ne \mathrm{D}^\mathrm{R}_\mathfrak{m} \circ \mathrm{D}^\mathrm{R}_{\nu^{1/2}\mathfrak{p}}\left( \nu^\frac{1}{2} \pi\right) \cong   \mathrm{D}^\mathrm{R}_\mathfrak{m}\left( \nu^\frac{1}{2}\mathrm{D}^\mathrm{R}_{\mathfrak{p}}\left( \pi\right)\right) \cong \mathrm{D}^\mathrm{L}_\mathfrak{n}\left( \pi'\right).
\end{equation}
The admissibility condition $\mathrm{D}^\mathrm{R}_{\mathfrak m + \nu^{1/2}\mathfrak p}(\nu^{\frac{1}{2}}\pi) \neq 0$ and \eqref{eq:P3} ensure that 
\begin{equation}\label{eq:P4}
 \mathrm{D}^\mathrm{R}_{\mathfrak m + \nu^{1/2}\mathfrak p}(\nu^{\frac{1}{2}}\pi) \cong \mathrm{D}^\mathrm{R}_\mathfrak{m} \circ \mathrm{D}^\mathrm{R}_{\nu^{1/2}\mathfrak{p}}\left( \nu^\frac{1}{2} \pi\right)   
\end{equation}
and there exists a minimal multisegment $\mathfrak{m}'$ such that   $\mathrm{D}^\mathrm{R}_\mathfrak{m'} \left( \nu^\frac{1}{2} \pi\right) \cong \mathrm{D}^\mathrm{R}_{\mathfrak m + \nu^{1/2}\mathfrak p}\left(\nu^{\frac{1}{2}}\pi\right)$. Therefore, by \eqref{eq:P3} and \eqref{eq:P4} we have
\[\mathrm{D}^\mathrm{R}_\mathfrak{m'} \left( \nu^\frac{1}{2} \pi\right) \cong \mathrm{D}^\mathrm{L}_\mathfrak{n}\left( \pi'\right).\]
By a similar argument as in the above part (A), we have $\nu^\frac{1}{2} \mathfrak{p} \subset \mathfrak{m'}$.  Now by \cite[Theorem 1.4]{Cha_csq_iii}, we have $\mathfrak{m}'-\nu^{1/2}\mathfrak{p}$ is Rd-minimal to $\mathrm{D}^\mathrm{R}_{\nu^{1/2}\mathfrak{p}}\left( \nu^\frac{1}{2} \pi\right)$ and 
\begin{equation}\label{eq:P5}
 \mathrm{D}^\mathrm{R}_{\mathfrak{m}'-\nu^{1/2}\mathfrak{p}} \left( \mathrm{D}^\mathrm{R}_{\nu^{1/2}\mathfrak{p}}\left( \nu^\frac{1}{2} \pi\right)\right) \cong   \mathrm{D}^\mathrm{R}_\mathfrak{m'} \left( \nu^\frac{1}{2} \pi\right) \cong \mathrm{D}^\mathrm{R}_\mathfrak{m} \circ \mathrm{D}^\mathrm{R}_{\nu^{1/2}\mathfrak{p}}\left( \nu^\frac{1}{2} \pi\right)
\end{equation}
where the last isomorphism follows from \eqref{eq:P4}. As $\mathfrak{m}$ is also Rd-minimal to $\mathrm{D}^\mathrm{R}_{\nu^{1/2}\mathfrak{p}}\left( \nu^\frac{1}{2} \pi\right)$, by \eqref{eq:P5} and the uniqueness of minimality, we have $\mathfrak{m}'-\nu^{1/2}\mathfrak{p}=\mathfrak{m}$. Therefore, $\mathfrak{m}+\nu^{1/2}\mathfrak{p}$ is minimal to $\nu^\frac{1}{2}\pi$. By Lemma \ref{lem:example}, $(\nu^{1/2}\mathfrak{p}, \mathfrak{n}, \nu^\frac{1}{2}\pi)$ is a strongly RdLi-commutative triple. Also given that $\left(\mathfrak{m}, \mathfrak{n},\mathrm{D}^\mathrm{R}_{\nu^{1/2}\mathfrak{p}}\left(\nu^\frac{1}{2} \pi\right) \right)$ is a strongly RdLi-commutative triple. Therefore, by \cite[Lemma 15.8]{Cha_qbl} and Rd-minimality of $\mathfrak{m}+\nu^{1/2}\mathfrak{p}$ for $\nu^\frac{1}{2}\pi$, we have that $\left(\mathfrak{m}+\nu^{1/2}\mathfrak{p}, \mathfrak{n}, \nu^\frac{1}{2}\pi\right)$ is also a strongly RdLi-commutative triple.
\end{proof}

\begin{proposition}[Reduction from the left]\label{prop:right_reduction}
Let $\pi, \pi' \in \mathrm{Irr}$ and $\rho'$ be a $\leq$-minimal element in $\mathrm{csupp}(\pi')$. Suppose $\rho' \notin \mathrm{csupp}(\nu^{\frac{1}{2}}\pi)$. Define the multisegment
\[\mathfrak m^s_{\rho'}(\pi')= \sum\limits_{\Delta \in \mathfrak{hd}^\mathrm{L}(\pi)_{s=\rho'}} \Delta.\]
Then the following hold:
\begin{itemize}
    \item[(A)](\textbf{Reduction}:) If $(\pi, \pi')$ is a generalized GGP relevant pair, then $\left(\pi, ~\mathrm{D}^\mathrm{L}_{\mathfrak m^s_{\rho'}(\pi')}(\pi')\right)$ is also generalized GGP relevant.
    \item[(B)](\textbf{Admissibility}:) Suppose  $\left(\pi, ~\mathrm{D}^\mathrm{L}_{\mathfrak m^s_{\rho'}(\pi')}(\pi')\right)$ is a generalized GGP relevant pair with Rd-minimal $\mathfrak m$ and Ld-minimal $\mathfrak n$. If the {\it admissibility condition} \[\mathrm{D}^\mathrm{L}_{\mathfrak n + \mathfrak m^s_{\rho'}(\pi')}(\pi') \neq 0\] holds, then $(\pi, \pi')$ is also generalized GGP relevant with Rd-minimal $\mathfrak m $ and Ld-minimal $\mathfrak n + \mathfrak m^s_{\rho'}(\pi')$.
\end{itemize}
\end{proposition}
\begin{proof}
    This is the left analogue of Proposition \ref{prop:left_reduction}. 
\end{proof}

\subsection{Interchange preserves relevance}
The following results allow us to swap the roles of $\pi$ and $\pi'$ when neither reduction step applies directly. This is crucial for the algorithm, as it ensures that we can always make progress toward a generic pair.
\begin{lemma}\cite[Theorem 1.4]{Cha_tams}\label{lem:interchange_1}
    Let $\pi \in \mathrm{Irr}$ and $\mathfrak{m} \in \mathrm{Mult}$. Then the following hold:
    \begin{itemize}
        \item If $\mathrm{D}^\mathrm{R}_\mathfrak{m}(\pi) \neq 0$, then for $\mathfrak{n}= \mathfrak{r}^\mathrm{R}\left(\mathfrak{m}, \mathfrak{hd}^\mathrm{R}(\pi)\right)$, we have
    \[\mathrm{D}^\mathrm{R}_\mathfrak{n} \circ \mathrm{D}^\mathrm{R}_\mathfrak{m}(\pi) \cong \pi^-,\] where $\pi^-$ denotes the highest right Bernstein-Zelevinsky derivative of $\pi$.
    \item If $\mathrm{D}^\mathrm{L}_\mathfrak{m}(\pi) \neq 0$, then for $\mathfrak{n}= \mathfrak{r}^\mathrm{L}\left(\mathfrak{m}, \mathfrak{hd}^\mathrm{L}(\pi)\right)$, we have
    \[\mathrm{D}^\mathrm{L}_\mathfrak{n} \circ \mathrm{D}^\mathrm{L}_\mathfrak{m}(\pi) \cong {^-}\pi ,\] where ${^-}\pi$ denotes the highest left Bernstein-Zelevinsky derivative of $\pi$.
    \end{itemize}
    
\end{lemma}

The following result directly follows from the proof of \cite[Theorem 17.4]{Cha_qbl}, which already proves the existence of such an $\mathfrak{n}$ but we need the exact expression of $\mathfrak{n}$

\begin{lemma}\label{lem:interchange_2}
    Let $\pi \in \mathrm{Irr}$ and $\mathfrak{m} \in \mathrm{Mult}$. Then for $\mathfrak{n}= \nu^{-1}\mathfrak{r}^\mathrm{L}\left(\mathfrak{m}, \mathfrak{hd}^\mathrm{L}(\mathrm{I}^\mathrm{L}_\mathfrak{m}(\pi))\right)$, we have
    \begin{equation}\label{eq:A1}
     {^-}\left(\mathrm{I}^\mathrm{L}_\mathfrak{n} \circ \mathrm{I}^\mathrm{L}_\mathfrak{m}(\pi) \right)\cong \pi \text{ and } \mathrm{lev}\left( \mathrm{I}^\mathrm{L}_\mathfrak{n} \circ \mathrm{I}^\mathrm{L}_\mathfrak{m}(\pi)\right) = \mathrm{lev}\left(\mathrm{I}^\mathrm{L}_\mathfrak{m}(\pi) \right).   
    \end{equation}
\end{lemma}
\begin{proof} We follow the same line of argument of \cite[Proof of Theorem 17.4]{Cha_qbl}. There exists $\mathfrak{p} \in \mathrm{Mult}$ such that $\mathrm{I}^\mathrm{L}_\mathfrak{m}(\pi) \cong Z(\mathfrak{p})$ and we denote $\tau=Z(\mathfrak{p})$. We define, $\tau_+ = Z\left(\sum\limits_{\Delta \in \mathfrak{p}} \Delta^+ \right)$. By \cite[Theorem 1.3]{Cha_tams}, there exists $\mathfrak{h} \in \mathrm{Mult}$ such that $\mathrm{D}^\mathrm{R}_\mathfrak{h}(\tau_+) \cong (\tau_+)^- \cong \tau$, where the last isomorphism follows from the algorithm of highest Bernstein-Zelevinsky derivative in Zelevinsky classification. Thus, we have $\tau_+ \cong \mathrm{I}^\mathrm{R}_\mathfrak{h}(\tau)$. By Lemma \ref{lem:interchange_1}, we have, 
\begin{equation}\label{eq:A2}
 \mathrm{D}^\mathrm{L}_{\nu \cdot \mathfrak{n}} \circ \mathrm{D}^\mathrm{L}_\mathfrak{m}(\mathrm{I}^\mathrm{R}_\mathfrak{h}(\tau)) \cong {^-}\left(\mathrm{I}^\mathrm{R}_\mathfrak{h}(\tau)\right) ~\text{ i.e., } ~\mathrm{D}^\mathrm{L}_{\nu \cdot \mathfrak{n}} \circ \mathrm{D}^\mathrm{L}_\mathfrak{m}(\tau_+) \cong {^-}(\tau_+),   
\end{equation}
for $\nu \cdot \mathfrak{n}= \mathfrak{r}^\mathrm{L}\left(\mathfrak{m}, \mathfrak{hd}^\mathrm{L}(\tau_+)\right)$. Claim 2. of \cite[The proof of Theorem 17.4]{Cha_qbl} follows from \eqref{eq:A2} and hence by \cite[The proof of Theorem 17.4]{Cha_qbl}, the equation \eqref{eq:A1} follows for $\mathfrak{n}=\nu^{-1} \cdot \mathfrak{r}^\mathrm{L}\left(\mathfrak{m}, \mathfrak{hd}^\mathrm{L}(\tau_+)\right)$. As $s(\Delta)=s(\Delta^+)$ for $\Delta \in \mathfrak{p}$, by Algorithm \ref{alg:hd:multisegment:zel},  we observe that
\begin{equation}\label{eq:A3}
 \mathfrak{hd}^\mathrm{L}(\tau) = \mathfrak{hd}^\mathrm{L}(\tau_+)   
\end{equation}
Hence, the result follows.
\end{proof}

\begin{proposition}[Interchange formula]\label{prop:interchange}
 Let $\pi, \pi^\prime \in \mathrm{Irr}$ and $(\pi, \pi^\prime)$ is a generalized GGP relevant with respect to Rd-minimal $\mathfrak{m}$ and Ld-minimal $\mathfrak{n}$. Then, $(\pi^\prime, \pi)$ is also a generalized GGP relevant with respect to Rd-minimal $\mathfrak{p}$ and Ld-minimal $\mathfrak{q}$, where
 \[\mathfrak{p}= \nu^\frac{1}{2} \mathfrak{r}^\mathrm{R}\left(\mathfrak{m}, \mathfrak{hd}^\mathrm{R}\left(\mathrm{I}^\mathrm{L}_{\mathfrak{n}}  \left( \nu^\frac{1}{2}\pi\right)\right)\right) = \nu^\frac{1}{2} \mathfrak{r}^\mathrm{R}\left(\mathfrak{m}, \mathfrak{hd}^\mathrm{R}\left(\mathrm{I}^\mathrm{R}_{\mathfrak{m}}  \left(\pi'\right)\right)\right)\] and 
\[\mathfrak{q}= \nu^{-\frac{1}{2}} \mathfrak{r}^\mathrm{L}\left(\mathfrak{n}, \mathfrak{hd}^\mathrm{L}\left(\mathrm{I}^\mathrm{L}_{\mathfrak{n}}  \left( \nu^\frac{1}{2}\pi\right)\right)\right)=\nu^{-\frac{1}{2}} \mathfrak{r}^\mathrm{L}\left(\mathfrak{n}, \mathfrak{hd}^\mathrm{L}\left(\mathrm{I}^\mathrm{R}_{\mathfrak{m}}  \left( \pi'\right)\right)\right).\]   
\end{proposition}
\begin{proof}
We now recall the proof of \cite[Theorem 18.1]{Cha_qbl}. It says that $(\pi',\pi)$ is a generalized GGP relevant pair under Rd-minimal $\nu^\frac{1}{2}\mathfrak{m}'$ and Ld-minimal $\nu^\frac{1}{2}\mathfrak{n}'$, where $\mathfrak{n}'$ is a multisegment such that
\[{^-}\left(\mathrm{I}^\mathrm{L}_{\mathfrak{n}'} \circ \mathrm{I}^\mathrm{L}_{\mathfrak{n}} \left( \nu^\frac{1}{2}\pi\right)  \right) \cong  \nu^\frac{1}{2}\pi,\] and $\mathfrak{m}'$ is a multisegment such that
\[\mathrm{D}^\mathrm{R}_{\mathfrak{m}'} \circ\mathrm{D}^\mathrm{R}_{\mathfrak{m}}\left( \mathrm{I}^\mathrm{L}_{\mathfrak{n}} \left( \nu^\frac{1}{2}\pi\right)  \right) \cong \left( \mathrm{I}^\mathrm{L}_{\mathfrak{n}}  \left( \nu^\frac{1}{2}\pi\right)\right)^-.\]
Therefore, by Lemma \ref{lem:interchange_1} and Lemma \ref{lem:interchange_2}, we conclude that $(\pi',\pi)$ is generalized GGP relevant pair for Rd-minimal $\mathfrak{p}$ and Ld-minimal $\mathfrak{q}$, where
\[\mathfrak{p}=\nu^\frac{1}{2}\mathfrak{m}' =\nu^\frac{1}{2} \mathfrak{r}^\mathrm{R}\left(\mathfrak{m}, \mathfrak{hd}^\mathrm{R}\left(\mathrm{I}^\mathrm{L}_{\mathfrak{n}}  \left( \nu^\frac{1}{2}\pi\right)\right)\right) \text{ and }\mathfrak{q}= \nu^\frac{1}{2}\mathfrak{n}' =\nu^{-\frac{1}{2}} \mathfrak{r}^\mathrm{L}\left(\mathfrak{n}, \mathfrak{hd}^\mathrm{L}\left(\mathrm{I}^\mathrm{L}_{\mathfrak{n}}  \left( \nu^\frac{1}{2}\pi\right)\right)\right).\]
\end{proof}

\subsection{An algorithm for general GGP relevant pairs} We now present the algorithm in full detail.

\begin{algorithm}\label{alg:relevant}
Suppose $\pi \in \mathrm{Irr}(\mathrm{GL}_n(F))$  and $\pi' \in \mathrm{Irr}(\mathrm{GL}_{n'}(F))$ are given in terms of their Langlands or Zelevinsky data. We wish to determine whether $(\pi, \pi')$ is a generalized GGP relevant pair. Set the initial pair $(\pi_1, \pi'_1)=(\pi, \pi')$.
\begin{itemize}
    \item {\bf Step 1 (Reduction from the right).} Let $\rho_1$ be a $\le$-maximal element in $\mathrm{csupp}(\pi_1)$ (i.e., a cuspidal representation with the largest possible value in the ordering). If $\rho_1 \notin \mathrm{csupp}(\nu^{-\frac{1}{2}}\pi_1')$, then define 
    \[\mathfrak{m}_1=\mathfrak m^e_{\rho_1}(\pi_1)=\sum\limits_{\Delta \in \mathfrak{hd}^\mathrm{R}(\pi_1)_{e=\rho_1}} \Delta,\] and set
    \[(\pi_2, \pi'_2)=\left(\mathrm{D}^\mathrm{R}_{\mathfrak m_1}(\pi_1), \pi_1'\right).\] Record that a reduction step has been applied, along with the multisegment $\mathfrak{m}_1$. Then proceed to Step 4.
    \item {\bf Step 2 (Reduction from the left).} If no such $\rho_1$ exists in Step 1, let $\rho_1'$ be a $\le$-minimal element in $\mathrm{csupp}(\pi_1')$ (i.e., a cuspidal representation with the smallest possible value in the ordering). If $\rho_1' \notin \mathrm{csupp}(\nu^{\frac{1}{2}}\pi_1)$, then define 
    \[\mathfrak{n}_1=\mathfrak m^s_{\rho_1'}(\pi_1')=\sum\limits_{\Delta \in \mathfrak{hd}^\mathrm{L}(\pi_1')_{s=\rho_1'}} \Delta,\] and set
    \[(\pi_2, \pi'_2)=\left(\pi_1, ~\mathrm{D}^\mathrm{L}_{\mathfrak n_1}(\pi'_1)\right).\] Record that a reduction step has been applied, along with the multisegment $\mathfrak{n}_1$. Then proceed to Step 4.
    \item {\bf Step 3 (Interchange).} If neither Step 1 nor Step 2 applies (i.e., every maximal cuspidal in $\pi_1$ lies in $\mathrm{csupp}(\nu^{-\frac{1}{2}}\pi_1')$ and every minimal cuspidal in $\pi_1'$ lies in $\mathrm{csupp}(\nu^{\frac{1}{2}}\pi_1)$, then interchange the pair:
    \[(\pi_2, \pi'_2)=(\pi'_1, \pi_1).\] Record that an interchange step has been applied. Then return to Step 1 with the new pair $(\pi_2, \pi'_2)$. (Note that after interchange, a reduction step will be possible because the cuspidal support conditions are now satisfied.)
    \item {\bf Step 4 (Repetition).} Set $k=2$ and continue the process. After a finite number of steps, we reach a pair $(\pi_r, \pi'_r)$ where both representations are generic. This termination is guaranteed because each reduction step strictly decreases the rank of the representation. Record the sequence of steps (reduction with recorded multisegments, or interchange) that led from $(\pi_1, \pi'_1)$ to $(\pi_r, \pi'_r)$. This completes the reduction process.
    \item {\bf Step 5 (Generic step).} Now $(\pi_r, \pi'_r)$ is a pair of generic representations. By Proposition \ref{prop:generic}, this pair is a generalized GGP relevant pair. Moreover, we can explicitly construct the Rd-minimal multisegment $\mathfrak p_r$ and Ld-minimal multisegment $\mathfrak q_r$ for which the generic pair $(\pi_r, \pi'_r)$ is a generalized GGP relevant, using the construction in Section \ref{sec:construction}:
    \[\mathfrak p_r= \mathfrak{p}_{\nu^{1/2}\mathfrak{w},\mathfrak{w}'} \text{ and }\mathfrak q_r= \mathfrak{q}_{\nu^{1/2}\mathfrak{w},\mathfrak{w}'},\]
    where $\pi_r=L(\mathfrak{w})$ and $ \pi'_r=L(\mathfrak{w}')$ in Langlands classification for the multisegments $\mathfrak{w} $ and $\mathfrak{w}'$.
    \item {\bf Step 6 (Backward propagation).} Now we work backwards through the recorded steps, from $i=r-1$ down to $i=1$, constructing Rd-minimal and Ld-minimal multisegments for each intermediate pair.

     {\bf Case 1. If the $i$-th step was a reduction step:}
        \begin{itemize}
            \item If it was a reduction from the right (using $\mathfrak{m}_i$ as recorded), then set
            \[\mathfrak{p}_i=\mathfrak{p}_{i+1} + \nu^{1/2}\mathfrak{m}_i \text{ and }\mathfrak{q}_{i}=\mathfrak{q}_{i+1}.\] Check the admissibility condition:
            \[\mathrm{D}^\mathrm{R}_{\mathfrak{p}_i}(\nu^{1/2}\pi_i) \ne 0.\] If the condition fails, terminate the algorithm — the original pair $(\pi,\pi')$ is not a generalized GGP relevant pair.
            \item If it was a reduction from the left (using $\mathfrak{n}_i$ as recorded), then set
            \[\mathfrak{p}_i=\mathfrak{p}_{i+1}  \text{ and }\mathfrak{q}_{i}=\mathfrak{q}_{i+1} + \mathfrak{n}_i.\] Check the admissibility condition:
            \[\mathrm{D}^\mathrm{L}_{\mathfrak{q}_i}(\pi_i') \ne 0.\] If the condition fails, terminate the algorithm — the original pair $(\pi,\pi')$ is not a generalized GGP relevant pair.
        \end{itemize}
        
     {\bf Case 2. If the $i$-th step was an interchange step:} Apply Proposition \ref{prop:interchange} to obtain $\mathfrak{p}_i$ and $\mathfrak{q}_i$ from $\mathfrak{p}_{i+1}$ and $\mathfrak{q}_{i+1}$. No admissibility check is required for interchange steps, as the interchange formula always yields a valid relevant pair.
    \item {\bf Step 7 (Conclusion).} If the algorithm successfully reaches $i=1$ without termination, then we have constructed multisegments $\mathfrak{p}_1$ and $\mathfrak{q}_1$ such that
    \[\mathrm{D}^\mathrm{R}_{\mathfrak{p}_1}(\nu^{1/2}\pi_1) \ne 0, ~ \mathrm{D}^\mathrm{R}_{\mathfrak{p}_1}(\nu^{1/2}\pi_1) \cong \mathrm{D}^\mathrm{L}_{\mathfrak{q}_1}(\pi_1'),\] and $(\mathfrak{p}_1,\mathfrak{q}_1, \nu^{1/2}\pi_1)$ is a strongly RdLi-commutative triple. Therefore, $(\pi,\pi^{\prime})$ is a generalized GGP relevant pair. 
\end{itemize}
\end{algorithm}

\begin{theorem}\label{thm:main} (Justification for Algorithm \ref{alg:relevant})
 Let $\pi$ and $\pi'$ be irreducible smooth representations. Then the pair $(\pi, \pi')$ is a generalized GGP relevant pair if and only if the admissibility conditions hold at every reduction step in Algorithm \ref{alg:relevant} (i.e., the algorithm terminates successfully at Step 7).
\end{theorem}
\begin{proof}
The forward direction follows from Propositions \ref{prop:left_reduction} and \ref{prop:right_reduction}: if $(\pi, \pi')$ is relevant, then each reduction step yields a relevant pair, and the admissibility conditions must hold by the properties of minimal multisegments. The backward direction follows from the admissibility conditions in Propositions \ref{prop:left_reduction} and \ref{prop:right_reduction}, which allow us to reconstruct the relevance from the reduced pair, together with Proposition \ref{prop:interchange} for interchange steps. The base case (generic pair) is handled by Proposition \ref{prop:generic}.
\end{proof}

\subsection{Some examples}\label{sec:example}
We now provide several examples to demonstrate the algorithm in practice. These examples illustrate the various steps of the algorithm: reduction from the right, reduction from the left, interchange, and the importance of admissibility checks.
\begin{example}[Relevant pair]
Let \[\pi=L\left(\left[\frac{1}{2},4\frac{1}{2}\right]+ \left[3\frac{1}{2},6\frac{1}{2}\right] \right) \text{ and } \pi'=L([0,3]+[3,6])\] be the irreducible smooth representations of $\mathrm{GL}_9(F)$ and  $\mathrm{GL}_8(F)$ respectively. We show that the pair $(\pi,\pi')$ is a generalized GGP relevant, and hence, $\mathrm{Hom_{GL_8(F)}}(\pi,\pi') \ne 0$.

Step 1 (Reduction from the right): We observe that $\nu^{6\frac{1}{2}}$ is $\le$-maximal element in $\mathrm{csupp}(\pi)$ and $\nu^{6\frac{1}{2}} \notin \mathrm{csupp}(\nu^{-1/2}\pi')$. Compute $\mathfrak{hd}^\mathrm{R}(\pi)=[\frac{1}{2}, 1\frac{1}{2}]+ [3\frac{1}{2}, 6\frac{1}{2}]$. Thus, we have $\mathfrak m_1=\mathfrak m^e_{\nu^{6\frac{1}{2}}}(\pi)=  [3\frac{1}{2}, 6\frac{1}{2}]$. Applying the right reduction process of Algorithm \ref{alg:relevant}, we set \[\pi_2= \mathrm{D}^\mathrm{R}_{[3\frac{1}{2}, 6\frac{1}{2}]}(\pi)= L \left( [\frac{1}{2}, 4\frac{1}{2}]\right) \text{ and }\pi_2'=\pi'.\] 

Step 2 (Reduction from the left): The $\le$-minimal cuspidal in $\mathrm{csupp}(\pi_2')$ is $\nu^{0}$ and $\nu^{0} \notin \mathrm{csupp}(\nu^{1/2}\pi_2)$. Compute $\mathfrak{hd}^\mathrm{L}(\pi_2')=[0,3]+[5,6] $. Thus, $\mathfrak n_2=\mathfrak m^s_{\nu^{0}}(\pi_2')=  [0,3]$. Applying the left reduction process of Algorithm \ref{alg:relevant}, we set \[\pi_3'= \mathrm{D}^\mathrm{L}_{[0,3]}(\pi_2')= L \left( [3,6]\right) \text{ and }\pi_3=\pi_2.\] 

Step 3 (Generic step): Now $\pi_3$ and $ \pi_3'$ are essentially square-integrable (hence generic). By Proposition \ref{prop:generic}, $(\pi_3, \pi_3')$ is relevant under Rd-minimal $\mathfrak{p}_3=[1,2]$ and Ld-minimal $\mathfrak{q}_3=[6,6]$. 

Step 4 (Backward propagation): Now, we need to check the admissibility conditions for both reduction processes. 
\begin{itemize}
    \item For the left reduction (Step 2): $\mathfrak{q}_2=\mathfrak{q}_3 +  \mathfrak n_2=[0,3]+[6,6]$ and  $\mathfrak{p}_2=\mathfrak{p}_3=[1,2]$. Check admissibility:
\[\mathrm{D}^\mathrm{L}_{\mathfrak{q}_2}(\pi'_2)=\mathrm{D}^\mathrm{L}_{[0,3]} \circ \mathrm{D}^\mathrm{L}_{[6]}(\pi_2')= \mathrm{D}^\mathrm{L}_{[0,3]}(L([0,3]+[3,5]))=L([3,5]) \neq 0,\]
    \item For the right reduction (Step 1): Here, $\mathfrak{p}=\mathfrak{p}_2 + \nu^{1/2}\mathfrak{m}_1=[1,2]+[4,7]$ and $\mathfrak{q}=\mathfrak{q}_2=[0,3]+[6,6]$. Check admissibility: \[\mathrm{D}^\mathrm{R}_\mathfrak{p}(\nu^\frac{1}{2} \pi)=\mathrm{D}^\mathrm{R}_{[4,7]} \circ \mathrm{D}^\mathrm{R}_{[1,2]}(\nu^\frac{1}{2} \pi)= \mathrm{D}^\mathrm{R}_{[4,7]}(L\left([3,5]+[4,7]\right))= L\left([3,5]\right) \neq 0,\]
\end{itemize}
Thus the algorithm terminates successfully, confirming that $(\pi,\pi')$ is generalized relevant pair under RdLd-minimal multisegments $(\mathfrak{p},\mathfrak{q})$.
\end{example}

\begin{example}[Relevant pair with interchange]
Let \[\pi=L\left(\left[-\frac{1}{2},2\frac{1}{2}\right]+ \left[2\frac{1}{2},5\frac{1}{2}\right] \right) \text{ and } \pi'=L([1,3]+[6,9])\] be the irreducible smooth representations of $\mathrm{GL}_8(F)$ and  $\mathrm{GL}_7(F)$ respectively. We show that $(\pi,\pi')$ is a generalized GGP relevant pair. 

Step 1 (Interchange): We observe that $\nu^{11/2}$ is $\le$-maximal element of $\mathrm{csupp}(\pi)$ with $\nu^{11/2} \in \mathrm{csupp}(\nu^{-1/2}\pi')$ and $\nu^{1}$ is $\le$-minimal element of $\mathrm{csupp}(\pi')$ with $\nu^{1} \in \mathrm{csupp}(\nu^{1/2}\pi)$. Both reduction conditions fail initially, so we interchange: $(\pi_2,\pi_2')=(\pi', \pi)$.

Step 2 (Reduction from the right on $\pi_2$): The $\le$-maximal element in $\mathrm{csupp}(\pi_2)$ is $\nu^{9}$, and $\nu^{9} \notin \mathrm{csupp}(\nu^{-1/2}\pi_2')$. Compute $\mathfrak{hd}^\mathrm{R}(\pi_2)=[1,3]+[6,9]$, and thus, $\mathfrak m_2= \mathfrak m^e_{\nu^9}(\pi_2)=  [6,9]$. Applying the reduction process of Algorithm \ref{alg:relevant}, we get \[\pi_3= \mathrm{D}^\mathrm{R}_{[6,9]}(\pi_2)= L \left( [1,3]\right), ~\pi_3'=\pi_2'.\] 

Step 2 (Reduction from the left on $\pi_3'$): We observe $\nu^{-1/2}$ is the $\le$-minimal element in $\mathrm{csupp}(\pi_3')$ and $\nu^{-1/2} \notin \mathrm{csupp}(\nu^{1/2}\pi_3)$. Compute $\mathfrak{hd}^\mathrm{L}(\pi_3')=\left[-\frac{1}{2},2\frac{1}{2}\right]+ \left[4\frac{1}{2},5\frac{1}{2}\right] $, and thus, we have $\mathfrak n_3=\mathfrak m^s_{\nu^{-1/2}}(\pi_3')= \left[-\frac{1}{2},2\frac{1}{2}\right]$. Applying the left reduction process of Algorithm \ref{alg:relevant}, we set \[\pi_4'= \mathrm{D}^\mathrm{L}_{\left[-\frac{1}{2},2\frac{1}{2}\right]}(\pi_3')= L \left( \left[2\frac{1}{2},5\frac{1}{2}\right]\right) \text{ and }\pi_4=\pi_3.\] 

Step 4 (Generic step): As $\pi_4$ and $ \pi_4'$ are essentially square integrable (hence generic) representations, the pair $(\pi_4,\pi_4')$ is a generalized GGP relevant. In fact, by  Proposition \ref{prop:generic}, $(\pi_4, \pi_4')$ is relevant under Rd-minimal $\mathfrak{p}_4=\left[1\frac{1}{2},1\frac{1}{2}\right]$ and Ld-minimal $\mathfrak{q}_4=\left[4\frac{1}{2},5\frac{1}{2}\right]$.

Step 5 (Backward propagation): 
\begin{itemize}
    \item For the left reduction (Step 3): we have $\mathfrak{q}_3=\mathfrak{q}_4 +  \mathfrak n_3=\left[4\frac{1}{2},5\frac{1}{2}\right]+\left[-\frac{1}{2},2\frac{1}{2}\right]$ and  $\mathfrak{p}_3=\mathfrak{p}_4$. Check admissibility: (using Algorithm \ref{alg:der:Lang})
\[\mathrm{D}^\mathrm{L}_{\mathfrak{q}_3}(\pi'_3)=\mathrm{D}^\mathrm{L}_{\left[-\frac{1}{2},2\frac{1}{2}\right]} \circ \mathrm{D}^\mathrm{L}_{\left[4\frac{1}{2},5\frac{1}{2}\right]}(\pi_2')=L\left(\left[2\frac{1}{2},3\frac{1}{2}\right]\right) \neq 0.\]
    \item For the right reduction (Step 2): we have $\mathfrak{p}_2=\mathfrak{p}_3 + \nu^{1/2}\mathfrak m_2=\left[1\frac{1}{2}\right]+\left[6\frac{1}{2},9\frac{1}{2}\right]$ and $\mathfrak{q}_2=\mathfrak{q}_3$. Check admissibility: (using Algorithm \ref{alg:der:Lang}) \[\mathrm{D}^\mathrm{R}_\mathfrak{p_2}(\nu^\frac{1}{2} \pi_2)=\mathrm{D}^\mathrm{R}_{\left[6\frac{1}{2},9\frac{1}{2}\right]} \circ \mathrm{D}^\mathrm{R}_{\left[1\frac{1}{2},1\frac{1}{2}\right]}(\nu^\frac{1}{2} \pi_2)= L\left(\left[2\frac{1}{2},3\frac{1}{2}\right]\right) \neq 0,\]
    \item Interchange (Step 1): Proposition \ref{prop:interchange} yields $(\mathfrak{p}, \mathfrak{q})$ for the original pair $(\pi,\pi')$ using $(\mathfrak{p}_2, \mathfrak{q}_2)$.
\end{itemize}
Therefore, the pair $(\pi,\pi')=(\pi_2', \pi_2)$ is a generalized relevant pair under multisegment $(\mathfrak{p}, \mathfrak{q})$.
 
\end{example}

\begin{example}[Non-relevant pair]
Let \[\pi=L\left(\left[\frac{1}{2},4\frac{1}{2}\right]+ \left[3\frac{1}{2},6\frac{1}{2}\right] + \left[5\frac{1}{2},6\frac{1}{2}\right] \right) \text{ and } \pi'=L([0,3]+[1,2]+[3,6])\] be the irreducible smooth representations of $\mathrm{GL}_{11}(F)$ and  $\mathrm{GL}_{10}(F)$ respectively. We show that $(\pi,\pi')$ is not a generalized GGP relevant pair (hence $\mathrm{Hom_{GL_{10}(F)}}(\pi,\pi') = 0$).   

Step 1 (Reduction from the right on $\pi$): The $\le$-maximal element of $\mathrm{csupp}(\pi)$ is $\nu^{13/2}$, and $\nu^{13/2} \notin \mathrm{csupp}(\nu^{-1/2}\pi')$.  Compute \[\mathfrak{hd}^\mathrm{R}(\pi)=\left[\frac{1}{2},1\frac{1}{2}\right]+ \left[3\frac{1}{2},6\frac{1}{2}\right] + \left[5\frac{1}{2},6\frac{1}{2}\right]\]
Thus, $\mathfrak m_1=\mathfrak m^e_{\nu^{13/2}}(\pi)=  \left[3\frac{1}{2},6\frac{1}{2}\right] + \left[5\frac{1}{2},6\frac{1}{2}\right]$. Applying the right reduction process of Algorithm \ref{alg:relevant}, set \[\pi_2= \mathrm{D}^\mathrm{R}_{\mathfrak m_1}(\pi)= L \left( \left[\frac{1}{2},4\frac{1}{2}\right]\right)\text{ and } \pi_2'=\pi'.\]

Step 2 (Reduction from the left on $\pi_2'$): The $\le$-minimal element in $\mathrm{csupp}(\pi_2')$ is $\nu^{0}$ and $\nu^{0} \notin \mathrm{csupp}(\nu^{1/2}\pi_2)$. Compute $\mathfrak{hd}^\mathrm{L}(\pi_2')=[0,2]+[1]+[3,6]$ and thus, $\mathfrak{n}_2=\mathfrak m^s_{\nu^{0}}(\pi_2')=  [0,3]$. The left reduction yields \[\pi_3'= \mathrm{D}^\mathrm{L}_{[0,3]}(\pi_2')= L \left( [1,2]+[3,6]\right) \text{ and }\pi_3=\pi_2.\]

Step 3 (Interchange): Further, both reduction conditions fail now, so we interchange: $(\pi_4,\pi_4' )=(\pi_3',\pi_3)$.

Step 4 (Reduction from the right on $\pi_4$): The maximal cuspidal $\nu^{6} \notin \mathrm{csupp}(\nu^{-1/2}\pi_4')$. Since $\mathfrak{hd}^\mathrm{R}(\pi)=[1]+ [3,6]$, we have $\mathfrak{m}_4=\mathfrak m^e_{\nu^{6}}(\pi_4)=  \left[3,6\right]$. The right reduction of Algorithm \ref{alg:relevant} yields \[\pi_5= \mathrm{D}^\mathrm{R}_{[3,6]}(\pi_4)= L \left( \left[1,2\right]\right) \text{ and } \pi_5'=\pi_4'=L \left( \left[\frac{1}{2},4\frac{1}{2}\right]\right).\]

Step 5 (Generic step): By  Proposition \ref{prop:generic}, the generic pair $(\pi_5, \pi_5')$ is generalized GGP relevant under Rd-minimal $\mathfrak{p}_5=\left[1\frac{1}{2},2\frac{1}{2}\right]$ and Ld-minimal $\mathfrak{q}_5=\left[\frac{1}{2},4\frac{1}{2}\right]$.

Step 6  {\bf Admissibility failure}: When propagating backward, we have $\mathfrak{p}_4=\mathfrak{p}_5+ \nu^{1/2}\mathfrak{m}_4=\left[1\frac{1}{2},2\frac{1}{2}\right]+\left[3\frac{1}{2},6\frac{1}{2}\right]$. By Algorithm \ref{alg:der:Lang}, we have $\mathrm{D}^\mathrm{R}_{\left[1\frac{1}{2},2\frac{1}{2}\right]}(\nu^{1/2}\pi_4)=0$ and hence,
\[ \mathrm{D}^\mathrm{R}_{\mathfrak{p}_4 }\left(\nu^{1/2}\pi_4\right)   =\mathrm{D}^\mathrm{R}_{\left[3\frac{1}{2},6\frac{1}{2}\right]}   \circ \mathrm{D}^\mathrm{R}_{\left[1\frac{1}{2},2\frac{1}{2}\right]}\left(\nu^{1/2}\pi_4\right)=0.\]
Therefore, the admissibility condition fails, and the algorithm terminates with the conclusion that $(\pi,\pi')$ is not a generalized GGP relevant pair.
\end{example}

\begin{example}[Non-relevant pair]
Let \[\pi=L\left(\left[-\frac{1}{2},2\frac{1}{2}\right]+\left[2\frac{1}{2},5\frac{1}{2}\right]+\left[4\frac{1}{2},5\frac{1}{2}\right] \right) \text{ and }\pi'=L([0,1]+[1,4]+[7,9]).\] One can similarly verify (by carrying out the algorithm) that this pair is not a generalized GGP relevant. We leave the details to the interested reader.
\end{example}

\section{Derivatives and quasi-Speh representations} 
For a unitary cuspidal representation $\rho$, and positive integers $u,v$, the unitary Speh multisegment is \[\mathfrak{m}_\rho(u,v)=\sum\limits_{j= -(v-1)/2 }^{(v-1)/2}\nu^j \Delta_\rho(u), \text{ where }\Delta_\rho(u)=\left[-\frac{u-1}{2},\frac{u-1}{2} \right]_\rho.\]
Further, we denote $\mathfrak{m}_\rho(u,0)=\emptyset$ and $\mathfrak{m}_\rho(u,-1)=\emptyset$. Following \cite{Gur}*{Section 4.1}, for a non-negative integer $w \le u$, we define the right quasi-Speh multisegment 
\[\mathfrak{m}_\rho(u,v,R^w)=\sum\limits_{j={-\frac{v-1}{2}}}^{\frac{v-1}{2}-1} \nu^j\Delta_\rho(u) + \nu^\frac{v-1}{2}\left[-\frac{u-1}{2}+w, \frac{u-1}{2}\right]_\rho,\] and the left quasi-Speh multisegment
\[\mathfrak{m}_\rho(u,v,L_w)=\nu^{-\frac{v-1}{2}}\left[-\frac{u-1}{2}, \frac{u-1}{2}-w\right]_\rho + \sum\limits_{j={-\frac{v-1}{2}}+1}^{\frac{v-1}{2}} \nu^j\Delta_\rho(u) .\]
The corresponding irreducible representations via the Langlands classification are denoted as \[\pi_\rho(u,v,R^w)=L(\mathfrak{m}_\rho(u,v,R^w)) \text{ and } \pi_\rho(u,v,L_w)=L(\mathfrak{m}_\rho(u,v,L_w)).\]
When $w =0$, we have $\pi_\rho(u,v,R^0)=\pi_\rho(u,v)=\pi_\rho(u,v,L_0)$ and when $w =u$, we have $\pi_\rho(u,v,R^u)=\nu^{-\frac{1}{2}}\pi_\rho(u,v-1)$ and $\pi_\rho(u,v,L_u)=\nu^{\frac{1}{2}}\pi_\rho(u,v-1)$. For convenience, we set
\[ \mathfrak{m}^-_\rho(u,v)=\mathfrak{m}_\rho(u,v,R^u) \text{ and } {^-}\mathfrak{m}_\rho(u,v)=\mathfrak{m}_\rho(u,v,L_u).\]
For any $-\frac{1}{2}<\alpha <\frac{1}{2}$, we also consider the shifted quasi-Speh multisegments $\mathfrak{m}_\rho(u,v,R^w)\nu^\alpha$ and $\mathfrak{m}_\rho(u,v,L_w)\nu^\alpha$, and the corresponding representations
\[\pi_\rho(u,v,R^w)\nu^\alpha=L\left(\mathfrak{m}_\rho(u,v,R^w)\nu^\alpha\right) \text{ and } \pi_\rho(u,v,L_w)\nu^\alpha=L\left(\mathfrak{m}_\rho(u,v,L_w)\nu^\alpha\right).\]
Finally, we introduce a pre-order on the set of right quasi-Speh representations (and analogously for left ones):
\[\pi_\rho(u_1,v_1,R^{w_1})\nu^{\alpha_1} \preceq \pi_\rho(u_2,v_2,R^{w_2})\nu^{\alpha_2},\]
if $u_1+v_1+\alpha_1 < u_2+v_2+\alpha_2$ or if $u_1+v_1+\alpha_1 = u_2+v_2+\alpha_2$ and $u_1 \le u_2$. 

Explicit combinatorial descriptions of the Bernstein–Zelevinsky derivatives of Speh representations and their complementary series have been established by Lapid and Mínguez \cites{LM14, LM16} (see also \cites{BZ2, Zel} for the general theory). In this section, we recall the properties of St-derivatives that will be used throughout the paper. We begin with some basic facts about derivatives of irreducible unitary representations.

\begin{lemma}\label{lem:der_Speh}
Consider a unitary cuspidal $\rho$, positive integers $u,v$ and real number $-\frac{1}{2}<\alpha <\frac{1}{2}$. For a segment $\Delta$, we have
 \[\mathrm{D}^\mathrm{R}_\Delta \left(\pi_\rho(u,v) \nu^\alpha \right)=\begin{cases}
      \pi_\rho(u,v,R^w)  \nu^\alpha &\mbox{ if } \Delta=\nu^{\frac{v-1}{2}+\alpha}\left[-\frac{u-1}{2}, -\frac{u-1}{2}+w-1\right]_\rho \text{ for some } 0 \le w \le u\\
      0 &\mbox{ otherwise},
 \end{cases}\] 
 and
  \[\mathrm{D}^\mathrm{L}_\Delta \left(\pi_\rho(u,v) \nu^\alpha \right)=\begin{cases}
      \pi_\rho(u,v,L_w)  \nu^\alpha &\mbox{ if } \Delta=\nu^{-\frac{v-1}{2}+\alpha}\left[\frac{u-1}{2}-w+1, \frac{u-1}{2}\right]_\rho \text{ for some } 0 \le w \le u\\
      0 &\mbox{ otherwise}.
 \end{cases}.\] 
\end{lemma}
\begin{proof}
    The proof immediately follows Algorithm \ref{alg:der:Lang}. 
\end{proof}

\begin{lemma}\label{lem:der_complementary}
 Consider a unitary cuspidal representation $\rho$, positive integers $u,v$ and a real number $0<\alpha <\frac{1}{2}$. Then, for a multisegment $\mathfrak{p}$, we have 
  \[\mathrm{D}^\mathrm{R}_\mathfrak{p} \left(\pi_\rho(u,v)(\alpha) \right)=\begin{cases}
      L \left(\mathfrak{m}_\rho(u,v,R^{w_1}) \nu^{-\alpha}+\mathfrak{m}_\rho(u,v,R^{w_2}) \nu^{\alpha} \right) &\mbox{ for some integers } 0 \le w_1,w_2\le u\\
      0 &\mbox{ otherwise}.
 \end{cases}\]  
\end{lemma}
\begin{proof}
The two multisegments $\mathfrak{m}_\rho(u,v)\nu^\alpha$ and $\mathfrak{m}_\rho(u,v)\nu^{-\alpha}$ lies in distinct cuspidal lines (because $\alpha \ne -\alpha$ modulo $\mathbb{Z}$). Therefore, by the basic properties of derivatives, the right derivative with respect to any segment $\Delta$ either vanishes or factors through one of the two components. Then, Lemma \ref{lem:der_Speh} yields the stated result. 
\end{proof}

\begin{proposition}\label{prop:mult_der}
Let $\pi=\pi_\rho(u_1, v_1)\nu^{\alpha_1} \times ... \times \pi_\rho(u_r, v_r)\nu^{\alpha_r}$ be an irreducible unitary representation with $\rho$ unitary cuspidal, $-\frac{1}{2}<\alpha_i <\frac{1}{2}$, and $u_i,v_i$ positive integers. Suppose $\mathfrak{p} \in \mathrm{Mult}_\rho$ is such that  $\mathrm{D}^\mathrm{R}_{\mathfrak{p}}(\pi) \ne 0$. Then there exist integers $0 \le w_i \le u_i$ such that
\[\mathrm{D}^\mathrm{R}_{\mathfrak{p}}(\pi) = L \left(\mathfrak{m}_\rho(u_1,v_1,R^{w_1}) \nu^{\alpha_1}+...+\mathfrak{m}_\rho(u_r,v_r,R^{w_r}) \nu^{\alpha_r} \right).\] 
An analogous statement holds for left derivatives $\mathrm{D}^\mathrm{L}_{\mathfrak{p}}(\pi)$ with $R^{w_i}$ replaced by $L_{w_i}$.
\end{proposition}

\begin{proof}
Let $\ell$ be the length of the multisegment $\mathfrak{p}$. By \cites{Cha_tams, Cha_qbl}, $\mathrm{D}^\mathrm{R}_{\mathfrak{p}}(\pi)$  is a simple quotient of the Bernstein-Zelevinsky derivative $\pi^{(\ell)}$. According to \cite{BZ2}*{Corollary 4.14}, $\pi^{(\ell)}$ admits a filtration whose successive quotients are representations of the form
\begin{align*}
  \left(\pi_\rho(u_1, v_1)\nu^{\alpha_1}\right)^{(\ell_1)} \times ... \times \left(\pi_\rho(u_r, v_r)\nu^{\alpha_r}\right)^{(\ell_r)},
\end{align*}
 for some non-negative integers $\ell_j$ such that $\ell=\ell_1+...+\ell_r$. Therefore, $\mathrm{D}^\mathrm{R}_{\mathfrak{p}}(\pi)$ is a simple quotient of one such product $\pi_\rho(u_1, v_1)^{(\ell_1)}\nu^{\alpha_1} \times ... \times \pi_\rho(u_r, v_r)^{(\ell_r)}\nu^{\alpha_r}$. For each $j$, the iterated derivative $\left(\pi_\rho(u_j, v_j)\nu^{\alpha_j}\right)^{(\ell_j)}$ is either zero or an irreducible representation. By the combinatorial description of derivatives of Speh representations established in \cite{LM14} (see also Lemma \ref{lem:der_Speh}), any such nonzero derivative is isomorphic to a Langlands quotient of the form
\[L \left( \mathfrak{m}_\rho(u_j,v_j,R^{w_j}) \nu^{\alpha_j}  \right) \text{ for some unique integer } 0 \le  w_j \le u_j.\]
Because $\pi$ is irreducible, the factors $\pi_\rho(u_i,v_i)\nu^{\alpha_i}$ can be assumed to be ordered so that
\[\pi_\rho(u_i,v_i)\nu^{\alpha_i} \preceq \pi_\rho(u_j,v_j)\nu^{\alpha_j} \text{ for }i<j.\]
The pre-order on quasi-Speh representations depends only on $u,v, \alpha$ and not on $w$. Therefore, the same ordering holds for the factors $L \left( \mathfrak{m}_\rho(u_i,v_i,R^{w_i}) \nu^{\alpha_i}  \right)$ i.e., we have 
 \[\pi_\rho(u_i,v_i,R^{w_i})\nu^{\alpha_i} \preceq \pi_\rho(u_j,v_j,R^{w_j})\nu^{\alpha_j}  \text{ for i < j}.\]
Consequently, the parabolic induced representation $L \left( \mathfrak{m}_\rho(u_1,v_1,R^{w_1}) \nu^{\alpha_1}  \right) \times ... \times L \left( \mathfrak{m}_\rho(u_r,v_r,R^{w_r}) \nu^{\alpha_r}  \right)$ satisfies the conditions of \cite{Gur}*{Proposition 4.3}. Therefore, $L \left( \mathfrak{m}_\rho(u_1,v_1,R^{w_1}) \nu^{\alpha_1}  \right) \times ... \times L \left( \mathfrak{m}_\rho(u_r,v_r,R^{w_r}) \nu^{\alpha_r}  \right)$ has the unique irreducible quotient of the form
\[L \left(\mathfrak{m}_\rho(u_1,v_1,R^{w_1}) \nu^{\alpha_1}+...+\mathfrak{m}_\rho(u_r,v_r,R^{w_r}) \nu^{\alpha_r} \right).\]
Since $\mathrm{D}^\mathrm{R}_{\mathfrak{p}}(\pi)$ is a simple quotient of this induced representation, we conclude that \[\mathrm{D}^\mathrm{R}_{\mathfrak{p}}(\pi) = L \left(\mathfrak{m}_\rho(u_1,v_1,R^{w_1}) \nu^{\alpha_1}+...+\mathfrak{m}_\rho(u_r,v_r,R^{w_r}) \nu^{\alpha_r} \right).\] 
\end{proof}

\begin{lemma}\label{lem:req}
Let $[a,b]_\rho=\left[-\frac{u-1}{2} + \beta + \frac{v-1}{2}, \frac{u-1}{2}+ \beta + \frac{v-1}{2} \right]_\rho$ and $\Delta_j^\mathrm{bot}$ be the bottom segment of the ladder multisegment $\mathfrak{n}=\mathfrak{m}_\rho (u_j,v_j)\nu^{\beta_j}$ (resp. $\mathfrak{m}^-_\rho(u_j,v_j)\nu^{\beta_j}$) for some positive integers $u, z, u_j, v_j$ and real numbers $-\frac{1}{2}<\beta, \beta_j <\frac{1}{2}$. If $[x,y] \in \mathfrak{n}$ with $[a,b]_\rho \prec [x,y]_\rho$, then $\Delta_j^\mathrm{bot}$ starts with at most $\nu^a \rho$ i.e.,
\[s_\rho \left( \Delta_j^\mathrm{bot}  \right) \le a.\]
\end{lemma}
\begin{proof}
If possible, let $s_\rho \left( \Delta_j^\mathrm{bot} \right) > a$. As $s_\rho \left( \Delta_j^\mathrm{bot}  \right) - a \in \mathbb{Z}$, we have $s_\rho \left( \Delta_j^\mathrm{bot}  \right) \ge a+1$ and so,
\begin{align}\label{eq:LA1}
s_\rho \left( \Delta_j^\mathrm{bot} \right)&= -\frac{u_j-1}{2} + \beta_j - \frac{v_j-1}{2} \ge  -\frac{u-1}{2} + \beta + \frac{v-1}{2} +1 \nonumber \\
 \implies & -\frac{u_j-1}{2}  - \frac{v_j-1}{2}  >   -\frac{u-1}{2}  + \frac{v-1}{2} \quad (\text{since } 1 + \beta -\beta_j >0)
\end{align}
As $y \ge b+1$ and $y = \frac{u_j-1}{2} + \beta_j - \frac{v_j-1}{2}+t$ for a non-negative integer $t$ with $t- \frac{v_j-1}{2} \le \frac{v_j-1}{2}$, we have
\begin{align*}
 &\frac{u_j-1}{2} + \beta_j - \frac{v_j-1}{2} + t \ge   \frac{u-1}{2} + \beta + \frac{v-1}{2} +  1 \\
 \implies & \frac{u_j-1}{2}  - \frac{v_j-1}{2} + t >   \frac{u-1}{2} - \frac{v-1}{2}   \quad (\text{since } 1 + \beta -\beta_j >0 \text{ and } z \ge 1)\\
 \implies &  -\frac{u_j-1}{2}  + \frac{v_j-1}{2} - t <   -\frac{u-1}{2} + \frac{v-1}{2}
\end{align*}
Since $- \frac{v_j-1}{2}+t \le \frac{v_j-1}{2}$, we have
\[-\frac{u_j-1}{2}  - \frac{v_j-1}{2} \le -\frac{u_j-1}{2}  + \frac{v_j-1}{2} - t < -\frac{u-1}{2} + \frac{v-1}{2},\]
this contradicts \eqref{eq:LA1}. Hence our assumption 
$s_\rho \left( \Delta_j^\mathrm{bot} \right) > a$ is false and we must have $s_\rho \left( \Delta_j^\mathrm{bot} \right) \le a$.
\end{proof}

\begin{proposition}\label{prop:der_unitary_mult}
  Let $\mathfrak{m}=\mathfrak{m}_\rho(u_1,v_1)\nu^{\beta_1}+ ...+ \mathfrak{m}_\rho(u_r,v_r)\nu^{\beta_r} +  \mathfrak{m}^-_\rho (u_{r+1},v_{r+1})\nu^{\beta_{r+1}}+ ...+ \mathfrak{m}^-_\rho (u_s,v_s)\nu^{\beta_s}$ for some unitary Speh multisegments $\mathfrak{m}_\rho(u_i,v_i)$ and real numbers $-\frac{1}{2}<\beta_i <\frac{1}{2}$. Suppose $\Delta =\Delta_1^{\mathrm{top}}$ be the top segment of the ladder multisegment $\mathfrak{m}_\rho(u_1,v_1)\nu^{\beta_1}$. Then, $\mathcal{D}^\mathrm{R}_\Delta \left(\mathfrak{m}\right)= \mathfrak{m}-\Delta$ i.e.,
  \[\mathcal{D}^\mathrm{R}_\Delta \left(\mathfrak{m}\right)=\mathfrak{m}^-_\rho(u_1,v_1)\nu^{\beta_1}+  \mathfrak{m}_\rho(u_2,v_2)\nu^{\beta_2}+ ...+ \mathfrak{m}_\rho(u_r,v_r)\nu^{\beta_r} +  \mathfrak{m}^-_\rho(u_{r+1},v_{r+1})\nu^{\beta_{r+1}}+ ...+ \mathfrak{m}^-_\rho(u_s,v_s)\nu^{\beta_s}.\] 
\end{proposition}

\begin{proof}
Assume all the notations of Algorithm \ref{alg:der:Lang}. Here $[a,b]_\rho=\Delta$. Clearly, $\mathfrak{m}_{[a,b]} \cap \mathfrak{m} (u_1,v_1)\nu^{\beta_1} = \{\Delta \}$. Suppose, $i_*$ is the highest positive integer $i \le k$ such that $[a,b]_\rho$ lies in the upward sequence $\mathcal{U}(\mathfrak{m}_{i})$. By Lemma \ref{lem:req} and the definition of upward sequences, we have the following properties:
\begin{itemize}
    \item[(P1)] If some segment $\Delta_j \in \mathcal{U}(\mathfrak{m}_{i})$ for $1 \le i \le k$ and $\overset{\rightarrow}{\Delta_j} \in \mathfrak{m}_{1} - \sum\limits_{t=1}^{i-1}\mathcal{U}(\mathfrak{m}_{t})$, then  $\overset{\rightarrow}{\Delta_j} \in \mathcal{U}(\mathfrak{m}_{i})$.
    \item[(P2)] If some segment $[x,y]_\rho \in \mathfrak{m}_{1}$ where $\ell_{rel}([x,y]_\rho) \ge \ell_{rel}([a,b]_\rho)$ and $[x,y]_\rho \in \mathfrak{m}_\rho(u_j,v_j)\nu^{\beta_j}$ (resp. $[x,y]_\rho \in \mathfrak{m}^-_\rho(u_j,v_j)\nu^{\beta_j}$) for some $j \ne 1$, then  $[x,y]_\rho \in \mathcal{U}(\mathfrak{m}_{i})$ with $i < i_*$.
    \item[(P3)] For $i \ge i_*$, $ e_\rho\left(\Delta_{i,1} \right)=b  \text{ and if $\Delta_{i,2}$ exists, } e_\rho\left(\Delta_{i,2} \right)=b+1$. This follows from (P1).
\end{itemize}
We describe the structure of upward sequences for $\mathcal{U}(\mathfrak{m}_{i})$ for $i \ge i_*$. We set $i_*=i_1$ and $a=a_1$. We consider two cases.

Case 1. $r_{i_1}=1$. Then, $\mathcal{U}(\mathfrak{m}_{i_1})=[a_1, b]_\rho$. By (P3), for all $i > i_1$, $\mathcal{U}(\mathfrak{m}_{i})=\Delta_{i,1}=[a'_i, b]_\rho$ with $ a'_i > a_1$. Then, by Lemma \ref{lem:rmk}, we will have
\[\mathcal{D}^\mathrm{R}_{[a_1,b]_\rho}(\mathfrak{m})=\mathfrak{m}-[a_1,b]_\rho=\mathfrak{m}-\Delta.\]

Case 2. $r_{i_1} > 1$. By (P2) and (P3), we have $\mathcal{U}(\mathfrak{m}_{i_1})=[a_1, b]_\rho + [a_2+1, b+1]_\rho + ...\text{ for some }a_2>a_1$. By Lemma \ref{lem:req} and (P1), $\overset{\leftarrow}{\Delta}_{i_1,2}=[a_2,b] \in \mathfrak{m}_{i_1 +1}$. 

Set $i_2$ for the largest index such that $[a_2,b]_\rho \in \mathcal{U}(\mathfrak{m}_{i_2})$. Similar to (P2), if $[x,y]_\rho \in \mathfrak{m}_{i_1 +1}$ where $\ell_{rel}([x,y]_\rho) \ge \ell_{rel}([a_2,b]_\rho)$ and $[x,y]_\rho \in \mathfrak{m} (u_j,v_j)\nu^{\beta_j}$ (resp. $[x,y]_\rho \in \mathfrak{m}^-(u_j,v_j)\nu^{\beta_j}$) with $[x,y]_\rho \ne [a_2,b]$, then  $[x,y]_\rho \in \mathcal{U}(\mathfrak{m}_{i})$ with $i < i_2$. Now by (P1) and (P3), we have  
\begin{itemize}
    \item for $i_1 < i <i_2$, $\mathcal{U}(\mathfrak{m}_{i})=[x_i,b]_\rho + [x_i',b+1]_\rho+...$ with $ a_1<x_i \le a_2<x_i'$,
    \item for $i=i_2$, $\mathcal{U}(\mathfrak{m}_{i_2})=[a_2,b]_\rho+....$
\end{itemize}
If $r_{i_2}=1$, we stop; otherwise continue inductively, producing indices 
\[i_1<i_2<...<i_d \le k \text{ with } a=a_1<a_2<...<a_d,\] such that:
\begin{itemize}
    \item For $i=i_t$,  $\mathcal{U}(\mathfrak{m}_{i_t})=[a_t, b]_\rho + [a_{t+1}+1, b+1]_\rho + ...$;
    \item For $i_t < i <i_{t+1}$, $\mathcal{U}(\mathfrak{m}_{i})=[x_i,b]_\rho + [x_i',b+1]_\rho+...$ with $ a_t<x_i \le a_{t+1}<x_i'$;
    \item For $i=i_d$, $r_{i_d}=1$ and $\mathcal{U}(\mathfrak{m}_{i_d})=[a_d,b]_\rho$;
    \item For $i>i_d$, $\mathcal{U}(\mathfrak{m}_{i})=\Delta_{i,1}=[a', b]_\rho$ with $a' \ge a_d$.
\end{itemize}
As $[a_t, a_{t+1}-1]_\rho \subset \mathfrak{rf}\left( \Delta_{i_t,1} \right)$ for $1 \le t < d$ and $\Delta_{i,1}= \mathfrak{rf}\left( \Delta_{i,1} \right)$ for $i \ge i_d$, we conclude that
\begin{align*}
\mathcal{D}^\mathrm{R}_{[a,b]_\rho}(\mathfrak{m}) \ne \infty   \text{ and } \mathcal{D}^\mathrm{R}_{[a,b]_\rho}(\mathfrak{m})=\mathfrak{m}-  \sum\limits_{i=i_1}^k \mathcal{U}(\mathfrak{m}_{i}) + \mathcal{D}^\mathrm{R}_{[a,b]} \left( \sum\limits_{i=i_1}^k \mathcal{U}(\mathfrak{m}_{i}) \right),
\end{align*}
where
\begin{align*}
\mathcal{D}^\mathrm{R}_{[a,b]_\rho} \left( \sum\limits_{i=i_1}^k \mathcal{U}(\mathfrak{m}_{i}) \right)&=  \sum\limits_{t=1}^{d-1} \mathcal{D}^\mathrm{R}_{[a_t,a_{t+1}-1]_\rho} \left( \sum\limits_{i=i_t}^{i_{t+1}-1} \mathcal{U}(\mathfrak{m}_{i}) \right) + \mathcal{D}^\mathrm{R}_{[a_d,b]_\rho} \left( \sum\limits_{i=i_d}^{k} \mathcal{U}(\mathfrak{m}_{i}) \right)\\
&= \sum\limits_{t=1}^{d-1}  \left( \sum\limits_{i=i_t}^{i_{t+1}-1} \mathcal{U}(\mathfrak{m}_{i}) -[a_t,b]_\rho+[a_{t+1},b]_\rho  \right) + \left( \sum\limits_{i=i_d}^{k} \mathcal{U}(\mathfrak{m}_{i})-[a_d, b]_\rho \right)\\
&= \sum\limits_{i=i_1}^k \mathcal{U}(\mathfrak{m}_{i}) - [a_1,b]_\rho = \sum\limits_{i=i_1}^k \mathcal{U}(\mathfrak{m}_{i}) - \Delta
\end{align*}
Here, the above first equality follows from $s\left(\Delta_{i,1}\right) > a_t $ for all $i > i_t$ and the second equality by Lemma \ref{lem:rmk}. Therefore, we have
\[\mathcal{D}^\mathrm{R}_{[a,b]_\rho}(\mathfrak{m})=\mathfrak{m}-\Delta.\]
\end{proof}

\begin{proposition}\label{prop:der_unitary_mult_left}
  Let $\mathfrak{m}=\mathfrak{m}_\rho(u_1,v_1)\nu^{\alpha_1}+ ...+ \mathfrak{m}_\rho(u_r,v_r)\nu^{\alpha_r} +  {^-}\mathfrak{m}_\rho(u_{r+1},v_{r+1})\nu^{\alpha_{r+1}}+ ...+ {^-}\mathfrak{m}_\rho(u_s,v_s)\nu^{\alpha_s}$ for some unitary Speh multisegments $\mathfrak{m}_\rho(u_i,v_i)$ and real numbers $-\frac{1}{2}<\alpha_i <\frac{1}{2}$. Suppose $\Delta =\Delta_j^{\mathrm{bot}}$ be the bottom segment of the ladder multisegment $\mathfrak{m}_\rho(u_j,v_j)\nu^{\alpha_j}$ for some $1 \le j \le r$. Then, \[\mathcal{D}^\mathrm{L}_\Delta \left(\mathfrak{m}\right)= \mathfrak{m}-\Delta.\]
\end{proposition}
\begin{proof}
 This is the left version of Proposition \ref{prop:der_unitary_mult}.
\end{proof}

Fix $\rho \in \mathrm{Irr^{unit}}$, an irreducible cuspidal representation. We want to define an involution on $\mathrm{Mult}_\rho$ as follows:
For a segment $\Delta=\nu^\alpha[a,b]_\rho=[a+\alpha,b+\alpha]_\rho \in \mathrm{Seg}_\rho$, we first define $\Delta^\dagger=\nu^{-\alpha}[-b, -a]_\rho=[-b-\alpha, -a-\alpha]_\rho.$ Then, we extend this to an involution on $\mathrm{Mult}_\rho$ by 
\[ \mathfrak{m} \mapsto \mathfrak{m}^\dagger,  \text{ where }   \mathfrak{m}^\dagger := \sum\limits_{\Delta \in \mathfrak{m}} \Delta^\dagger  \text{ for all } \mathfrak{m} \in \mathrm{Mult}_\rho.\]
For each $\mathfrak{m} \in \mathrm{Mult}_\rho$, we have a unique decomposition $\mathfrak{m}=\mathfrak{m}_\mathrm{sym} + \mathfrak{m}_\mathrm{ant}$, where $\mathfrak{m}_\mathrm{sym} \in \mathrm{Mult}_\rho$  such that $\mathfrak{m}_\mathrm{sym}^\dagger=\mathfrak{m}_\mathrm{sym}$ and $\mathfrak{m}_\mathrm{sym}$ is the maximal multisegment with that property. Therefore, for a segment $\Delta \in \mathfrak{m}_\mathrm{ant}$, we have $\Delta^\dagger \notin \mathfrak{m}_\mathrm{ant}$.

\begin{proposition}\label{prop:match_prod_speh_der}
Let $\rho \in \mathrm{Irr^{unit}}$ be an irreducible cuspidal representation.
Suppose $\mathfrak{m}=\mathfrak{m}_\rho(u_1,v_1,R^{w_1}) \nu^{\alpha_1}+...+\mathfrak{m}_\rho(u_{2r},v_{2r},R^{w_{2r}})\nu^{\alpha_{2r}}+\mathfrak{m}_\rho(u_{2r+1},v_{2r+1},R^{w_{2r+1}})\nu^{\alpha_{2r+1}} + ...+\mathfrak{m}_\rho(u_k,v_k,R^{w_k}) \nu^{\alpha_{k}}$ for integer $r \ge 0$,  and $\mathfrak{n}=\mathfrak{m}_\rho(u_1',v_1',L_{w_1'}) \nu^{\beta_1}+...+\mathfrak{m}_\rho(u_{2s}',v_{2s}',L_{w_{2s}'}) \nu^{\beta_{2s}}   + \mathfrak{m}_\rho(u_{2s+1}',v_{2s+1}',L_{w_{2s+1}'}) \nu^{\beta_{2s+1}} +...+\mathfrak{m}_\rho(u_l',v_l',L_{w_l'}) \nu^{\beta_l}$ for some unitary Speh multisegments $\mathfrak{m}_\rho(u_i,v_i), \mathfrak{m}_\rho(u_j',v_j')$, some integers $0 \le w_i \le u_i, 0 \le w_j' \le u_j'$, and some real numbers $-\frac{1}{2}<\alpha_i, \beta_j <\frac{1}{2}$, where
$u_{2i}=u_{2i-1}, v_{2i}=v_{2i-1}, \alpha_{2i}=-\alpha_{2i-1}$ for $i \le r$, $u_{2j}'=u_{2j-1}', v_{2j}'=v_{2j-1}', \beta_{2j}=-\beta_{2j-1}$ for $j \le s$, and $\alpha_i=0$ for $i > 2r$, $\beta_j=0$ for $j > 2s$. Then,
\[\nu^\frac{1}{2}\mathfrak{m}=\mathfrak{n} \implies w_i \in \{0, u_i\} \text{ and }w_j' \in \{0, u_j'\} \text{ for all } 1 \le i \le k, 1 \le j \le l.\]
\end{proposition}

\begin{proof}
To prove $w_i \in \{0, u_i\}$ for $1 \le i \le k$, fix $\mathfrak{n}'=\nu^{-\frac{1}{2}}\mathfrak{n}$. As $\left(u_{2j}',v_{2j}', \beta_{2j} \right) = \left(u_{2j-1}', v_{2j-1}', -\beta_{2j-1}\right)$, we have $\left(\mathfrak{m}_\rho(u_{2j}',v_{2j}'-1)\nu^{\beta_{2j}}\right)^\dagger = \mathfrak{m}_\rho(u_{2j-1}',v_{2j-1}'-1)\nu^{\beta_{2j-1}}$ and $\left(\mathfrak{m}_\rho(u_{2j-1}',v_{2j-1}'-1)\nu^{\beta_{2j-1}}\right)^\dagger = \mathfrak{m}_\rho(u_{2j}',v_{2j}'-1)\nu^{\beta_{2j}}$ for $j \le s$. Also, as $\beta_j=0$, we have $\left(\mathfrak{m}_\rho(u_j',v_j'-1)\nu^{\beta_j}\right)^\dagger=\mathfrak{m}_\rho(u_j',v_j'-1)\nu^{\beta_j}$ for $j > 2s$. Therefore, $\mathfrak{n}'=\mathfrak{n}_\mathrm{sym}'+\mathfrak{n}_\mathrm{ant}' $, where 
\[\mathfrak{n}_\mathrm{sym}' = \mathfrak{m}_\rho(u_1',v_1'-1)\nu^{\beta_1}+...+\mathfrak{m}_\rho(u_l',v_l'-1)\nu^{\beta_l},\] and 
\[\mathfrak{n}_\mathrm{ant}' = \sum\limits_{j=1}^l \left[ -\frac{v_j'-1}{2}-\frac{u_j'-1}{2},~ -\frac{v_j'-1}{2}+\frac{u_j'-1}{2}-w_j'\right]_\rho \nu^{-\frac{1}{2}+\beta_j}.\]Therefore, if $\Delta \in \mathfrak{n}_\mathrm{ant}'$, we have
\begin{equation}\label{eq:SD1}
    s_\rho(\Delta) + e_\rho(\Delta)=-(v_j'-1)-w_j'-1+2\beta_j <0, \text{ since }  v_j'\ge 1 \text{ and } \beta_j <\frac{1}{2}.
\end{equation}

If $i=2d$ (resp. $i=2d-1$) for some $ d \le r$, then we denote $i^c=2d-1$ (resp. $i^c=2d$). If $i > 2r$, we write $i^c=i$. We now decompose $\mathfrak{m}$ as $\mathfrak{m}=\mathfrak{p} + \sum\limits_{i=1}^k \Delta^\mathrm{top}_i + \sum\limits_{i=1}^k \Delta^\mathrm{bot}_i$, where
\[\mathfrak{p}=\mathfrak{m}_\rho(u_1,v_1-2) \nu^{\alpha_1}+...+ \mathfrak{m}_\rho(u_k,v_k-2) \nu^{\alpha_k},\] and
\[ \Delta^\mathrm{top}_i = 
  \left[\frac{v_i-1}{2}-\frac{u_i-1}{2}+w_i, \frac{v_i-1}{2} +\frac{u_i-1}{2}\right]_\rho\nu^{\alpha_i} ,\]
together with
\[\Delta^\mathrm{bot}_i = \begin{cases}
  \left[-\frac{v_i-1}{2}-\frac{u_i-1}{2}, -\frac{v_i-1}{2} +\frac{u_i-1}{2}\right]_\rho\nu^{\alpha_i} &\mbox{ if } v_i >1\\ 
  \emptyset &\mbox{ if } v_i =1.
 \end{cases} \]
Also, we have $\left(\mathfrak{m}_\rho(u_{i},v_{i}-2)\nu^{\alpha_{i}}\right)^\dagger = \mathfrak{m}_\rho(u_{i^c},v_{i^c}-2)\nu^{\alpha_{i^c}}$, and $\left(\mathfrak{m}_\rho(u_{i^c},v_{i^c}-2)\nu^{\alpha_{i^c}}\right)^\dagger=\mathfrak{m}_\rho(u_i,v_i-2)\nu^{\alpha_i}$ for $1 \le i \le k$. Therefore, we get
\begin{equation*}
 \mathfrak{p}^\dagger=\mathfrak{p},   \text{ and hence, } \mathfrak{p} \subset \mathfrak{n}'_\mathrm{sym}.   
\end{equation*}
For each $1 \le i \le k$,
\begin{align*}\label{eq:SD3}
    s_\rho\left( \Delta^\mathrm{top}_i \right) + e_\rho\left( \Delta^\mathrm{top}_i \right)=(v_i-1)+ w_i + 2 \alpha_i \quad \begin{cases}
        > 0 &\mbox{ if } \alpha_i > 0, \text{ or } w_i >0, \text{ or } v_i >1\\
        = 0 &\mbox{ if } \alpha_i = 0 \text{ and } w_i =0 \text{ and } v_i = 1,\\
        < 0 &\mbox{ if } \alpha_i < 0 \text{ and } w_i =0 \text{ and } v_i = 1,
    \end{cases}
\end{align*}
and
\begin{equation*}
 s_\rho\left( \Delta^\mathrm{bot}_i \right) + e_\rho\left( \Delta^\mathrm{bot}_i \right)=-(v_i-1)+2\alpha_i < 0  \text{ whenever } \Delta^\mathrm{bot}_i \ne \emptyset.
\end{equation*}
Therefore, we have $w_i=0$ and $v_i=1$, whenever $s_\rho\left( \Delta^\mathrm{top}_i \right) + e_\rho\left( \Delta^\mathrm{top}_i \right) \le 0$ for some $1 \le i \le k$. If $s_\rho\left( \Delta^\mathrm{top}_i \right) + e_\rho\left( \Delta^\mathrm{top}_i \right)=0$, we have
\[\Delta^\mathrm{top}_i=\left[-\frac{u_i-1}{2}, \frac{u_i-1}{2} \right]_\rho \text{ and }\left(\Delta^\mathrm{top}_i\right)^\dagger=\Delta^\mathrm{top}_i=\mathfrak{m}_\rho(u_i,v_i)(\alpha_i) \subset \mathfrak{n}_\mathrm{sym}'.\]
If $s_\rho\left( \Delta^\mathrm{top}_i \right) + e_\rho\left( \Delta^\mathrm{top}_i \right)<0$, we have $\Delta^\mathrm{top}_i=\left[-\frac{u_i-1}{2}, \frac{u_i-1}{2} \right]_\rho\nu^{\alpha_i}$ with $\alpha_i <0$ and we denote $\Delta^\mathrm{top}_{i,w_i=0}=\Delta^\mathrm{top}_i$. Let $I \subsetneq \{1,...,k\}$ such that for $i \in I$, we have $s_\rho\left( \Delta^\mathrm{top}_i \right) + e_\rho\left( \Delta^\mathrm{top}_i \right)<0$ and $\Delta^\mathrm{top}_{i, w_i=0} \in \mathfrak{n}_\mathrm{ant}'$. If $i \in \{1,...,k\} \setminus I$ and $s_\rho\left( \Delta^\mathrm{top}_i \right) + e_\rho\left( \Delta^\mathrm{top}_i \right)<0$, we have
\[\Delta^\mathrm{top}_i, \left(\Delta^\mathrm{top}_i\right)^\dagger=\left[-\frac{u_i-1}{2}, \frac{u_i-1}{2} \right]_\rho\nu^{-\alpha_i} \in \mathfrak{n}_\mathrm{sym}' ,\]
together with
\[\mathfrak{m}_\rho(u_i,v_i)(\alpha_i)=\mathfrak{m}_\rho(u_i,1)(\alpha_i)=\Delta^\mathrm{top}_i + \left(\Delta^\mathrm{top}_i\right)^\dagger \subset \mathfrak{n}_\mathrm{sym}'.\]

We now consider only those $\Delta^\mathrm{top}_i$ where $\Delta^\mathrm{top}_i \ne \emptyset$ (i.e. $w_i \ne u_i$) and $s_\rho\left( \Delta^\mathrm{top}_i \right) + e_\rho\left( \Delta^\mathrm{top}_i \right) > 0$. As $\mathfrak{m}=\mathfrak{n}'$, applying \eqref{eq:SD1}, we have 
 \[\Delta^\mathrm{top}_i \in  \mathfrak{n}_\mathrm{sym}', \text{ and so, } \left(\Delta^\mathrm{top}_i\right)^\dagger \in  \mathfrak{n}_\mathrm{sym}'. \]
If $\left(\Delta^\mathrm{top}_i\right)^\dagger= \Delta^\mathrm{top}_i$, we have $w_i=0=\alpha_i$, $v_i=1$ and $\mathfrak{m}_\rho(u_i,1)^\dagger=\mathfrak{m}_\rho(u_i,1) \subset \mathfrak{n}_\mathrm{sym}'$. Now, suppose $\left(\Delta^\mathrm{top}_i\right)^\dagger \ne  \Delta^\mathrm{top}_i$. Observe that
\begin{equation*}
 s_\rho\left( \left(\Delta^\mathrm{top}_i\right)^\dagger \right) + e_\rho\left( \left(\Delta^\mathrm{top}_i\right)^\dagger \right)=-(v_i-1)-w_i-2\alpha_i < 0.
\end{equation*}
As $\left(\Delta^\mathrm{top}_i\right)^\dagger \in  \mathfrak{n}_\mathrm{sym}'$ and $ \mathfrak{p}^\dagger=\mathfrak{p}$, we have either \[\left(\Delta^\mathrm{top}_i\right)^\dagger=\Delta^\mathrm{top}_j \text{ for some } j \text{ such that } \alpha_j <0 \text{ and }(w_j,v_j)=(0,1),\]
or, we have
\[\left(\Delta^\mathrm{top}_i\right)^\dagger=\Delta^\mathrm{bot}_j \text{ for some } j \text{ such that } \Delta^\mathrm{bot}_j \ne \emptyset.\]

Case 1. $\left(\Delta^\mathrm{top}_i\right)^\dagger=\Delta^\mathrm{top}_j$ with $\alpha_j <0 \text{ and }(w_j,v_j)=(0,1)$. As $\Delta^\mathrm{top}_j=\mathfrak{m}_\rho(u_j,1)\nu^{\alpha_j}$, we have $\Delta^\mathrm{top}_i=\left(\Delta^\mathrm{top}_j\right)^\dagger=\mathfrak{m}_\rho(u_j,1)\nu^{-\alpha_j}$ and hence,
\[ \mathfrak{m}_\rho(u_j,v_j)(\alpha_j)=\mathfrak{m}_\rho(u_j,1)\nu^{\alpha_j} + \mathfrak{m}_\rho(u_{j},1)\nu^{-\alpha_j}=\Delta^\mathrm{top}_j + \Delta^\mathrm{top}_i \subset \mathfrak{n}_\mathrm{sym}'.\]

Case 2. $\left(\Delta^\mathrm{top}_i\right)^\dagger=\Delta^\mathrm{bot}_j$ with $v_j >1$. Then, we have
\[\mathfrak{m}_\rho(u_{j^c},v_{j^c})\nu^{\alpha_{j^c}} =\mathfrak{m}_\rho(u_j,v_j)\nu^{-\alpha_j}=\Delta^\mathrm{top}_i+\mathfrak{m}_\rho(u_j,v_j-2)\nu^{-\alpha_j} + \Delta^\mathrm{bot}_{j^c} \subset \mathfrak{n}_\mathrm{sym}'.\]
Therefore, there exist disjoint subsets $I', J \subset \{1,...,k\}\setminus I$ such that $\mathfrak{m}=\mathfrak{n}_\mathrm{sym}' + \mathfrak{n}_\mathrm{ant}'$, where
\[\mathfrak{n}_\mathrm{sym}'= \sum\limits_{i \in I'} \mathfrak{m}_\rho(u_i,v_i-2)\nu^{\alpha_i}    + \sum\limits_{j \in J} \mathfrak{m}_\rho(u_j,v_j)\nu^{\alpha_j}      \text{ and }    \mathfrak{n}_\mathrm{ant}'=\sum\limits_{i \in I'}\Delta^\mathrm{bot}_{i} + \sum\limits_{i \in I} \Delta^\mathrm{top}_{i, w_i=0}.\]
In particular, whenever $\Delta^\mathrm{top}_i \ne \emptyset$ (i.e. $w_i \ne u_i$), they are actually the top segment of the ladder part $\mathfrak{m}_\rho(u_i,v_i)\nu^{\alpha_i}$. Hence, $w_i \in \{0, u_i\}$ for all $i$. Similarly, one can prove that $w_j' \in \{0, u_j'\}$ for $1 \le j \le l$.
\end{proof}

\begin{corollary}\label{cor:match_hd}
Let $\pi=L \left(\sum\limits_{i=1}^r \mathfrak{m}_\rho(u_i,v_i)\nu^{\alpha_i} \right)$ and $\pi'=L \left(\sum\limits_{i=1}^s \mathfrak{m}_\rho(u_j',v_j')\nu^{\beta_j} \right)$ be two unitary representations for unitary cuspidal $\rho$, $-\frac{1}{2} < \alpha_i, \beta_j <\frac{1}{2}$ and positive integers $u_i,v_i,u_j',v_j'$. Suppose, $\mathrm{D}^\mathrm{R}_\mathfrak{p}(\nu^\frac{1}{2}\pi) \cong \mathrm{D}^\mathrm{L}_\mathfrak{q}(\pi')$ for some $\mathfrak{p,q} \in \mathrm{Mult}_\rho$. Then, there exists $\widetilde{\mathfrak{p}}$ (resp. $\widetilde{\mathfrak{q}}$) in $\mathrm{Mult}_\rho$, which consists only of the top (resp. bottom) segment of the ladder multisegment $\mathfrak{m}_\rho(u_i,v_i)\nu^{\frac{1}{2}+\alpha_i}$ (resp. $\mathfrak{m}_\rho(u_j',v_j')\nu^{\beta_j}$) for some $1\le i \le r$ and some $1\le j \le s$ such that
\[\mathrm{D}^\mathrm{R}_{\widetilde{\mathfrak{p}}} (\nu^\frac{1}{2}\pi) \cong \mathrm{D}^\mathrm{L}_{\widetilde{\mathfrak{q}}}(\pi')\]
\end{corollary}
\begin{proof}
    This follows immediately from Proposition \ref{prop:mult_der} and Proposition \ref{prop:match_prod_speh_der}.
\end{proof}

\begin{lemma}\label{lem:inclution}
Let $\rho$ be unitary supercuspidal, $ 0\le \alpha <\frac{1}{2}, -\frac{1}{2}<\beta_i <\frac{1}{2}$, and $u,v,u_i,v_i$ be positive integers such that $u+v+\alpha \ge u_i+v_i+\beta_i$ and $v \ne 1$. If $\mathfrak{m}_\rho(u,v-1)\nu^{\pm \alpha} \subset \sum\limits_{i=1}^r \mathfrak{m}_\rho (u_i,v_i)\nu^{\beta_i} + \sum\limits_{i=r+1}^s {^-}\mathfrak{m}_\rho (u_i,v_i)\nu^{\beta_i}$, then there exists positive integer $i \le s$ such that $(u_i,v_i)=(u,v)$ and $|\beta_i|=\frac{1}{2}-\alpha$, or there exists positive integer $i_0 \le r$ such that 
\[(u_{i_0},v_{i_0},\beta_{i_0})=(u, v-1, \pm \alpha) \text{ and }   \mathfrak{m}_\rho(u,v-1)\nu^{\pm \alpha}=\mathfrak{m}_\rho (u_{i_0},v_{i_0})\nu^{\beta_{i_0}}.\]
\end{lemma}
\begin{proof}
Let $\Delta^\mathrm{bot}_i$ be the bottom segment of $\mathfrak{m}_\rho (u_i,v_i)\nu^{\beta_i}$ and $\Delta^\mathrm{bot}$ be the bottom segment of $\mathfrak{m}_\rho (u,v-1)\nu^{\alpha}$. There exists a positive integer $i \le s$ such that  
$\Delta^\mathrm{bot} \in \mathfrak{m}_\rho (u_i,v_i)\nu^{\beta_i}.$ Then, $u=u_i$ and for some non-negative integer $t$, we have
$s_\rho \left( \Delta^\mathrm{bot}_i \right) + t = s_\rho \left( \Delta^\mathrm{bot} \right)$. Therefore,
$-\frac{v_i-1}{2} + \beta_i + t= -\frac{v-2}{2} + \alpha$, which implies
\[v_j - (2t-1)= v - ( 2 \alpha -2 \beta_i).\]
Then, $2\alpha-2\beta_i$ is an integer with $\alpha=\beta_i$ if $\beta_i \ge 0$ and $\alpha - \beta_i=\frac{1}{2}$ if  $\beta_i < 0$. Since $v+\alpha \ge v_i + \beta_i$, we conclude 
\[(u_i, v_i, \beta_i)=\begin{cases}
    (u,v-1, \alpha) &\mbox{ if } \beta_i \ge 0\\
    (u,v, \alpha-\frac{1}{2}) &\mbox{ if } \beta_i < 0
\end{cases}\]
Similarly, if $\Delta^\mathrm{bot}$ be the bottom segment of $\mathfrak{m}_\rho (u,v-1)\nu^{-\alpha}$ such that 
$\Delta^\mathrm{bot} \in \mathfrak{m}_\rho (u_i,v_i)\nu^{\beta_i},$  and we have
\[(u_i, v_i, \beta_i)=\begin{cases}
    (u,v-1, -\alpha) &\mbox{ if } \beta_i \le 0\\
    (u,v, \frac{1}{2}-\alpha) &\mbox{ if } \beta_i > 0.
\end{cases}\]
\end{proof}

\begin{lemma}\label{lem:inclution_2}
Let $\mathfrak{m}=\sum\limits_{i=1}^r \mathfrak{m}_\rho (u_i,v_i)\nu^{\beta_i}$ and $\Delta^\mathrm{top}$ be the top segment of the ladder multisegment $\mathfrak{m}_\rho(u,v)\nu^{\frac{1}{2}-\alpha}$, where $\rho$ is a unitary supercuspidal, $ 0< \alpha <\frac{1}{2}, -\frac{1}{2}<\beta_i <\frac{1}{2}$, and $u,v,u_i,v_i$ be positive integers such that $u+v+\alpha \ge u_i+v_i+\beta_i$. Let $\mathfrak{p}$ be consists of the top segment of the ladder multisegment $\mathfrak{m}_\rho(u_i,v_i)\nu^{\beta_i}$ for some $i$. If $\Delta^\mathrm{top} \in \mathcal{D}^\mathrm{R}_\mathfrak{p}(\mathfrak{m})$, there exists positive integer $i_0 \le r$ such that 
\[(u_{i_0},v_{i_0},\beta_{i_0})=(u, v, \frac{1}{2}-\alpha) \text{ and }   \mathcal{D}^\mathrm{R}_\mathfrak{p}(\mathfrak{m})=\mathfrak{m}_\rho (u_{i_0},v_{i_0})\nu^{\beta_{i_0}} + \mathcal{D}^\mathrm{R}_\mathfrak{p}\left(\mathfrak{m}- \mathfrak{m}_\rho (u_{i_0},v_{i_0})\nu^{\beta_{i_0}}\right).\]
An analogous result holds for the left derivative $\mathcal{D}^\mathrm{L}_\mathfrak{p}(\mathfrak{m})$ with $\mathfrak{p}$ (resp. $\Delta^\mathrm{top}$) replaced by bottom segments.
\end{lemma}
\begin{proof}
There exists a positive integer $i \le r$ such that $\Delta^\mathrm{top} \in \mathfrak{m}_\rho (u_i,v_i)\nu^{\beta_i}$ and let $\Delta^\mathrm{top}_i$ be the top segment of $\mathfrak{m}_\rho (u_i,v_i)\nu^{\beta_i}$ and $\Delta^\mathrm{bot}$ be the bottom segment of the ladder $\mathfrak{m}_\rho (u,v-1)\nu^{\alpha}$. Then, $u=u_i$ and for some non-negative integer $t$, we have
$e_\rho \left( \Delta^\mathrm{top}_i \right) - t = e_\rho \left( \Delta^\mathrm{top} \right)$. With a similar argument as for Lemma \ref{lem:inclution}, we conclude that $t=0$ and $(u_{i},v_{i},\beta_{i})=(u, v, \frac{1}{2}-\alpha)$. Hence, the result follows from Proposition \ref{prop:der_unitary_mult}.
\end{proof}

\section{Notion of RdLd-matching}\label{sec:RdLi}
In this section, we introduce the notion of RdLd-matching, which isolates the first condition of Definition \ref{def:relevant} (the derivative matching condition) and provides a framework for analyzing it independently of the strong commutativity condition. This separation is crucial for the proof of Theorem \ref{thm:unitary}, as it allows us to study the matching condition using induction on the complexity of unitary representations. We also establish reduction results for unitary representations, which form the core of the proof that generalized GGP relevance implies the extended Gan-Gross-Prasad relevance criterion (Definition \ref{def:extended_relevance}).

\begin{definition}
    Let $\pi, \pi' \in \mathrm{Irr}$. We say that the pair $(\pi,\pi')$ is \textit{RdLd-matching} if there exist multisegments $\mathfrak{p}$ and $\mathfrak{q}$ such that \[\mathrm{D}^\mathrm{R}_\mathfrak{p}(\nu^\frac{1}{2}\pi) \ne 0 \text{ and }
    \mathrm{D}^\mathrm{R}_\mathfrak{p}(\nu^\frac{1}{2}\pi) \cong \mathrm{D}^\mathrm{L}_\mathfrak{q}(\pi').\]
\end{definition}

\begin{lemma}[Duality]\label{lem:duality}
 Let $\pi, \pi' \in \mathrm{Irr}$. The pair $(\pi,\pi')$ is RdLd-matching if and only if $(\pi',\pi)$ is RdLd-matching.
\end{lemma}
\begin{proof}
The duality follows from the proof of \cite[Theorem 18.1]{Cha_qbl}. Specifically, the argument in Step 4 of that proof—which relies only on Steps 1 and 2 of the same theorem—is independent of the strongly RdLi-commutativity condition used elsewhere in \cite[Theorem 18.1]{Cha_qbl}. Consequently, the duality holds for RdLd-matching pairs without requiring the full strength of the generalized GGP relevance.

In particular, if $(\pi,\pi')$ is RdLd-matching, then there are Rd-minimal $\mathfrak{m}$ and Ld-minimal $ \mathfrak{n}$ such that 
    \[\mathrm{D}^\mathrm{R}_\mathfrak{m}(\nu^\frac{1}{2}\pi) \ne 0 \text{ and }
    \mathrm{D}^\mathrm{R}_\mathfrak{m}(\nu^\frac{1}{2}\pi) \cong \mathrm{D}^\mathrm{L}_\mathfrak{n}(\pi').\]
Applying Proposition \ref{prop:interchange} (which gives explicit formulas for the multisegments under interchange), we obtain that the pair $(\pi',\pi)$ is RdLd-matching with respect to the multisegments   
\[\mathfrak{p}=  \nu^\frac{1}{2} \mathfrak{r}^\mathrm{R}\left(\mathfrak{m}, \mathfrak{hd}^\mathrm{R}\left(\mathrm{I}^\mathrm{R}_{\mathfrak{m}}  \left(\pi'\right)\right)\right) \text{ and }
\mathfrak{q}=\nu^{-\frac{1}{2}} \mathfrak{r}^\mathrm{L}\left(\mathfrak{n}, \mathfrak{hd}^\mathrm{L}\left(\mathrm{I}^\mathrm{R}_{\mathfrak{m}}  \left( \pi'\right)\right)\right).\] 
\end{proof}
\subsection{RdLd-matching for unitary representations}\label{sec:RdLi-match}
For the remainder of this section, we specialize to the case where $\pi$ and $\pi'$ are irreducible unitary representations. Let us fix a unitary cuspidal representation $\rho \in \mathrm{Irr^{cusp}}$. We will assume that all factors in $\pi$ and $\pi'$ lie in the same cuspidal line $\rho$; the general case (with multiple distinct cuspidal types) follows by decomposing over cuspidal lines, as the derivative and integral operations factor accordingly. Assume that $\pi$ and $\pi'$ are of the form
\begin{equation}
 \pi=L(\mathfrak{m}) \cong \pi_1(\alpha_1) \times ... \times \pi_r(\alpha_r) \text{ and }   \pi'=L(\mathfrak{n}) \cong \pi_1'(\beta_1) \times ... \times \pi_l'(\beta_l),
\end{equation}
where for each $i$ and $j$,
\[\pi_i=\pi_{\rho}(u_i,v_i) \text{ with } u_i,v_i \in \mathbb{Z}_{>0};~\pi_j'=\pi_{\rho}(u_j',v_j') \text{ with } u_j',v_j' \in \mathbb{Z}_{>0},\] and $0 \le \alpha_i,\beta_j <\frac{1}{2}$. Recall that $\pi_{\rho}(u,v)(\alpha)$ denotes the complementary series representation when $\alpha>0$ and the Speh representation itself when $\alpha=0$.

We fix an ordering of the factors of $\pi$ and $\pi'$ such that the size parameters $u_{i} + v_{i} +\alpha_{i}$ and $u_{j}' + v_{j}' +\beta_{j}$ are non-increasing, and when equality occurs, the $u_i$ (resp. $u_j'$) are non-increasing. Without loss of generality, we may assume that the largest factor in $\pi$ is $\pi_1(\alpha_1)$; i.e.,
\[u_{1} + v_{1} +\alpha_{1}\ge u_{i} + v_{i} +\alpha_{i} \text{ for all } i,\] and 
\[u_{1} + v_{1} +\alpha_{1} \ge u_{j}' + v_{j}' +\beta_{j} \text{ for all } j,\] with the understanding that if equality holds in the second inequality, we also have $u_1 \ge u_j'$.

We now establish two reduction results that allow us to “peel off” the largest factor $\pi_1(\alpha_1)$ from $\pi$ under the assumption that the pair is RdLd-matching. These results are the technical core of the proof that generalized GGP relevance implies the extended GGP relevance (Theorem \ref{thm:unitary}). The analysis splits into two cases, depending on whether the largest factor can be matched with a complementary factor in $\pi'$ in a specific way.

\subsection{Case A}\label{sec:Case A}{\bf No factor in $\pi'$  matches $\pi_1(\frac{1}{2}-\alpha_1)$.} In this case, we assume that \[\pi'_j(\beta_j) \ncong \pi_1(\frac{1}{2}-\alpha_1) \text{ for all } j=1,...,l.\] Let us define \[\pi_*=\pi_{2}(\alpha_{2}) \times ... \times \pi_r(\alpha_r)=L(\mathfrak{m}_*).\]

\begin{proposition}[Reduction in Case A]\label{prop:match_1}
Suppose $\pi'_j(\beta_j) \ncong \pi_1(\frac{1}{2}-\alpha_1)$ for all $j$. If $(\pi, \pi')$ is RdLd-matching, then  there exists an index $j_0 \in \{1,...,l\}$ such that the pair $(\pi_*, \pi'_*)$ is also RdLd-matching, where  \[\pi'_* \cong \begin{cases}
   \pi_1'(\beta_1) \times ...\times \pi_{j_0-1}'(\beta_{j_0-1})\times \pi_{j_0+1}'(\beta_{j_0+1})\times... \times \pi_l'(\beta_l)  &\mbox{ if } v_1 \ne 1\\ 
    \pi' &\mbox{ if } v_1 = 1.
\end{cases}
\]
\end{proposition}
\begin{proof}
Since $(\pi, \pi')$ is RdLd-matching, by Corollary \ref{cor:match_hd}, there exists a multisegment $\mathfrak{p}$ (resp. $\mathfrak{q}$) that consists only of top segment $\mathfrak{hd}^{\mathrm{R}}\left(\pi_\rho(u_i,v_i)\nu^{\frac{1}{2} \pm \alpha_i}\right)$ (resp. bottom segment $\mathfrak{hd}^{\mathrm{L}}\left(\pi_\rho(u_j',v_j')\nu^{\pm\beta_j}\right)$) for some $1\le i \le r$ and some $1\le j \le l$ such that
 \begin{equation} \label{eq:red_A1}
     \mathrm{D}^\mathrm{R}_\mathfrak{p}(\nu^\frac{1}{2}\pi)  \cong \mathrm{D}^\mathrm{L}_\mathfrak{q}(\pi') \ne 0 \quad (\text{equivalently } \mathcal{D}^\mathrm{R}_\mathfrak{p}(\nu^\frac{1}{2}\mathfrak{m}) = \mathcal{D}^\mathrm{L}_\mathfrak{q}(\mathfrak{n}) \ne \infty).
 \end{equation} 
We consider two subcases based on whether $\alpha_1 =0$.

Subcase A1. $\alpha_1 \ne 0$. Because $u_{1} + v_{1} +\alpha_{1}$ is maximal, the ending point
$e(\nu^{\frac{1}{2}+\alpha_1}\pi_1) \notin \mathrm{csupp}(\pi')$. By the matching condition \eqref{eq:red_A1}, $e(\nu^{\frac{1}{2}+\alpha_1}\pi_1) \notin \mathrm{csupp}\left( \mathrm{D}^\mathrm{R}_\mathfrak{p}(\nu^\frac{1}{2}\pi)\right)$ too. Therefore, the segment 
\[\Delta_1=\mathfrak{hd}^\mathrm{R}\left(\nu^{\frac{1}{2}+\alpha_1}\pi_1 \right) \in \mathfrak{p}.\] 
Applying Proposition \ref{prop:der_unitary_mult} (which describes the derivative of a Speh factor with respect to its top segment), we obtain
\begin{equation}\label{eq:red_A2}
\mathcal{D}^\mathrm{R}_{\mathfrak{p}}(\nu^\frac{1}{2} \mathfrak{m})= \mathfrak{m}^-_\rho(u_1,v_1)\nu^{\frac{1}{2}+\alpha_1} +   \mathcal{D}^\mathrm{R}_{\mathfrak{p}-\Delta_1} \left( \mathfrak{m}_\rho(u_1,v_1)\nu^{\frac{1}{2}-\alpha_1} +\nu^\frac{1}{2}\mathfrak{m}_*   \right) ,    
\end{equation}
By the matching condition \eqref{eq:red_A1},
$\mathfrak{m}^-_\rho(u_1,v_1)\nu^{\frac{1}{2}+\alpha_1}=\mathfrak{m}_\rho(u_1,v_1-1)\nu^{\alpha_1}$  is a submultisegment of $\mathcal{D}^\mathrm{L}_\mathfrak{q}(\mathfrak{n})$. 
Lemma \ref{lem:inclution} then forces either $v_1=1$ (in which case $\mathfrak{m}_\rho(u_1,v_1-1)\nu^{\alpha_1}=\emptyset$), or $v_1\ne 1$ and there exists an index $1 \le j_0\le l$ such that \[\mathfrak{m}_\rho(u_{j_0}',v_{j_0}')\nu^{\beta_{j_0}}=\mathfrak{m}_\rho(u_1,v_1-1)\nu^{\alpha_1}.\] 
If $v_1 \ne 1$, by Proposition \ref{prop:der_unitary_mult_left} we can `peel off' this factor from the left derivative:
\begin{equation}\label{eq:red_A3}
\mathcal{D}^\mathrm{L}_\mathfrak{q}(\mathfrak{n})= \mathfrak{m}_\rho(u_{j_0}',v_{j_0}')\nu^{\beta_{j_0}}+ \mathcal{D}^\mathrm{L}_\mathfrak{q}(\mathfrak{n}-\mathfrak{m}_\rho(u_{j_0}',v_{j_0}')\nu^{\beta_{j_0}}).    
\end{equation}
Therefore, combining relations \eqref{eq:red_A1}, \eqref{eq:red_A2}, and \eqref{eq:red_A3}, we have
\begin{equation}\label{eq:red_A4}
\mathcal{D}^\mathrm{R}_{\mathfrak{p}-\Delta_1} \left( \mathfrak{m}_\rho(u_1,v_1)\nu^{\frac{1}{2}-\alpha_1} +\nu^\frac{1}{2}\mathfrak{m}_*   \right) = \begin{cases}
    \mathcal{D}^\mathrm{L}_\mathfrak{q}\left(\mathfrak{n}\right) &\mbox{ if } v_1 = 1\\
    \mathcal{D}^\mathrm{L}_\mathfrak{q}\left(\mathfrak{n}-\mathfrak{m}_\rho(u_{j_0}',v_{j_0}')\nu^{\beta_{j_0}}\right) &\mbox{ if } v_1 \ne 1.
\end{cases}
\end{equation}
A further analysis using the fact that $\pi'_j(\beta_j) \ncong \pi_1(\frac{1}{2}-\alpha_1)$ for all $1 \le j \le l$ shows that the segment $\Delta_2=\mathfrak{hd}^\mathrm{R}\left(\nu^{\frac{1}{2}-\alpha_1}\pi_1 \right) $ also belongs to $ \mathfrak{p}-\Delta_1$. Applying Proposition \ref{prop:der_unitary_mult} again, we have 
\begin{equation}\label{eq:red_A5}
\mathcal{D}^\mathrm{R}_{\mathfrak{p}-\Delta_1} \left( \mathfrak{m}_\rho(u_1,v_1)\nu^{\frac{1}{2}-\alpha_1} +\nu^\frac{1}{2}\mathfrak{m}_*   \right)= \mathfrak{m}_\rho(u_1,v_1-1)\nu^{-\alpha_1} +   \mathcal{D}^\mathrm{R}_{\mathfrak{p}-\Delta_1-\Delta_2} \left(\nu^\frac{1}{2}\mathfrak{m}_*   \right).   
\end{equation}
If $v_1\ne 1$, by the matching condition \eqref{eq:red_A4}, $\mathfrak{m}_\rho(u_1,v_1-1)\nu^{-\alpha_1} $ (which is $\mathfrak{m}_\rho(u_{j_0}',v_{j_0}')\nu^{-\beta_{j_0}}$) is a submultisegment of $ \mathcal{D}^\mathrm{L}_\mathfrak{q}\left(\mathfrak{n}-\mathfrak{m}_\rho(u_{j_0}',v_{j_0}')\nu^{\beta_{j_0}}\right)$ and applying Lemma \ref{lem:inclution} and Proposition \ref{prop:der_unitary_mult_left}, we conclude 
\begin{equation}\label{eq:red_A6}
\mathcal{D}^\mathrm{L}_\mathfrak{q}\left(\mathfrak{n}-\mathfrak{m}_\rho(u_{j_0}',v_{j_0}')\nu^{\beta_{j_0}}\right)= \mathfrak{m}_\rho(u_{j_0}',v_{j_0}')\nu^{-\beta_{j_0}}+ \mathcal{D}^\mathrm{L}_\mathfrak{q}\left(\mathfrak{n}-\mathfrak{m}_\rho(u_{j_0}',v_{j_0}')\nu^{\beta_{j_0}}-\mathfrak{m}_\rho(u_{j_0}',v_{j_0}')\nu^{-\beta_{j_0}}\right).    
\end{equation}
We now set $\pi'_*=L(\mathfrak{n}_*)$ where $\mathfrak{n}_*=\mathfrak{n}-\mathfrak{m}_\rho(u_{j_0}',v_{j_0}')({\beta_{j_0}})$ if $v_1 \ne 1$ and $\mathfrak{n}_*= \mathfrak{n} $ if $ v_1=1$. Therefore, combining relations \eqref{eq:red_A4}, \eqref{eq:red_A5}, and \eqref{eq:red_A6}, we have
\begin{equation}\label{eq:red_A7}
\mathcal{D}^\mathrm{R}_{\mathfrak{p}-\Delta_1-\Delta_2} \left( \nu^\frac{1}{2}\mathfrak{m}_*   \right) = \mathcal{D}^\mathrm{L}_\mathfrak{q}\left(\mathfrak{n}_*\right), \text{ and equivalently, }  \mathrm{D}^\mathrm{R}_{\mathfrak{p}-\Delta_1-\Delta_2} \left( \nu^\frac{1}{2}\pi_*   \right) \cong \mathrm{D}^\mathrm{L}_\mathfrak{q}\left(\pi'_*\right).
\end{equation}
This establishes the RdLd-matching for $(\pi_*, \pi_*')$.

Subcase A2. $\alpha_1 = 0$. Because $u_{1} + v_{1} +\alpha_{1}$ is maximal, similar to Subcase A1, the segment $\Delta_0=\mathfrak{hd}^\mathrm{R}\left(\nu^{\frac{1}{2}}\pi_1 \right) $ lies in $ \mathfrak{p}$ and applying Proposition \ref{prop:der_unitary_mult}, we have
\begin{equation}\label{eq:red_A8}
\mathcal{D}^\mathrm{R}_{\mathfrak{p}}(\nu^\frac{1}{2} \mathfrak{m})= \mathfrak{m}_\rho(u_1,v_1-1) +   \mathcal{D}^\mathrm{R}_{\mathfrak{p}-\Delta_0} \left( \nu^\frac{1}{2}\mathfrak{m}_*   \right) ,    
\end{equation}
By the matching condition \eqref{eq:red_A1}, $\mathfrak{m}_\rho(u_1,v_1-1) $ is a submultisegment of $ \mathcal{D}^\mathrm{L}_\mathfrak{q}(\mathfrak{n})$. If $v_1 \ne 1$, appling Lemma \ref{lem:inclution}, there exists an index $1 \le j_0\le l$ such that $\mathfrak{m}_\rho(u_{j_0}',v_{j_0}')=\mathfrak{m}_\rho(u_1,v_1-1)$ and by Proposition \ref{prop:der_unitary_mult_left},
\begin{equation}\label{eq:red_A9}
\mathcal{D}^\mathrm{L}_\mathfrak{q}(\mathfrak{n})= \mathfrak{m}_\rho(u_{j_0}',v_{j_0}')+ \mathcal{D}^\mathrm{L}_\mathfrak{q}(\mathfrak{n}-\mathfrak{m}_\rho(u_{j_0}',v_{j_0}')).  
\end{equation}
Therefore, combining matching conditions \eqref{eq:red_A1}, \eqref{eq:red_A8}, and \eqref{eq:red_A9}, we conclude that \[\mathrm{D}^\mathrm{R}_{\mathfrak{p}-\Delta_0} \left( \nu^\frac{1}{2}\pi_*   \right) \cong \mathrm{D}^\mathrm{L}_\mathfrak{q}\left(\pi'_*\right),\]
where $\pi_*'$  is as defined in the proposition statement.
\end{proof}

\subsection{Case B}\label{sec:Case B}{\bf A factor in $\pi'$ matches $\pi_1(\frac{1}{2}-\alpha_1)$.} In this case, we assume that there exists an index $j_1 \in \{1,...,l\}$ such that \[\pi'_{j_1}(\beta_{j_1}) \cong \pi_1(\frac{1}{2}-\alpha_1) .\] 
Note that this forces $\alpha_1 \ne 0$ and $\beta_{j_1} \ne 0$.
\begin{proposition}[Reduction in Case B]\label{prop:match_2}
Suppose $\pi'_{j_1}(\beta_{j_1}) \cong \pi_1(\frac{1}{2}-\alpha_1)$ for some $j_1$.  If $(\pi, \pi')$ is RdLd-matching, then the pair 
$(\pi_*, \pi'_*)$ is also RdLd-matching, where \[\pi_* \cong \prod\limits_{i \ne 1} \pi_i(\alpha_i)  \text{ and }\pi'_* \cong 
   \prod\limits_{j \ne j_1} \pi_j'(\beta_j) .
\]
\end{proposition}

\begin{proof}
As in Proposition \ref{prop:match_1}, there exists a multisegment $\mathfrak{p}$ (resp. $\mathfrak{q}$) such that
 \begin{equation} \label{eq:red_B1}
     \mathcal{D}^\mathrm{R}_\mathfrak{p}(\nu^\frac{1}{2}\mathfrak{m}) = \mathcal{D}^\mathrm{L}_\mathfrak{q}(\mathfrak{n}) \ne \infty.
 \end{equation} 
and the maximality of $u_{1} + v_{1} +\alpha_{1} $ forces the segment $\Delta_1=\mathfrak{hd}^\mathrm{R}\left(\nu^{\frac{1}{2}+\alpha_1}\pi_1 \right) $ lies in $ \mathfrak{p}$. Applying Proposition \ref{prop:der_unitary_mult} gives 
\begin{equation}\label{eq:red_B2}
\mathcal{D}^\mathrm{R}_{\mathfrak{p}}(\nu^\frac{1}{2} \mathfrak{m})= \mathfrak{m}_\rho(u_1,v_1-1)\nu^{\alpha_1} +   \mathcal{D}^\mathrm{R}_{\mathfrak{p}-\Delta_1} \left( \mathfrak{m}_\rho(u_1,v_1)\nu^{\frac{1}{2}-\alpha_1} +\nu^\frac{1}{2}\mathfrak{m}_*   \right),    
\end{equation}
Now, because $\pi'_{j_1}(\beta_{j_1}) \cong \pi_1(\frac{1}{2}-\alpha_1)$, the matching condition and the maximality argument also imply that the segment $\Delta_1'=\mathfrak{hd}^\mathrm{L}\left(\pi'_{j_1}\nu^{-\beta_{j_1}} \right) $ lies in $ \mathfrak{q}$. Applying Proposition \ref{prop:der_unitary_mult_left} yields
\begin{equation}\label{eq:red_B3}
\mathcal{D}^\mathrm{L}_{\mathfrak{q}}(\mathfrak{n})= \mathfrak{m}_\rho(u_1,v_1-1)\nu^{\alpha_1} +   \mathcal{D}^\mathrm{L}_{\mathfrak{q}-\Delta_1'} \left( \mathfrak{n}_*+\mathfrak{m}_\rho(u_1,v_1)\nu^{\frac{1}{2}-\alpha_1}   \right),    
\end{equation}
where $\mathfrak{n}_*$ corresponds to $\pi_*'$. Combining the relations \eqref{eq:red_B1}, \eqref{eq:red_B2}, and \eqref{eq:red_B3}, we have
\begin{equation}\label{eq:red_B4}
    \mathcal{D}^\mathrm{R}_{\mathfrak{p}-\Delta_1} \left( \mathfrak{m}_\rho(u_1,v_1)\nu^{\frac{1}{2}-\alpha_1} +\nu^\frac{1}{2}\mathfrak{m}_*   \right) = \mathcal{D}^\mathrm{L}_{\mathfrak{q}-\Delta_1'} \left( \mathfrak{n}_* +\mathfrak{m}_\rho(u_1,v_1)\nu^{\frac{1}{2}-\alpha_1}   \right).
\end{equation}

Subcase 1. $v_1 \ne 1$. By Proposition \ref{prop:der_unitary_mult_left}, the top segment of the ladder $\mathfrak{m}_\rho(u_1,v_1)\nu^{\frac{1}{2}-\alpha_1}$ lies in the multisegment $\mathcal{D}^\mathrm{L}_{\mathfrak{q}-\Delta_1'} \left( \mathfrak{n}_* +\mathfrak{m}_\rho(u_1,v_1)\nu^{\frac{1}{2}-\alpha_1}   \right)$ and hence, applying Lemma \ref{lem:inclution_2} and relation \eqref{eq:red_B4} gives
\begin{equation}
    \mathcal{D}^\mathrm{R}_{\mathfrak{p}-\Delta_1} \left( \mathfrak{m}_\rho(u_1,v_1)\nu^{\frac{1}{2}-\alpha_1} +\nu^\frac{1}{2}\mathfrak{m}_*   \right)= \mathfrak{m}_\rho(u_1,v_1)\nu^{\frac{1}{2}-\alpha_1} + \mathcal{D}^\mathrm{R}_{\mathfrak{p}-\Delta_1} \left( \nu^\frac{1}{2}\mathfrak{m}_*   \right).
\end{equation}
Similarly, by Proposition \ref{prop:der_unitary_mult} and the bottom segment of the ladder $\mathfrak{m}_\rho(u_1,v_1)\nu^{\frac{1}{2}-\alpha_1}$, we have
\begin{equation}
    \mathcal{D}^\mathrm{L}_{\mathfrak{q}-\Delta_1'} \left( \mathfrak{n}_* +\mathfrak{m}_\rho(u_1,v_1)\nu^{\frac{1}{2}-\alpha_1}   \right)=\mathfrak{m}_\rho(u_1,v_1)\nu^{\frac{1}{2}-\alpha_1}+\mathcal{D}^\mathrm{L}_{\mathfrak{q}-\Delta_1'} \left( \mathfrak{n}_*   \right).
\end{equation}
Therefore, 
\[\mathcal{D}^\mathrm{R}_{\mathfrak{p}-\Delta_1} \left( \nu^\frac{1}{2}\mathfrak{m}_*   \right)=\mathcal{D}^\mathrm{L}_{\mathfrak{q}-\Delta_1'} \left( \mathfrak{n}_*   \right).\]

Subcase 2. $v_1 = 1$ and the segment $\Delta=\mathfrak{m}_\rho(u_1,v_1)\nu^{\frac{1}{2}-\alpha_1}$ lies in $\mathcal{D}^\mathrm{R}_{\mathfrak{p}-\Delta_1} \left( \mathfrak{m}_\rho(u_1,v_1)\nu^{\frac{1}{2}-\alpha_1} +\nu^\frac{1}{2}\mathfrak{m}_*   \right)$. Then, applying Lemma \ref{lem:inclution_2} and relation \eqref{eq:red_B4} as above, we again have 
\[\mathcal{D}^\mathrm{R}_{\mathfrak{p}-\Delta_1} \left( \nu^\frac{1}{2}\mathfrak{m}_*   \right)=\mathcal{D}^\mathrm{L}_{\mathfrak{q}-\Delta_1'} \left( \mathfrak{n}_*   \right).\]

Subcase 3. $v_1 = 1$ and the segment $\Delta=\mathfrak{m}_\rho(u_1,v_1)\nu^{\frac{1}{2}-\alpha_1}$ does not lie in $\mathcal{D}^\mathrm{R}_{\mathfrak{p}-\Delta_1} \left( \mathfrak{m}_\rho(u_1,v_1)\nu^{\frac{1}{2}-\alpha_1} +\nu^\frac{1}{2}\mathfrak{m}_*   \right)$. Then, $\Delta \in \mathfrak{p}-\Delta_1$ and by relation \eqref{eq:red_B4}, $\Delta \in \mathfrak{q}-\Delta_1'$. Therefore, Proposition \ref{prop:der_unitary_mult} and \ref{prop:der_unitary_mult_left} yield 
\[\mathcal{D}^\mathrm{R}_{\mathfrak{p}-\Delta_1-\Delta} \left( \nu^\frac{1}{2}\mathfrak{m}_*   \right)=\mathcal{D}^\mathrm{L}_{\mathfrak{q}-\Delta_1'-\Delta} \left( \mathfrak{n}_*   \right).\]
\end{proof}

\section{Unitary Gan–Gross–Prasad relevance}
In this section, we prove Theorem \ref{thm:unitary}, which establishes the equivalence between the generalized GGP relevance criterion (Definition \ref{def:relevant}) and the extended Gan–Gross–Prasad relevance criterion (Definition \ref{def:extended_relevance}) for irreducible unitary representations. This equivalence is the central theoretical result of the paper, as it bridges Chan's general characterization of quotient branching \cite{Cha_qbl} with a more explicit and easily verifiable condition in terms of the Tadić parameters of unitary representations.

We recall the setting. Let $\pi$ and $\pi'$ be irreducible unitary representations of general linear groups over a non-archimedean local field $F$. By Tadić's classification, $\pi$ and $\pi'$ are of the form
\begin{equation}\tag{$\star \star$}\label{eq:forw_1}
 \pi=L(\mathfrak{m}) \cong \pi_1(\alpha_1) \times ... \times \pi_r(\alpha_r) \text{ and }   \pi'=L(\mathfrak{n}) \cong \pi_1'(\beta_1) \times ... \times \pi_l'(\beta_l),
\end{equation}
where for each $i$ and $j$,
\[\pi_i=\pi_{\rho_i}(u_i,v_i) \text{ with } u_i,v_i \in \mathbb{Z}_{>0};~\pi_j'=\pi_{\rho_j}(u_j',v_j') \text{ with } u_j',v_j' \in \mathbb{Z}_{>0},\] and $0 \le \alpha_i,\beta_j <\frac{1}{2}$. Moreover, since derivatives and integrals factor over distinct cuspidal lines, it suffices to consider the case where all factors lie in the same cuspidal line; the general case follows by taking disjoint unions. Thus, for the remainder of this section, we assume that $\rho_i \cong \rho_j'\cong \rho$ for all $i,j$, where $\rho$ is a fixed unitary cuspidal representation.

Here, a factor of the form  $\pi_\rho(u,v)(\alpha)$ is called a generalized Steinberg representation if $v=1$ and $\alpha=0$.

\begin{theorem}
 Let $\pi$ and $\pi'$ be irreducible unitary representations of general linear groups such that $(\pi, \pi')$ is a generalized GGP relevant pair (Definition \ref{def:relevant}). Then, $\pi$ and $\pi'$ are Gan–Gross–Prasad relevance (Definition \ref{def:extended_relevance}).  
\end{theorem}

\begin{proof}
Let $\pi$ and $\pi'$ be of the form \eqref{eq:forw_1}. Let $\mathcal{N}(\pi,\pi')$ be the total number of factors $\pi_i(\alpha_i)$ and $\pi_j'(\beta_j)$ in \eqref{eq:forw_1} which are not generalized Steinberg representation. By Definition \ref{def:relevant}, $(\pi, \pi')$ is RdLd-matching. Using induction method on the number $\mathcal{N}(\pi,\pi')$, we show that RdLd-matching implies Gan–Gross–Prasad relevance.

Base case: $\mathcal{N}(\pi,\pi')=0$. In this case, all the factors $\pi_i(\alpha_i)$ and $\pi_j'(\beta_j)$ are generalized Steinberg representations meaning $v_i=v_j'=1$ and $\alpha_i=\beta_j=0$ for all $i,j$. Hence $\pi$ and $\pi'$ are generic representations. By Proposition \ref{prop:generic}, they are generalized GGP relevant, and they satisfy Definition \ref{def:extended_relevance} with $I_1=I_2=I_3=\emptyset$, $I_4=\{1,...,r\}$ and $J_4=\{1,...,l\}$. Thus the base case holds.

Inductive step: Assume $\mathcal{N}(\pi,\pi')>0$. Without loss of generality, we may order the factors of $\pi$ and $\pi'$ as in Section \ref{sec:RdLi-match}, so that the largest factor in $\pi$ is $\pi_1(\alpha_1)$ (in the sense of the size parameter $u_{1} + v_{1} +\alpha_{1}$ is maximal). If $\pi_1(\alpha_1)$ is a generalized Steinberg representation (i.e. $v_1 = 1$ and $\alpha_1=0$), we apply Proposition \ref{prop:match_1} and if needed using Lemma \ref{lem:duality}, we apply Proposition \ref{prop:match_1} repeatedly to reach a pair of unitary representations $(\tilde \pi, \tilde \pi')$ which is RdLd-matching and the largest factor is not a generalized Steinberg representation. The generalized Steinberg representations, which are peeled off in this removal process, will be adjusted either in $I_4$ or in $J_4$ of the Gan–Gross–Prasad relevance criterion. Therefore, we may also assume that the RdLd-matching representations $\pi$ and $\pi'$ in \eqref{eq:forw_1} having a non-generalized Steinberg largest factor $\pi_1(\alpha_1)$.  

Now we consider two cases, corresponding to the two reduction propositions established in Section \ref{sec:RdLi}.

Case 1: No factor in $\pi'$ is isomorphic to $\pi_1(\frac{1}{2}-\alpha_1)$. In this case, by Proposition 6.2, there exists an index $j_0$ (or possibly none when $v_1=1$) such that the reduced pair $(\pi_*, \pi'_*)$ is RdLd-matching, where
\[\pi_* \cong \prod\limits_{i \neq 1} \pi_i(\alpha_i)  \text{ and }\pi'_* \cong \begin{cases}
    \prod\limits_{j \neq j_0} \pi_j'(\beta_j) &\mbox{ if } v_1 \ne 1\\
    \pi'    &\mbox{ if } v_1 = 1.
\end{cases}\]
By construction, $\mathcal{N}(\pi_*,\pi_*')<\mathcal{N}(\pi,\pi')$. Applying the induction hypothesis, $\pi_*$ and $\pi_*'$ satisfy Definition \ref{def:extended_relevance}. By examining the removed factors, we see that $\pi$ and $\pi'$ also satisfy Definition \ref{def:extended_relevance}: the removed factor $\pi_1(\alpha_1)$ is accounted for in either $I_4$ (if $v_1=1$) or in $I_2$ (if $v_1\ne 1$).

Case 2: There exists $j_1 \in \{1,...,l\}$ such that $\pi'_{j_1}(\beta_{j_1}) \cong \pi_1(\frac{1}{2}-\alpha_1)$. In this case, by Proposition \ref{prop:match_2}, the reduced pair $(\pi_*, \pi'_*)$ is RdLd-matching, where 
\begin{equation}\label{eq:con_1}
  \pi_* \cong \prod\limits_{i \neq 1} \pi_i(\alpha_i)  \text{ and }\pi'_* \cong 
   \prod\limits_{j \neq j_1} \pi_j'(\beta_j) .  
\end{equation}
As before, $\mathcal{N}(\pi_*,\pi_*')<\mathcal{N}(\pi,\pi')$ and by the induction hypothesis, $\pi_*$ and $\pi_*'$ satisfy Definition \ref{def:extended_relevance}.  Adding back the matched pair $\left(\pi_1(\alpha_1),\pi'_{j_1}(\beta_{j_1})\right)$ corresponds precisely to the relation (R3) in Definition \ref{def:extended_relevance} with $\beta_{j_1}=\frac{1}{2}-\alpha_1$.

Thus, in all cases, the induction proceeds, and we conclude that $\pi$ and $\pi'$ are Gan–Gross–Prasad relevant as in Definition \ref{def:extended_relevance}. 
\end{proof}

\begin{theorem}
 Let $\pi$ and $\pi'$ be the irreducible unitary representations of general linear groups such that $\pi$ and $\pi'$ are Gan–Gross–Prasad relevant. Then, $(\pi, \pi')$ is a generalized GGP relevant pair. 
\end{theorem}
\begin{proof}
We assume that $\pi$ and $\pi'$ are the irreducible unitary representations of the form \eqref{eq:forw_1} satisfying the relations (R1)-(R4) of Definition \ref{def:extended_relevance}. Our goal is to construct multisegments $\mathfrak{p}$ and $\mathfrak{q}$ satisfying the two conditions of Definition \ref{def:relevant}.

If there exist $i_* \in I_4$ and $j_* \in J_4$ such that $\pi'_{j_*}(\beta_{j_*}) \cong \pi_{i_*}(\frac{1}{2}-\alpha_{i_*})$, then we extend  $I_3$ to $I_3'=I_3 \cup \{i_*\}$, $J_3$ to $J_3'=J_3 \cup \{j_*\}$ and having natural extended bijection $\lambda'_3: I_3' \rightarrow J_3'$. Also, we reduce $I_4$ to $I'_4=I_4 - \{i_*\}$ and reduce $J_4$ to $J_4'=J_4 - \{j_*\}$. By repeatedly applying the (R3) relation, we may enlarge $I_3$ to a maximal subset $I_3^\mathrm{max}$ (and correspondingly $J_3^\mathrm{max}$) and have a unique decomposition $I_3 \sqcup I_4=I_3^\mathrm{max} \sqcup I_4^\mathrm{min}$ and $J_3 \sqcup J_4=J_3^\mathrm{max} \sqcup J_4^\mathrm{min}$ such that there exists a bijection $\lambda_3^\mathrm{max}:I_3^\mathrm{max} \rightarrow J_3^\mathrm{max}$ with \[\pi'_{\lambda_3^\mathrm{max}(i)}\left(\beta_{\lambda_3^\mathrm{max}(i)}\right) \cong \pi_{i}\left(\frac{1}{2}-\alpha_{i}\right) \text { for }i \in I_3^\mathrm{max}.\] 
After this maximal matching, there are no remaining pairs $(i,j)$ with $i \in I_4^\mathrm{min}$ and $j \in J_4^\mathrm{min}$ such that $\pi'_{j}(\beta_{j}) \cong \pi_{i}(\frac{1}{2}-\alpha_{i})$.

 For each factor in $\pi$ and $\pi'$, we define contributions to $\mathfrak{p}$ and $\mathfrak{q}$ as follows:
 \begin{align*}
     \mathfrak{p}_i &= \mathfrak{hd}^\mathrm{R}(\pi_i)\nu^{\frac{1}{2}-\alpha_i} + \mathfrak{hd}^\mathrm{R}(\pi_i)\nu^{\frac{1}{2}+\alpha_i} &&\text{ and } \mathfrak{q}_{\lambda_1(i)}=\emptyset && \text{ for } i \in I_1\\
   \mathfrak{p}_i&=\emptyset   &&\text{ and } \mathfrak{q}_{\lambda_2(i)} = \mathfrak{hd}^\mathrm{L}(\pi'_{\lambda_2(i)})\nu^{-\beta_{\lambda_2(i)}} + \mathfrak{hd}^\mathrm{L}(\pi_{\lambda_2(i)}')\nu^{\beta_{\lambda_2(i)}} && \text{ for } i \in I_2\\
   \mathfrak{p}_i&=\mathfrak{hd}^\mathrm{R}(\pi_i)\nu^{\frac{1}{2}+\alpha_i}   &&\text{ and } \mathfrak{q}_{\lambda_3^\mathrm{max}(i)} = \mathfrak{hd}^\mathrm{L}(\pi'_{\lambda_3^\mathrm{max}(i)})\nu^{-\beta_{\lambda_3^\mathrm{max}(i)}}  && \text{ for } i \in I_3^\mathrm{max}
 \end{align*}
 \[\mathfrak{p}_i=\pi_i(\alpha_i) \text{ for } i \in I_4^\mathrm{min}   \text{ and } \mathfrak{q}_j=\pi_j'(\beta_j) \text{ for } j \in J_4^\mathrm{min}.\]
Finally, we define \[\mathfrak{p}=\sum\limits_{i=1}^r \mathfrak{p}_i \text{ and }\mathfrak{q}=\sum\limits_{j=1}^l \mathfrak{q}_j.\] Then, using Proposition \ref{prop:der_unitary_mult} and \ref{prop:der_unitary_mult_left}, we conclude that
\begin{equation}\label{eq:RdLi-matching}
0 \ne \mathrm{D}^\mathrm{R}_\mathfrak{p}\left(\nu^\frac{1}{2} \pi\right)    \cong \mathrm{D}^\mathrm{R}_\mathfrak{p}(\pi').    
\end{equation}

We now verify the strong RdLi-commutativity of $(\mathfrak{p}, \mathfrak{q}, \nu^\frac{1}{2}\pi)$. By Definition \ref{def:RdLi}, it suffices to check that for any $\Delta \in \mathfrak{p}$ and $\Delta' \in \mathfrak{q}$, the tuple $(\Delta,\Delta', \sigma)$ is combinatorially RdLi-commutative for the appropriate intermediate representations $\sigma$. We use the sufficient condition from Lemma \ref{lem:example}. Let $\Delta \in \mathfrak{p}$ and $\Delta' \in \mathfrak{q}$. Write \[\Delta=\left[-\frac{u-1}{2}+ \frac{v-1}{2}, \frac{u-1}{2}+ \frac{v-1}{2} \right]_\rho \nu^{\frac{1}{2}\pm \alpha} \text{ and } \Delta'=\left[-\frac{u'-1}{2}- \frac{v'-1}{2}, \frac{u'-1}{2}- \frac{v'-1}{2} \right]_\rho \nu^{\pm \beta}\] for some positive integers $u,v,u',v'$ and $0\le\alpha,\beta <\frac{1}{2}$. We put $A=\frac{u-1}{2}, A'=\frac{u'-1}{2}, x= \frac{v-1}{2},$ $ x'=\frac{v'-1}{2}$ and $-\frac{1}{2}<\gamma <\frac{1}{2}$. To analyze the relative positions of $\Delta$ and $\Delta'$, we consider the following cases:

Case 1. Either $v \neq 1$ or $v'\ne 1$: Then, either $x' \ge \frac{1}{2}$ or $x'=0$ and $x \ge \frac{1}{2}$.  Fix $x \ge \frac{1}{2}$ and $x'=0$. Then, 
\begin{align}\label{eq:s}
s(\Delta) \le s(\Delta') &\Longleftrightarrow
    -A + x + \left(\frac{1}{2} + \gamma\right) \le -A' \pm \beta \nonumber\\
    &\Longleftrightarrow -A - x - \left(\frac{1}{2} + \gamma\right)+2\left(x + \frac{1}{2} + \gamma \mp \beta\right) \le -A' +(\mp \beta) \nonumber\\
    &\Longrightarrow -A - x -\left(\frac{1}{2} + \gamma\right) < -A' +(\mp \beta), \text{ as }x + \frac{1}{2} + \gamma \mp \beta >0 \nonumber\\
    &\Longrightarrow A + x +\left(\frac{1}{2} + \gamma\right) > A' \pm \beta   \Longleftrightarrow e(\Delta) > e(\Delta').
\end{align}
Now we assume $x' \ge \frac{1}{2}$. Then, by a similar argument as for \eqref{eq:s}, we have
\begin{align*}
s(\Delta) \le s(\Delta') &\Longleftrightarrow    -A + x + \left(\frac{1}{2} + \gamma\right) \le -A' - x' \pm \beta  
    &\Longrightarrow A + x +\left(\frac{1}{2} + \gamma\right) > A' -(x'\mp \beta).
\end{align*}
Therefore, in both situations, if $s(\Delta) \le s(\Delta')$, we have $e(\Delta) > e(\Delta')$.

Case 2. $v =v'= 1$: Suppose the cuspidal supports of $\Delta$ and $\Delta'$ are in same cuspidal line and $s(\Delta) \le s(\Delta')$. If $\Delta=\left[-\frac{u-1}{2}, \frac{u-1}{2}\right]_\rho \nu^{\frac{1}{2}+\alpha}$, by a similar argument as for \eqref{eq:s}, we have
\begin{align*}
    -A +  \left(\frac{1}{2} + \alpha\right) \le -A' \pm \beta  
    &\Longrightarrow A + \left(\frac{1}{2} + \alpha \right) > A' \pm \beta.
\end{align*}
Therefore, if $\Delta=\left[-\frac{u-1}{2}, \frac{u-1}{2}\right]_\rho \nu^{\frac{1}{2}+\alpha}$ and $s(\Delta) \le s(\Delta')$, we have $e(\Delta) > e(\Delta')$. We now assume that $\Delta=\left[-\frac{u-1}{2}, \frac{u-1}{2}\right]_\rho \nu^{\frac{1}{2}-\alpha}$.  If $\Delta'=\left[-\frac{u'-1}{2}, \frac{u'-1}{2}\right]_\rho \nu^{-\beta}$, by a similar argument as for \eqref{eq:s}, we have
\begin{align*}
 s(\Delta) \le s(\Delta') \Longleftrightarrow   -A +  \left(\frac{1}{2} - \alpha\right) \le -A' - \beta  
    &\Longrightarrow A + \left(\frac{1}{2} - \alpha \right) > A' - \beta \Longleftrightarrow e(\Delta) > e(\Delta').
\end{align*}
Finally, we fix $\Delta'=\left[-\frac{u'-1}{2}, \frac{u'-1}{2}\right]_\rho \nu^{\beta}$. As the cuspidal supports of $\Delta$ and $\Delta'$ are in the same cuspidal line, for some integer $z \in \mathbb{Z}_{\ge 0}$, we have
\[-A + z+ \left(\frac{1}{2} - \alpha \right) = -A' + \beta.\] Therefore, $-A+z=-A'$ and $(\frac{1}{2} - \alpha)=\beta$. If $z \ne 0$, then $A + (\frac{1}{2} - \alpha) > A' + \beta$ and we have $e(\Delta) > e(\Delta')$. If $z=0$, we have $A+ (\frac{1}{2} - \alpha) = A' + \beta$. This happen, only when there exists some $i \in I_4^\mathrm{min}$ and $j \in J_4^\mathrm{min}$ such that $\pi_j'(\beta) \cong \pi_i(\frac{1}{2}-\alpha)$, which contradicts the maximal property of $I_3^\mathrm{max}$ and $J_3^\mathrm{max}$. Hence, for any $\Delta \in \mathfrak{p}$ and $\Delta' \in \mathfrak{q}$, one of the following holds:
\begin{itemize}
    \item The cuspidal supports of $\Delta$ and $\Delta'$ are in different cuspidal lines;
    \item $s(\Delta) > s(\Delta')$;
    \item  $e(\Delta) > e(\Delta')$.
\end{itemize}
By Lemma \ref{lem:example}, any such pair $(\Delta, \Delta')$ satisfies the strongly RdLi-commutative condition for any $\sigma$ with $\mathrm{D}^\mathrm{R}_\Delta(\sigma) \ne 0$. Applying this to the intermediate representations that appear in Definition \ref{def:RdLi} (which are obtained by successively applying derivatives and integrals), we conclude that $(\mathfrak{p}, \mathfrak{q}, \nu^\frac{1}{2}\pi)$ is a strongly RdLi-commutative triple. Thus, together with \eqref{eq:RdLi-matching}, $(\pi, \pi')$ becomes a generalized GGP relevant pair. 
\end{proof}

\appendix
\section{(written by Maxim Gurevich)}\label{app:gur}

The purpose of this appendix is to clarify and resolve an inconsistency between the statements of Theorem~\ref{main_thm:unitary} of the present paper and \cite[Theorem~5.7]{Gur}.

We emphasize again that the modified notion of Gan–Gross–Prasad relevance adopted in the present work does not coincide with the condition imposed in \cite[Theorem~5.7]{Gur}. The latter formulation did not allow for the additional case (R3) listed in Definition~\ref{def:extended_relevance}.

As a consequence, \cite[Theorem~5.7]{Gur}, as originally stated, is not correct. This failure stems from a specific gap in its proof, which went unnoticed at the time.

While the incorrect statement has not, to the best of our knowledge, been used elsewhere in the literature, it is still worthwhile to address the issue explicitly.  We also stress that \cite[Theorem~5.7]{Gur} was logically independent of the other results and arguments in \cite{Gur}, so that their validity remains unaffected.

\subsection{Clarification of the missing argument in \cite[Theorem~5.7]{Gur}}

We now describe the gap in the proof of \cite[Theorem~5.7]{Gur} in more detail.

The proof relied on the following paragraph, which served as the key step in extending the Arthur-type result \cite[Theorem~5.6]{Gur} to general unitary representations:

\begin{quote}
\textit{
``It is easy to see that given Arthur-type representations $\tau,\tau',\tau''$ and numbers $\gamma',\gamma''$ with $0< \gamma' < \gamma''<\frac12$, the representations $\tau,\tau'(\gamma'),\tau''(\gamma'')$ have pairwise disjoint supercuspidal supports. The same clearly holds for their derivatives."
}
\end{quote}

The statement quoted above is correct as it stands. However, the subsequent arguments implicitly use a different condition, namely that the supercuspidal support of $\tau'(\gamma')$ be disjoint from that of the \emph{shifted} representation $\tau''(\gamma'')\nu^{\frac12}$. The last requirement does not follow from the quoted paragraph and is not equivalent to it.

Indeed, when $\alpha+\beta=\frac12$ with $0<\alpha<\beta<\frac{1}{2}$ and $\tau=\tau'=\tau''$, it is possible for the supercuspidal support of certain (left) derivatives of
\[
\tau(\alpha)=\tau\nu^{\alpha}\times\tau\nu^{-\alpha}
\]
to overlap with that of (right) derivatives of
\[
\tau(\beta)\nu^{\frac12}=(\tau\nu^{-\alpha})\nu^{1}\times\tau\nu^{\alpha}.
\]

For instance, taking $\tau=\mathbb{1}_2$, one obtains the identity
\[
{}^{-}(\mathbb{1}_2\nu^{-\alpha}) \cong \nu^{-\alpha+1/2} \cong (\mathbb{1}_2\nu^{-\alpha+1})^{-}
\]
of $\mathrm{GL}_1(F)$-representations. This is the identity that allows for the phenomenon of Example 1 to be reconcilable with the arguments in \cite{Gur}.

\subsection{Acknowledgments}

A minimal-rank counterexample to \cite[Theorem~5.7]{Gur}, essentially corresponding to Example~\ref{ex}, was explained to me in detail by Rui Chen in June~2025. I am grateful to him for first drawing my attention to this issue. I also thank the author of the present paper for resolving the problem by establishing Theorem~\ref{main_thm:unitary}, which should now be regarded as the correct and complete formulation replacing \cite[Theorem~5.7]{Gur}. Finally, I apologize to readers for any confusion the original oversight may have caused.
%

\end{document}